\documentclass[a4paper,reqno]{amsart}

\usepackage{amssymb}
\usepackage{amscd}
\usepackage{amsthm}
\usepackage{amsmath}
\usepackage{enumerate}
\usepackage{mathrsfs}
\usepackage{graphics}
\usepackage[all]{xy} 
\usepackage{tikz}
\usepackage{tikz-cd}
\usetikzlibrary{calc}
\usepackage[retainorgcmds]{IEEEtrantools}
\usetikzlibrary{matrix,arrows,decorations.pathmorphing}

\newtheorem{Thm}{Theorem}[section]\newtheorem*{Thm*}{Theorem}
\newtheorem{Lem}[Thm]{Lemma} 
\newtheorem{Cor}[Thm]{Corollary}

\newtheorem{Prop}[Thm]{Proposition}
\newtheorem{Prop-Def}[Thm]{Proposition-Definition}

\theoremstyle{definition}
\newtheorem{Ex}[Thm]{Example}
\newtheorem{Def}[Thm]{Definition}
\newtheorem{Assu}[Thm]{Assumption}
\newtheorem{Rem}[Thm]{Remark}

\newcommand{\ra}{\rightarrow}
\newcommand{\la}{\leftarrow}
\newcommand{\D}{\mathsf{D}}
\newcommand{\K}{\mathsf{K}}
\newcommand{\SSS}{\mathbb{S}}
\newcommand{\EEE}{\mathbb{E}}
\newcommand{\s}{\mathfrak{s}}
\newcommand{\CCC}{\mathscr{C}}
\newcommand{\cal}{\mathcal}
\newcommand{\T}{\mathcal T}

\newcommand{\bb}{\mathrm{b}}

\newcommand{\Filt}{\mathsf{Filt}}

\newcommand{\h}{{\mathrm H}}
\renewcommand{\dim}{{\rm dim}}
\newcommand{\B}{{\bf B}}
\newcommand{\R}{\mathbf{R}}
\newcommand{\xra}{\xrightarrow}
\newcommand{\Z}{{\mathbb Z}}
\newcommand{\N}{{\mathbb N}}
\newcommand{\lra}{\longrightarrow}
\newcommand{\hra}{\hookrightarrow}

\newcommand{\un}{\underline}
\newcommand{\op}{\oplus}
\newcommand{\bop}{\bigoplus}
\newcommand{\ot}{\otimes}

\newcommand{\sg}{\operatorname{sg}\nolimits}
\newcommand{\Ext}{\operatorname{Ext}\nolimits}
\newcommand{\Hom}{\operatorname{Hom}\nolimits}
\newcommand{\Top}{\operatorname{top}\nolimits}
\newcommand{\rad}{\operatorname{rad}\nolimits}
\newcommand{\soc}{\operatorname{soc}\nolimits}

\newcommand{\ke}{\operatorname{Ker}\nolimits}

\newcommand{\End}{\operatorname{End}\nolimits}
\newcommand{\RHom}{\mathbf{R}\strut\kern-.2em\operatorname{Hom}\nolimits}
\newcommand{\RshHom}{\mathbf{R}\strut\kern-.2em\mathscr{H}\strut\kern-.3em\operatorname{om}\nolimits}
\newcommand{\shHom}{\mathscr{H}\strut\kern-.3em\operatorname{om}\nolimits}
\newcommand{\shEnd}{\mathscr{E}\strut\kern-.3em\operatorname{nd}\nolimits}

\DeclareMathOperator{\moduleCategory}{{\mathsf{mod}}} \renewcommand{\mod}{\moduleCategory}
\DeclareMathOperator{\proj}{\mathsf {proj}}
\DeclareMathOperator{\thick}{\mathsf{thick}}
\DeclareMathOperator{\per}{\mathsf{per}}
\DeclareMathOperator{\add}{\mathsf {add}}
\DeclareMathOperator{\CM}{\mathsf {CM}}

\DeclareMathOperator{\Inj}{\mathsf{Inj}}
\DeclareMathOperator{\Proj}{\mathsf{Proj}}
\DeclareMathOperator{\ind}{\mathsf {ind}}
\numberwithin{equation}{section}

\title[CM dg modules and negative CY configurations]{Cohen-Macaulay differential graded modules \\ and Negative Calabi-Yau configurations}

\begin{document}

\author{Haibo Jin}
\address{Haibo Jin: Graduate School of Mathematics, Nagoya University, Furocho, Chikusaku, Nagoya 464-8602, Japan}
\email{d16002n@math.nagoya-u.ac.jp}

\begin{abstract}
In this paper, we introduce  the class of Cohen-Macaulay (=CM)   dg (=differential graded)  modules over  Gorenstein dg algebras  and study their basic properties. We show that the category of CM dg modules forms a Frobenius extriangulated category, in the sense of Nakaoka and Palu, and it admits almost split extensions. 
We also study representation-finite $d$-self-injective dg algebras $A$ in detail for some positive integer $d$. In particular,  we classify the Auslander-Reiten (=AR) quivers of $\CM A$ for a large class of $d$-self-injective dg algebras $A$ in terms of $(-d)$-Calabi-Yau (=CY) configurations, which are Riedtmann's configurations for the case $d=1$.
  For any given $(-d)$-CY configuration $C$, 
  we show there exists a $d$-self-injective dg algebra $A$,  such that the AR quiver of $\rm{CM} A$ is given by $C$.
For type $A_{n}$, by
  using a bijection between  $(-d)$-CY configurations and certain purely combinatorial objects which we call maximal $d$-Brauer relations given by Coelho Sim\~oes, we construct such $A$ through a Brauer tree dg algebra.  
\end{abstract}

\thanks{The  author is  supported by China Scholarship Council (No.201606140033).}

\maketitle
\tableofcontents
\setcounter{section}{-1}

\section{Introduction}

The notion of Cohen-Macaulay (CM) modules is classical   in commutative algebra \cite{M,BH}, and has natural generalizations for non-commutative algebras \cite{Buc, H2, IW}, often  called Gorenstein projective modules \cite{AB1, C, EJ}. 
The category of CM modules
has been studied by many researchers in representation theory (see, for example, \cite{CR,Yoshino, Simson, LW}).  On the other hand,  the derived categories of   differential graded (dg) categories  introduced by Bondal-Kapranov \cite{BK} and Keller \cite{Keller94, Keller06} is an active subject appearing in various areas of mathematics \cite{Mi, T, Ye}. Among others, we refer to \cite{Andrei, Jo, J,KY2, Schmidt} for the representation  theory of dg categories.

In this paper,  
we introduce Cohen-Macaulay dg modules over dg algebras and develop their representation theory  to
 build a connection between these two subjects. 
 One of the main properties of the category of Cohen-Macaulay dg modules is that it has 
   a structure of extriangulated category and the stable category is equivalent to the singularity category, which is  an analogue of Buchweitz's equivalence. Moreover, it admits almost split extensions and we can study it by Auslander-Reiten theory. In fact, there are many nice dg algebras (including those given in this paper), whose categories of Cohen-Macaulay dg modules can be well understood, while the derived categories of dg algebras are usually wild and it is hopeless to  classify all  the indecomposable objects. 
  

To make everything works well, we need to add some restrictions on dg algebras. More precisely, we work on dg algebras $A$ over a field $k$ satisfying the following assumptions.
\begin{Assu}\label{assumption}
 \begin{enumerate}[\rm(1)]
 \item $A$ is \emph{non-positive}, $i.e.$ $\h^{i}(A)=0$ for $i>0$ (without loss of generality, we may assume $A^{i}=0$ for $i>0$, see Section \ref{Section:nonpositivedg});
 \item $A$ is \emph{proper}, $i.e.$ $\dim_{k}\bigoplus_{i\in\Z}\h^{i}(A)<\infty$;
 \item $A$ is \emph{Gorenstein}, $i.e.$ the thick subcategory $\per A$ of the derived category $\D (A)$  generated by $A$ coincides with the thick subcategory generated by $DA$, where $D=\Hom_{k}(?, k)$ is the $k$-dual.  
 \end{enumerate}
 \end{Assu}
In this case, we define Cohen-Macaulay dg $A$-modules as follows, where we denote by $\D^{\bb}(A)$ the full subcategory of $\D(A)$ consisting of the dg $A$-modules whose total cohomology is finite-dimensional. 

\begin{Def}[Definition \ref{Def:CM}]\label{Def:CM0}
\begin{enumerate}[\rm(1)]
\item
A dg $A$-module $M$ in $\D^{\bb}(A)$ is called a \emph{Cohen-Macaulay dg $A$-module} if $\h^{i}(M)=0$ and  $\Hom_{\D A}(M, A[i])=0$ for $i>0$;
\item
We denote by  $\CM A$  the subcategory of $\D^{\bb}(A)$ consisting of Cohen-Macaulay dg $A$-modules.
\end{enumerate}
\end{Def}

 Definition   \ref{Def:CM0}  is   motived by the fact that if 
$A$ is concentrated in degree zero, then the condition (1) above gives an alternative description of classical Cohen-Macaulay modules \cite[Theorem 3.10]{IY2}. Moreover, in this case  the category $\CM A$  forms a Frobenius category in the sense of \cite{Happel} and the stable category $\un{\CM}A$ is a triangulated category which is triangle equivalent to the singularity category $\D_{\sg}(A)=\D^{\rm b}({\mod} A)/\K^{\rm b}(\proj A)$ introduced by Buchweitz \cite{Buc} and Orlov \cite{O1}. However, $\CM A$ does not necessarily have a natural structure of exact category in our setting. Instead, the following result shows it has a natural structure of extriangulated category introduced by  Nakaoka and Palu \cite{Palu}. 
\begin{Thm} [{Theorems \ref{Thm:properties}}, \ref{Thm:AR} and \ref{Thm:cmappro}]
Let $A$ be a non-positive proper Gorenstein dg algebra. Then 
\begin{enumerate}[\rm(1)]
\item $\CM A$ is functorially finite in $\D^{\bb}(A)$;
\item $\CM A$ is a Frobenius extriangulated category with $\Proj (\CM A)=\add A$;
 \item The stable category $\underline{\CM} A:=(\CM A)/[\add A] $ is a triangulated category;
 \item The composition $\CM A \hookrightarrow \D^{\bb}(A)\ra \D^{\bb}(A)/\per A$ induces a triangle equivalence
\[ \underline{\CM} A=(\CM A)/[\add A] \simeq  \D^{\bb}(A)/\per A=\D_{\sg}(A); \]
\item $\un{\CM}A$ admits a Serre functor and  $\CM A$ admits almost split extensions.
\end{enumerate}
 \end{Thm}

 The main examples considered in this paper are trivial extension dg algebras  and truncated polynomial dg algebras. We determine all indecomposable Cohen-Macaulay dg modules over truncated polynomial dg algebras concretely and give their AR quivers (see Theorem \ref{Thm:ARquiver} for the details). We also show that, in this case, the stable categories are cluster categories by using a criterion given by Keller and Reiten \cite{Keller08} (see Theorem \ref{main}).

One of the traditional subjects is the classification of Gorenstein rings which are representation-finite in the sense that they have only finitely many indecomposable Cohen-Macaulay modules. 
Riedtmann \cite{Riedtmann, Riedtmann2} and Wiedemann \cite{Wiedemann} considered the classification of representation-finite self-injective algebras and Gorenstein orders respectively. In both classifications, configurations play an  important role. We may regard Wiedemann's configurations as ``0-Calabi-Yau'' since they are preserved by Serre functor $\SSS$ and regard Riedtmann's configurations as ``$(-1)$-Calabi-Yau'' since they are preserved by $\SSS\circ [1]$. Inspired by this, we introduce the negative Calabi-Yau configurations to study the   AR quivers of $\CM A$.
 
 \begin{Def}[Definitions \ref{Def:configuration} and \ref{Def:SMS}]
Let $\T$ be a $k$-linear Hom-finite Krull-Schmidt  triangulated category and let $C$ be a set of indecomposable objects of $\T$.  We call $C$ a \emph{$(-d)$-Calabi-Yau configuration} (or \emph{$(-d)$-CY configuration} for short) for $d\ge 1$ if the following conditions hold.
 \begin{enumerate} [\rm (1)]
 \item $\dim_{k} \Hom_{\T}(X,Y)=\delta_{X,Y}$ for $X, Y\in C$;
\item $\Hom_{\T}(X, Y[-j])=0$ for any two objects $X, Y$ in $ C$ and $0<j \le d-1$;

\item For any indecomposable object $M$ in $
\T$, there exists $X\in C$ and $0\le j \le d-1$, such that $\Hom_{\T}(X, M[-j])\not=0$.
\end{enumerate}
We call $C$ a \emph{$d$-simple-minded system} (or \emph{$d$-SMS}), if it satisfies (1), (2) and 

\noindent
(3') $\T=\add \Filt\{C, C[1], \cdots, C[d-1]\}$.
\end{Def}
 It is precisely Riedtmann's configuration if $d=1$ and $\T$ is the mesh category of $\Z \Delta$ for a Dynkin diagram $\Delta$ (see \cite[Definition 2.3]{Riedtmann} for the details). 
 It is easy to see $d$-SMS implies $(-d)$-CY configuration and the converse is also  true if $\Filt(C)$ is functorially finite in $\T$ due to \cite[Proposition 2.13]{CSP}.
 ``$(-d)$-Calabi-Yau configuration'' are also introduced as ``left $d$-Riedtmann configuration'' in \cite{CS}, and further studied in \cite{CS2, CSP}. 
 When the AR quiver of $\T$ is $\Z\Delta/G$ for some Dynkin diagram $\Delta$ and some group $G$, Calabi-Yau configuration can be characterized combinatorially (see Section \ref{Section:combi}). 
Our name ``$(-d)$-Calabi-Yau configuration''   is motivated by the following theorem, which is new 
even for  $d=1$ (see Remark \ref{Rem:closed}).
\begin{Thm}[Theorem \ref{Thm:closed}]\label{Thm:preserved}
Let $\T$ be a $k$-linear Hom-finite  Krull-Schmidt  triangulated category with a Serre functor $\SSS$. 
Let $C$ be a $(-d)$-CY configuration in $\T$, then $\SSS[d]C= C$.
\end{Thm}

We say a dg $k$-algebra $A$ in Assumption~\ref{assumption} is \emph{$d$-self-injective} (resp. \emph{$d$-symmetric}) if 
$\add A=\add DA[d-1]$
 in $\D(A)$ (resp. $\D(A^{\rm e})$). 
The following result, characterizing   simple dg $A$-modules as a $(-d)$-CY configuration,  generalizes \cite[Proposition 2.4]{Riedtmann}.
\begin{Thm}[Theorem \ref{Thm:configuration}]\label{Thm:con}
Let $A$ be a $d$-self-injective dg algebra. Then the set of simple dg $A$-modules is a  $d$-SMS, and hence a $(-d)$-CY configuration in $\un{\CM}A$. 
\end{Thm}

 Let $\Delta$ be a Dynkin digram. 
For a subset $C$ of vertices of $\Z\Delta$, we define a translation quiver $(\Z\Delta)_{C}$ by 
 adding to $\Z\Delta$ a vertex $p_{c}$ and two arrows $c\ra p_{c} \ra \tau^{-1}(c)$ for each $c\in C$ (see Definition \ref{Def:newstable}). 
 Our main result in this paper states that the converse of Theorem \ref{Thm:con} also holds in the following sense.

\begin{Thm}[Theorem \ref{Thm:main}]\label{Thm:intromain}
Let $\Delta$ be a Dynkin digram. 
Let $C$ be a subset of vertices of $\Z \Delta/\SSS[d]$. The following are equivalent.
\begin{enumerate}[\rm (1)]
 \item $C$ is a $(-d)$-CY configuration;
 \item There exists a $d$-symmetric   dg $k$-algebra $A$ with AR quiver  of $\CM A$ being $(\Z \Delta)_{C}/\SSS[d]$.
\end{enumerate}
\end{Thm}

To study the classification of configurations, Riedtmann \cite{Riedtmann} gave a geometrical description of configurations by  Brauer relations,  and Luo \cite{Luo} gave a description of Wiedemann's configuration by $2$-Brauer relations.  Similarly, we introduce \emph{maximal $d$-Brauer relations} (see Definition \ref{Def:Brauer relation}). 
It gives a nice description of $(-d)$-CY configurations of type $A_{n}$. This geometric model has been studied by Coelho Sim\~oes \cite[Theorem 6.5]{CS}. By using this model, we show the number of $(-d)$-CY configurations in $\Z A_{n}/\SSS[d]$ is $\frac{1}{n+1}\binom{(d+1)n+d-1}{n}$ (Corollary \ref{Cor:n}).
We develop several technical concepts and results on maximal $d$-Brauer relations and by using them we give another proof of Theorem \ref{Thm:intromain} for the case $\Delta=A_{n}$ (Theorem \ref{Thm:section}). 
In this case,  for any given  $(-d)$-CY configuration $C$, the corresponding  $d$-symmetric dg $k$-algebra is given explicitly by \emph{Brauer tree dg algebra} (see Section \ref{Brauertreedga} for the details). 
 The following table explains the comparison among different configurations.
\begin{center}
\begin{tabular}{|c|c|c|}
\hline  
$(-d)$-CY ($d\ge 1$)&$(-1)$-CY &$0$-CY \\
\hline  
$(-d)$-CY configuration&Riedtmann's configuration& Wiedemann's configuration\\
\hline
maximal $d$-Brauer relation & Brauer relation &$2$-Brauer relation\\
\hline
$d$-self-injective dg algebras & self-injective algebras& Gorenstein orders\\
\hline 
\end{tabular}
\end{center}

The paper is organized as follows. Section \ref{Section:preliminaries} provides the necessary material on dg algebras, extriangulated categories and translation quivers. 
In Section \ref{Section:CMdg}, we introduce Cohen-Macaulay dg modules and show some basic properties of them. 
Section \ref{Section:AR} deals with the Auslander-Reiten theory in $\CM A$. 
We compute the AR quiver of  a truncated polynomial dg algebra in Section \ref{Section:poly}. From its own point of view, this example is also interesting.
   We introduce (negative) CY-configurations and combinatorial configurations in Section \ref{Section:confi} and then show they coincide with each other in our context. 
 In Section \ref{Section:trivial}, we construct a class of self-injective dg algebras by taking trivial extension. In Section \ref{Section:mainsection}, we prove our main theorem that any CY configuration is given by simples of  symmetric dg algebras.
 In section \ref{Section:maximal},
 we introduce the maximal $d$-Brauer relations  and give a formula of  the number of $(-d)$-CY configurations in $\Z A_{n}/\SSS[d]$.
 We prove Theorem \ref{Thm:intromain} for the case $\Delta=A_{n}$ by constructing a Brauer tree dg algebra from given maximal $d$-Brauer relation.
   In Appendix \ref{appendixA}, we give a new proof of the bijection between $(-d)$-CY configurations and maximal $d$-Brauer relations.

\medskip\noindent
{\bf Acknowledgements }
The author would like  to thank his supervisor Osamu Iyama for many useful discussions and for his consistent support.
He also thanks Raquel Coelho Sim\~oes and David Pauksztello for pointing out their results on   $d$-Riedtmann configurations  in \cite{CS, CS2, CSP}.

\section{Preliminaries}\label{Section:preliminaries}
\setcounter{section}{1}

\subsection{Notations}
Throughout this paper, $k$ will be an algebraically closed field. All algebras, modules and categories are over the base field $k$. We denote by $D=\Hom_{k}(?, k)$ the $k$-dual. When we consider graded $k$-module, $D$ means the graded dual. We denote by $[1]$ the suspension functors for  the triangulated categories. 
Let $\T$ be a Krull-Schmidt $k$-linear category. We denote by $\ind \T$ the set of indecomposable objects in $\T$. Let $\cal S$ and $\cal S'$ be   two full subcategories  of $\T$. Denote by $\add \cal S$ the smallest full subcategory of $\T$ containing $\cal S$ and closed under isomorphisms, finite direct sums and direct summands. If $\T$ is a  triangulated category, we denote by  $\thick(\cal S)$ the smallest triangulated subcategory of $\T$ containing $\cal S$ and stable under direct summands. We denote by $\cal S\ast \cal S'$ the full subcategory of $\T$ consisting of objects $T\in \T$ such that there exists a triangle $X \ra T\ra Y\ra X[1]$ for $X\in \cal S$ and $Y\in \cal S'$.
 If $\Hom_{\T}(\cal S, \cal S')=0$, that is,  $\Hom_{\T}(S, S')=0$ for any $S\in \cal S$ and $S'\in \cal S'$, we denote $\cal S\ast \cal S'$ by $\cal S\perp \cal S'$. 
We define $\Filt \cal S$ as the category  $\bigcup_{n\ge 0}\underbrace{\cal S\ast\cal S\ast\cdots\ast\cal S}_{n}$. 
 If $\cal S=\{S\}$ has only one object, we write $\thick(\{S\})$ as $\thick(S)$, and we use same convention for $\Filt$ and $*$.

\subsection{DG algebras and the Nakayama functor}

Let $A$  be a dg $k$-algebra, that is, a graded algebra endows with a compatible structure of a complex. A (right) dg $A$-module is a graded $A$-module endows with a compatible structure of a complex. Let $\D(A)$ be the derived category of right dg $A$-modules (see \cite{Keller94, Keller06}). It is a triangulated category obtained from the category of dg $A$-modules by formally inverting all quasi-isomorphisms. The shift functor is given by the shift of complexes.

 Let $\per A=\thick(A_{A})\subset\D(A)$ be the perfect category and let $\D^{\bb}(A)$ be the 
full subcategory of $\D(A)$ consisting of the objects whose total cohomology is finite-dimensional.

We consider the derived dg functor 
\[\nu:=?\ot_{A}^{\bf L}DA: \D(A) \ra \D(A),\]
 called the \emph{Nakayama functor}. We have the following Auslander-Reiten formula.
\begin{Lem}\cite[Section 10.4]{Keller94}\label{Lem:bifunc}
 There is a bifunctorial isomorphism
     \begin{equation}\label{equation:Serre}
  D\Hom_{\D A}(X,Y)\cong \Hom_{\D A}(Y,\nu(X))           
      \end{equation}
     for $X\in \per A$ and $Y\in \D(A)$.
\end{Lem} 

It is clear that $\nu$ restricts to a triangle functor 
\begin{equation} \label{equation:Nakayama}
\nu: \per A \xra{} \thick(DA). \end{equation}
By Lemma \ref{Lem:bifunc}, \eqref{equation:Nakayama} is a triangle equivalence provides that $A$ has finite-dimensional cohomology in each degree. In this case, if we have $\per A=\thick(DA)$ (for example, $A$ is a finite-dimensional Gorenstein $k$-algebra), then $\nu$ defines a Serre functor on $\per A$. Immediately, we have the following result.
\begin{Lem}\label{Lem:degreefinite}
Assume $A$ has finite-dimensional cohomology in each degree and $\per A=\thick(DA)$ in $\D(A)$. Let $X, Y$ be two dg $A$-modules with finite-dimensional cohomology in each degree. Then the isomorphism \eqref{equation:Serre} also holds for $Y\in \per A$ and $X\in \D(A)$.
\end{Lem}

Let $A$ be a dg $k$-algebra and let $M$ be a dg $A$-module. Then $\h^{0}(A)$ is an ordinary $k$-algebra and we regard $\h^{n}(M)$ as a $\h^{0}(A)$-module for $n\in \Z$.
\begin{Lem}\label{Lem:KN}
Let $A$ be a dg $k$-algebra and $M\in \D(A)$.  Then
 \begin{enumerate}[\rm (1)]
  \item \cite[Lemma 4.4]{KN}
 For $P\in \add A$, the morphism of $k$-modules induced by $\h^{0}$
\[ \Hom_{\D(A)}(P, M) \ra \Hom_{\h^{0}(A)}(\h^{0}(P),\h^{0}(M))\]
is an isomorphism;
\item Dually, for $I\in \add DA$, the morphism of $k$-modules induced by $\h^{0}$
\[ \Hom_{\D(A)}(M, I) \ra \Hom_{\h^{0}(A)}(\h^{0}(M),\h^{0}(I))\]
is an isomorphism.
\end{enumerate}
\end{Lem}

We need the following lemma for later use.
\begin{Lem}\label{Lem:retraction}
Let $A$ be a  dg $k$-algebra and $M\in \D(A)$. Then 
 \begin{enumerate}[\rm (1)]
   \item Let $P\in \add A$ and $f\in \Hom_{\D(A)}(M, P)$. If the induced map $\h^{0}(f): \h^{0}(M)\ra \h^{0}(P)$ is surjective, then $f$ is a retraction in $\D(A)$;
   \item Let $I\in \add DA$ and $g\in  \Hom_{\D(A)}(I, M)$. If the induced map $\h^{0}(g): \h^{0}(I)\ra \h^{0}(M)$ is injective, then $g$ is a section in $\D(A)$.
 \end{enumerate}
\end{Lem}

\begin{proof} We only prove (1), since (2) is dual. 
Because $\h^{0}(P)$ is a projective $\h^{0}(A)$-module and $\h^{0}(f)$ is surjective, then $\h^{0}(f)$  is a retraction. Then by Lemma \ref{Lem:KN}, there is $p\in \Hom_{\D(A)}(P, M)$, such that $\h^{0}(f)\circ \h^{0}(p)=\rm{Id}_{\h^{0}(P)}$. By Lemma \ref{Lem:KN} again, we have $f\circ g=\rm{Id}_{P}$. Therefore $f$ is a retraction in $\D(A)$.
\end{proof}

\subsection{Non-positive dg algebras}\label{Section:nonpositivedg}

 We call a dg $k$-algebra $A$ \emph{non-positive} if it satisfies $\h^{i}(A)=0$ for $i>0$. 
 Write  $A$ as a complex over $k$,
\[ A:= \cdots \lra A^{-1} \xra{d^{-1}}  A^{0} \xra{d^{0}} A^{1} \lra \cdots \]
Consider the following standard truncation.
\[ A':= \cdots \lra A^{-1} \xra{d^{-1}} \ker d^{0} \lra 0 \lra \cdots \]
\noindent
It is easy to see $A'$ is a sub-dg-algebra of $A$ and the inclusion $A' \hra A$ is a quasi-isomorphism of dg $k$-algebras. Thus in this paper,  when we mention non-positive dg $k$-algebra $A$, we always assume that $A^{i}=0$ for $i>0$.
In this case, the canonical projection $A\ra \h^{0}(A)$ is a homomorphism of dg $k$-algebras (here we regard $\h^{0}(A)$ as a dg algebra concentrated in degree $0$). Then we can regard a module over $\h^{0}(A)$ as a dg module over $A$ via this homomorphism. This induces a natural functor $ \mod\h^{0}(A) \ra \D(A)$. 
Let $\{ S_{1}, \dots, S_{r}\}$ be the set of isomorphic classes of simple $\h^{0}(A)$-modules. We may regard them as simple dg $A$-modules.  For any $S_{i}$, there exists $P_{i}\in\add A$ such that $S_{i}=\Top \h^{0}(P_{i}):=\h^{0}(P_{i})/\rad \h^{0}(P_{i})$. By Lemma \ref{Lem:KN}, $P_{i}$ is indecomposable in $\D^{\bb}(A)$. Now we introduce the radical of $A$, which will be used later.
Let $P\in \add A$. We have a short exact sequence in $\mod \h^{0}(A)$
\[ 0\ra \rad \h^{0}(P) \ra \h^{0}(P) \xra{f} \Top \h^{0}(P)\ra 0.\] 
By Lemma \ref{Lem:KN}, there is a morphism $f'\in\Hom_{\D^{\rm b}(A)}(P, \Top\h^{0}(P))$ which is sent  to $f$ by  $\h^{0}$. Then we define the radical of $P$ in $\D^{\rm b}(A)$ as the third term of the following triangle.
\[  \rad P \ra P \xra{f'} \Top \h^{0}(P) \ra \rad P[1]. \]
It is well-known that simple dg $A$-modules generate $\D^{\bb}(A)$ in the following sense.
\begin{Prop}\label{Prop:Dbfinite}
Let $A$ be a non-positive dg $k$-algebra with $\h^{i}(A)$ finite-dimensional for $i\in\Z$. Then $\D^{\bb}(A)=\thick(\bop_{i=1}^{r} S_{i})$ and 
$\D^{\bb}(A)$ is Hom-finite.
\end{Prop}
\begin{proof}
The  first statement can be shown by truncations and induction easily. The second one is  from \cite[Theorem 3.1 (c)]{Keller94}.
\end{proof}
 
For $i\in \Z$, Let $\D^{\bb}_{\le i}$ (respectively, $\D^{\bb}_{\ge i}$) denote the full subcategory of $\D^{\bb}(A)$ consisting of those dg $A$-modules whose  cohomologies are concentrated in degree $\le i$ (respectively, $\ge i$).
This gives us a standard $t$-structure \cite{BBD} as the following result shows.
\begin{Prop}\cite[Proposition 2.1]{KY}\label{Prop:heart}
Let $A$ be a non-positive dg algebra. Then 
\begin{enumerate}[\rm(1)]
\item The pair
$(\D^{\bb}_{\le 0}, \D^{\bb}_{\ge 0})$ is a   $t$-structure on $\D^{\bb}(A)$;
\item 
For $n\le m$, we have the following, where $\cal H$ is the heart of $(\D^{\bb}_{\le 0}, \D^{\bb}_{\ge 0})$.
\[ \D^{\bb}_{\ge n}\cap \D^{\bb}_{\le m}=\cal H[-n]\ast \cal H[-n+1]\ast\cdots \ast \cal H[-m]; \]
\item
Moreover, taking $\h^{0}$ is an equivalence from the $\cal H$ to  $\mod \h^{0}(A)$, and the natural functor $\mod \h^{0}(A) \ra \D^{\bb}(A)$ is a quasi-inverse to this equivalence.
\end{enumerate}
\end{Prop}

We call a dg algebra $A$ \emph{proper}, if $A\in \D^{\bb}(A)$.
We end this section by a proposition, which plays an important role in the proof of Theorem \ref{Thm:section}. 
\begin{Prop}\label{Prop:key2}
Let $A$ be a non-positive proper dg $k$-algebra whose underlying graded algebra is a quotient $kQ/I$ of the path algebra  of a graded quiver $Q$. Let  $j$ and $j'$ be   vertices  in $Q$. 
\begin{enumerate}[\rm (1)]
\item
If  $\Hom_{\D^{\rm b}(A)}(S_{j},S_{j'}[l])\not=0$ for some $l>0$, then there exists a path from $j$ to $j'$ with degree bigger than $-l$;
\item Assume that  the differential of $A$ is zero and $I$ is an admissible ideal of $kQ$.
 If there is an arrow $j\ra j'$ with degree $-l\le 0$, then $\Hom_{\D^{\bb}(A)}(S_{j},S_{j'}[l+1])\not=0$.
 \end{enumerate}
\end{Prop}

\begin{proof}
(1)
Let $X_{0}:=S_{j}$.
For each $i\ge 0$,
we take  the following triangle,
\[ X_{i+1} \ra Q_{i} \ra X_{i} \ra X_{i+1} [1], \]
such that  $Q_{i}\in \add A[\ge0]$ and  the induced map $\h^{\ast}(Q_{i})\ra \h^{\ast}(X_{i})$ is the projective cover of $\h^{\ast}(X_{i})$ as a graded $\h^{\ast}(A)$-module. Then we have  an exact sequence 
 \[0\ra \h^{\ast}(X_{i+1}) \ra \h^{\ast}(Q_{i})\ra \h^{\ast}(X_{i})\ra 0.\] 
 Thus the composition $Q_{i+1} \ra X_{i+1} \ra Q_{i}$ is non-zero.  For each direct summand  $P_{a_{i}}[s_{i}]$ of $Q_{i}$ with a vertex $a_{i}$ of $Q$ and $s_{i}\in \Z$,  
   there exists a direct summand $P_{a_{i-1}}[s_{i-1}]$ of $ Q_{i-1}$ with a vertex $a_{i-1}$ in  $Q$ and $s_{i-1}\in \Z$,  such that $\Hom_{\D^{\bb}(A)}(P_{a_{i}}[s_{i}],P_{a_{i-1}}[s_{i-1}])\not=0$. Then there is a path  $a_{i-1}\leadsto a_{i}$ with degree $s_{i-1}-s_{i}$. Repeating  this,  we obtain a path $j=a_{0}\leadsto a_{1} \leadsto \cdots \leadsto a_{i}$ with degree $\sum_{k=1}^{i}(s_{k-1}-s_{k})=s_{0}-s_{i}=-s_{i}$.  
   
   By the construction above, we have
$S_{j} \in Q_{0}\ast Q_{1}[1] \ast Q_{2}[2]\ast \cdots \ast Q_{l}[l] \ast \D^{\bb}_{\le -l-1}$. Since  $\Hom_{\D^{\bb}(A)}(Y, S_{j'}[l])=0$ for any $Y\in \D^{\bb}_{\le -l-1}$ and by our assumption   $\Hom_{\D^{\bb}(A)}(S_{j}, S_{j'}[l])\not=0$, then there exists a non-zero map from  some object in $\add (Q_{0}\ast Q_{1}[1] \ast Q_{2}[2]\ast \cdots \ast Q_{l}[l])$ to $S_{j'}[l]$, which implies  $P_{j'}[l]\in \add Q_{k}[k]$ for some $0\le k\le l$. Since $Q_{0}=P_{j}$ and $l$ is positive, we have $1\le k\le l$.
 Since $P_{j'}[l-k]\in \add Q_{k}$, by our argument above, there is a path from $j$ to $j'$ with degree $k-l$, which is bigger than $-l$.

(2)
Consider the following triangle, 
\[ \rad P_{j} \ra P_{j} \ra S_{j} \ra \rad P_{j}[1].\]
Since  the differential of $A$  is zero, $I$ is admissible, and there is an arrow $j\ra j'$ with degree $-l$, then $S_{j'}\in \add \h^{-l}(\frac{\rad P_{j}}{\rad^{2} P_{j}})$. Then
 the composition $\rad P_{j}\ra \frac{\rad P_{j}}{\rad^{2} P_{j}} \ra \h^{-l}(\frac{\rad P_{j}}{\rad^{2} P_{j}})[l] \ra S_{j'}[l]$  is non-zero.  
Thus  $\Hom_{\D^{\bb}(A)}(\rad P_{j}, S_{j'}[l])\not=0$. 
Applying the functor $\Hom_{\D^{\bb}(A)}(?, S_{j'}[l+1])$ to the triangle above,   
we obtain an exact sequence
\[ \Hom_{\D^{\bb}(A)}(S_{j}, S_{j'}[l])\ra \Hom_{\D^{\bb}(A)}(P_{j}, S_{j'}[l]) \ra \Hom_{\D^{\bb}(A)}(\rad P_{j},S_{j'}[l]) \ra \Hom_{\D^{\bb}(A)}(S_{j},S_{j'}[l+1]). \]
By dividing into two cases, $(l,j)=(0, j')$ or not, one   can check that the left map is always surjective.
 Then $\Hom_{\D^{\bb}(A)}(S_{j},S_{j'}[l+1])\neq0$.
\end{proof}

\subsection{Extriangulated categories}\label{Section:extri}

In this section, we briefly recall the definition and basic properties of extriangulated categories from \cite{Palu}.  We omit some details here, but the reader can find them in \cite{Palu}.

Let $\CCC$ be an additive category equipped with an additive bifunctor $\EEE: \CCC^{\rm op}\ot \CCC \ra Ab$. For any pair of objects $A,C\in \CCC$, an element $\delta\in \EEE(C,A)$ is called an \emph{$\EEE$-extension}. Let $\s$ be a correspondence which associates an equivalence class $\s(\delta)=[A\xra{x}B\xra{y}C]$ to any $\EEE$-extension $\delta\in \EEE(C,A)$. This $\s$ is called a \emph{realization} of $\EEE$ if it makes the diagrams  in \cite[Definition 2.9]{Palu} commutative. 
  A triple $(\CCC, \EEE,\s)$ is called an \emph{extriangulated category} if it satisfies the following conditions.
\begin{enumerate}[\rm(1)]
 \item $\EEE: \CCC^{\rm op}\ot \CCC\ra Ab$ is an additive bifunctor;
 \item $\s$ is an additive realization of $\EEE$;
 \item $\EEE$ and $\s$ satisfy the compatibility conditions in \cite[Definition 2.12]{Palu}.
\end{enumerate}

Extriangulated categories is a generalization of exact categories and triangulated categories. Let us see some easy examples.
\begin{Ex}
\begin{enumerate}[\rm(1)]
  \item Let $\CCC$ be an exact category. Then $\CCC$ is extriangulated by taking $\EEE$ as the bifunctor $\Ext_{\CCC}^{1}(?,?):\CCC^{\rm op}\ot \CCC\ra Ab$ and for any $\delta\in \Ext_{\CCC}^{1}(C,A)$, taking $\s(\delta)$ as the equivalence class of short exact sequences (=conflations) correspond to $\delta$;
   \item Let $\CCC$ be a triangulated category. Then $\CCC$ is extriangulated by taking $\EEE$ as the bifunctor $\Hom_{\CCC}(?,?[1]): \CCC^{\rm op}\ot \CCC\ra Ab$, and for any $\delta\in\Hom_{\CCC}(C,A[1])$, taking $\s(\delta)$ as the equivalence class of the  triangle
   $A \ra B \ra C \xra{\delta} A[1]$;
   
   \item Let $\CCC$ be a triangulated category and let $\mathscr D$ be an extension-closed (that is, for any triangle $X\ra Y\ra Z \ra X[1]$ in $\CCC$, if  $ X, Z\in \mathscr{D}$, then $Y\in \mathscr{D}$) subcategory of $\CCC$. Then $\mathscr D$ has an extriangulated structure given by restricting the extriangulated structure of $\CCC$ on $\mathscr D$.
\end{enumerate}
\end{Ex}

Let $(\CCC, \EEE, \s)$ be an extriangulated category. An  object $X$ in $\CCC$ is called \emph{projective}  if $\EEE(X,Y)=0$  for any $Y\in \CCC$. We say $\CCC$ has \emph{enough projective objects} if for any $Y\in \CCC$, there exists  $Z\in \CCC$ and $\delta\in \EEE(Y,Z)$, such that the middle term of the realization $\s(\delta)$ is projective. 
We denote by $\mathscr P$ (resp. $\mathscr I$) the subcategory of projective (resp. injectvie) objects.
When $\CCC$ has enough projective (resp. injective) objects, we define \emph{the stable (resp. costable) category of $\CCC$} as the ideal quotient $\un{\CCC}:=\CCC/[\mathscr P]$ (resp. $\overline{\CCC}:=\CCC/[\mathscr I]$). 
 We call $\CCC$  \emph{Frobenius} if it has enough projective objects and enough injective objects, and projective objects coincide with injective ones. In this case $\un{\CCC}$ coincides with $\overline{\CCC}$, and we call $\un{\CCC}$ \emph{the stable category of $\CCC$}. 

\begin{Prop}\cite[Corollary 7.4]{Palu}\label{Prop:Frobeniusextri}
Let $(\CCC, \EEE, \s)$ be a Frobenius extriangulated category and let $\cal I$ be subcategory of injective objects. Then $\un{\CCC}$ is a triangulated category.
\end{Prop}

\subsection{Auslander-Reiten theory in extriangulated categories}\label{Section:ARtheoryinex}
Let us briefly recall  Auslander-Reiten theory in extriangulated categories form \cite{INP}.
In this subsection, let $(\CCC,\EEE, \s)$ be an extriangulated category.
\begin{Def}\cite[Definition 2.1]{INP}
A non-split $\EEE$-extension $\delta\in \EEE(C,A)$ is said to be \emph{almost split} if it satisfies the following conditions
\begin{enumerate}[\rm (1)]
 \item $\EEE(C,a)(\delta)=0$ for any non-section $a\in \CCC(A,A')$; 
 \item $\EEE(c,A)(\delta)=0$ for any non-retraction $c\in \CCC(C',C)$.
\end{enumerate}
\end{Def}

We say that $\CCC$ \emph{has right almost split extensions} if for any endo-local non-projective object $A\in \CCC$, there exists an almost split extension $\delta\in \EEE(A,B)$ for some $B\in \CCC$. Dually, we say that  $\CCC$ \emph{has left almost split extensions} if for any endo-local non-projective object $B\in \CCC$, there exists an almost split extension $\delta\in \EEE(A,B)$ for some $A\in \CCC$. We say that $\CCC$ \emph{has almost split extensions} if it has right and left almost split extensions. 

Let $A\in \CCC$. If there exists an almost split extension $\delta\in \EEE(A,B)$, then it is unique up to isomorphism of $\EEE$-extensions.  

\begin{Def}\cite[Definition 3.2]{INP}
Let $(\CCC, \EEE, \s)$ be a $k$-linear extriangulated category.
\begin{enumerate}[\rm(1)]
 \item A \emph{right Auslander-Reiten-Serre (ARS) duality} is a pair $(\tau, \eta)$ of an additive functor $\tau: \un{\CCC}\ra \overline{\CCC}$ and a binatural isomorphism
 \[ \eta_{A,B}: \un{\CCC}(A,B) \simeq D\EEE(B,\tau A) \mbox{ for any } A,B\in \CCC;\]
 \item If moreover $\tau$  is an equivalence, we say that $(\tau, \eta)$ is an \emph{Auslander-Reiten-Serre (ARS) duality}. 
\end{enumerate}
\end{Def}

We say a $k$-linear extriangulated category $(\CCC,\EEE,\s)$ is \emph{Ext-finite}, if $\dim_{k}\EEE(A,B)<\infty$ for any $A, B\in \CCC$.

\begin{Prop}\cite[Theorem 3.4]{INP} \label{Prop:INP}
Let $\CCC$ be a $k$-linear Ext-finite Krull-Schmidt extriangulated category. Then the following are equivalent.
\begin{enumerate}[\rm(1)]
 \item $\CCC$ has almost split extensions;
 \item $\CCC$ has an Auslander-Reiten-Serre duality.
\end{enumerate}
\end{Prop}

The following characterization of almost split extensions are analogous to the corresponding result on Auslander-Reiten triangles (see \cite[Proposition I.2.1]{Reiten}) and on almost split sequences (see \cite{ARS}).
\begin{Prop}\label{Prop:characAR}
Assume $(\CCC,\EEE,\s)$ has almost split extensions. Assume $A\in \CCC$ is an end-local object and $\delta\in \EEE(A,B)$. Then the following are equivalent.
\begin{enumerate}[\rm(1)]
\item $\delta$ is an almost split extension;
\item $\delta$ is in the socle of $\EEE(A,B)$ as right $\End_{\CCC}(A)$-module and $B\cong \tau(A)$;
\item $\delta$ is in the socle of $\EEE(A,B)$ as left $\End_{\CCC}(B)$-module and $B\cong \tau(A)$.
\end{enumerate}
\end{Prop}

 \subsection{Translation quivers}\label{Section:quiver}
 
We  recall some definitions and notations concerning quivers. 
A \emph{quiver $Q = (Q_{0} , Q_{1}, s, t)$} is given by the set $Q_{0}$ of its vertices, the set $Q_{1}$ of its arrows, a source map $s$ and a target map $t$. If $x\in Q_{0}$ is a vertex, we denote by $x^{+}$ the set of direct successors of $x$, and by $x^{-}$ the set of its direct predecessors. We say that Q is \emph{locally finite} if for each vertex $x\in Q_{0}$, there are finitely many arrows ending at $x$ and starting at $x$. An automorphism group $G$ of $Q$ is  said to be \emph{weakly admissible}  if for each $g \in G  \backslash \{1\}$ and for each $x\in Q_{0}$ , we have $x^{+}\cap (gx)^{+}=\varnothing$.


A \emph{stable translation quiver $(Q, \tau)$} is a locally finite quiver $Q$ without double arrows with a bijection $\tau : Q_{0} \ra Q_{0}$ such that $(\tau x)^{+} = x^{-}$ for each vertex $x$. For each arrow $\alpha : x \ra y$, we denote by  $\sigma\alpha$  the unique arrow $\tau y\ra x$.

\begin{Def}\label{Def:newstable}
Let $Q$ be a stable translation quiver and $C$ be a subset of $Q_{0}$. We define a translation quiver $Q_{C}$ by adding to $Q_{0}$ a vertex $p_{c}$ and two arrows $c\ra p_{c} \ra \tau^{-1}(c)$ for each $c\in C$. The translation of $Q_{C}$ coincides with the translation of $Q$ on $Q_{0}$ and is not defined on $\{ p_{c}\mid c\in C \}$.
\end{Def}

Let $\Delta$ be an oriented tree, then the  \emph{repetition quiver of $\Delta$} is defined as follows:
  \begin{enumerate}[\rm (1)]
  \item $(\Z\Delta)_{0}=\Z\times \Delta_{0}$
  \item  $(\Z\Delta)_{1}=\Z\times \Delta_{1}\cup \sigma(\Z\times \Delta_{1})$ with arrows $(n,\alpha):(n,x)\ra (n,y)$ and $\sigma(n,\alpha): (n-1,y)\ra (n,x)$ for each arrow $\alpha :x \ra y$ of $\Delta$.
  \end{enumerate}
The quiver $\Z\Delta$ with the translation $\tau(n, x) = (n-1, x)$ is a stable translation quiver which does not depend (up to isomorphism) on the orientation of $\Delta$ (see \cite{Riedtmann1}). 

From now on, we assume $\Delta$ is a Dynkin diagram. Let us fix a numbering and an orientation of the simply-laced Dynkin trees.
\[ {\small \xymatrixrowsep{1pc}\xymatrix{ A_{n}(n\ge1): & 1 \ar[r] & 2 \ar[r] & \cdots \ar[r] &n-1 \ar[r] & n \\  & & & & &n-1 \ar[dl]
\\ D_{n}(n\ge 4): & 1 \ar[r] & 2 \ar[r] & \cdots \ar[r] &n-2 \\ & & & & & n \ar[ul]
\\ & & & 4
\\ E_{n}(n=6,7,8): &1 & 2\ar[l] &3\ar[u] \ar[l] \ar[r] & 5 \ar[r] & \cdots \ar[r] & n
} }\]
We define the ``Nakayama permutation'' $\SSS$ of $\Z \Delta$ as follows:
\begin{enumerate}[\rm $\bullet$]
\item if $\Delta=A_{n}$, then $\SSS(p,q)=(p+q-1, n+1-q)$;
\item if $\Delta=D_{n}$ with $n$ even, then $\SSS=\tau^{-n+2}$;
\item if $\Delta=D_{n}$ with $n$ odd, then $\SSS=\tau^{-n+2}\phi$, where $\phi$ is the automorphism which exchanges $n$ and $n-1$;
\item if $\Delta=E_{6}$, then $\SSS=\phi \tau^{-5}$, where $\phi$ is the automorphism which exchanges $2$ and $5$, and $1$ and $6$;
\item if $\Delta=E_{7}$, then $\SSS=\tau^{-8}$;
\item if $\Delta=E_{8}$, then $\SSS=\tau^{-14}$.
\end{enumerate}
By \cite[Proposition 6.5]{Gabriel}, when we identify the Auslander-Reiten quiver of $k\Delta$ as the full subquiver of $\Z \Delta$, the Nakayama functor is related to $\SSS$ defined above. We can also define ``shift permutation'' $[1]$ of $\Z\Delta$ by $\SSS\tau^{-1}$.

 \section{Cohen-Macaulay dg modules} \label{Section:CMdg}

Let $A$ be a dg $k$-algebra.
In this section, we assume $A$ satisfies Assumption \ref{assumption}.

\begin{Def}\label{Def:CM}
\begin{enumerate}[\rm(1)]
\item
A dg $A$-module $M$ is called  \emph{Cohen-Macaulay}  if $M\in \D_{\le 0}^{\bb}(A)$ and  $\Hom_{\D^{\bb}(A)}(M, A[i])=0$ for $i>0$;
\item
We denote by  $\CM A$  the subcategory of $\D^{\bb}(A)$ consisting of Cohen-Macaulay dg $A$-modules.
\end{enumerate}
\end{Def}

If $A$ is an ordinary $k$-algebra, then $\CM A$ defined here is canonically equivalent to the usual one.
We mention that Yekutieli  also introduced Cohen-Macaulay dg modules (see \cite[Section 8]{Ye2}), but they are  different from ours.
Now we introduce some special  dg algebras which are the main  objects in this paper. 

\begin{Def}\label{Def:dselfinjective}
Let $A$ be a non-positive dg $k$-algebra and let $d$ be a positive integer. 
\begin{enumerate}[\rm(1)]
\item We call $A$  \emph{$d$-self-injective} if 
$\add A=\add(DA[d-1])$ in $\D(A)$;
\item 
We call $A$  \emph{$d$-symmetric} if  
$\add A=\add(DA[d-1])$ in $\D(A^{\rm e})$.
\end{enumerate}
\end{Def}

Since in our setting $A$ is Gorenstein,  then the equivalence  \eqref{equation:Nakayama} induces the following triangle auto-equivalence. 
\begin{equation*}\nu: \per A \simeq \per A, \end{equation*}
and moreover we get another triangle auto-equivalence
$\nu:\D^{\bb}(A)\simeq \D^{\bb}(A)$. 
In particular, $\nu$ is a Serre functor on $\per A$.
 We give another description of $\CM A$ as follows.
 \begin{Prop}\label{Prop:cm}
 \begin{enumerate}[\rm(1)]
\item $\CM A=\D_{\le 0}^{\bb}\cap \nu^{-1}(\D^{\bb}_{\ge 0})$;
\item In particular,  If $A$ is a $d$-self-injective dg algebra, then $\CM A= \D_{\le 0}^{\bb}\cap \D_{\ge -d+1}^{\bb}$. 
\end{enumerate}
\end{Prop}

\begin{proof}
(1)
By definition, 
\[ A[<0]^{\perp} = \{ X\in \D^{\bb}(A)\ |\ \Hom_{\D^{\bb} (A)}(A[<0], X)=0\}= \D^{\bb}_{\le 0}. \]
By Lemma \ref{Lem:degreefinite}, $\h^{<0}(\nu(X))=\Hom_{\D^{\bb}(A)}(A[>0], \nu(X))=D\Hom_{\D^{\bb}(A)}(X, A[>0])$, then 
 \[ {}^{\perp}A[>0] 
 =\{ X\in \D^{\bb}(A)\ |\  \h^{<0}(\nu(X))=0 \}
  =\{ X\in \D^{\bb}(A)\ |\  \nu(X)\in \D^{\bb}_{\ge 0}\}.\]
    Then $\CM A=\D^{\bb}_{\le 0}\cap \nu^{-1}(\D_{\ge 0}^{\bb})$. 
  
 (2) Let $X\in \CM A$. If $A$  is $d$-self-injective, then we have
\[\Hom_{\D A}(X, A[>0])=\Hom_{\D^{\bb}(A)}(X, DA[>d-1])=D\Hom_{\D^{\bb}(A)}(A[>d-1], X)=0,\]
which  implies $X\in \D_{\ge -d+1}^{\bb}$. So $\CM A= \D_{\le 0}^{\bb}\cap \D_{\ge -d+1}^{\bb}$.
 \end{proof}

The first properties of $\CM A$ are the following, which  are  analogues  of the well-known properties of Cohen-Macaulay modules. We refer to Section \ref{Section:extri}
 for the notion of extriangulated category. We call the Verdier quotient $\D^{\bb}(A)/\per A$  the \emph{singularity category $\D_{\sg}(A)$} of $A$.
\begin{Thm}\label{Thm:properties}
Let $A$ be a dg $k$-algebra satisfying Assumption \ref{assumption}. Then
\begin{enumerate}[\rm (1)]
\item $\CM A$ is an Ext-finite Frobenius extriangulated category with $\Proj (\CM A)=\add A$; 
\item The stable category $\underline{\CM}A:=(\CM A)/[\add A] $ is a triangulated category;
\item The composition $\CM A \hookrightarrow \D^{\bb}(A)\ra \D^{\bb}(A)/\per A$ induces a triangle equivalence
\[ \underline{\CM} A=(\CM A)/[\add A] \simeq  \D^{\bb}(A)/\per A=\D_{\sg}(A). \]
\end{enumerate}
\end{Thm}

\begin{proof}
 By our definition, $\CM A$ is an extension-closed subcategory of $\D^{\bb}(A)$, then it has a natural extriangulated category structure by restricting the triangles of $\D^{\bb}(A)$ on $\CM A$ (see \cite[Remark 2.18]{Palu}). By Proposition \ref{Prop:Dbfinite}, $\D^{\bb}(A)$ is Hom-finite, then so is $\CM A$. It implies that $\CM A$ is Ext-finite and $\add A$ is functorially finite in $\CM A$.
 
  Since $\Hom_{\CM A}(P, X[1])=0=\Hom_{\CM A}(X, P[1])$ for any $P\in \add A$ and $X\in \CM A$,  then we have $\add A\subset \sf{Proj}(\CM A)\cap \sf{Inj}(\CM A)$. For any $X\in \CM A$, we 
consider the right $(\add A)$-approximation $P \ra X$, which extends to the following triangle in $\D^{\bb}(A)$.
\[   Y\ra P\ra X\ra Y[1]. \]
It is easy to check $Y\in \CM A$ by applying the functors $\Hom_{\D^{\bb}(A)}(A[<0],?)$ and $\Hom_{\D^{\bb}(A)}(?, A[>0])$ to the triangle above. So $\CM A$ has enough projectives. Similarly, it also has enough injectives. 

Finally, we show $\Proj(\CM A)= \add A= \Inj(\CM A)$. Assume $X\in \CM A$ is projective,  and take a right 
 $(\add A)$-approximation $P \ra X$. As we have shown above, we have a triangle  $Y\ra P\ra X\ra Y[1]$
where $Y\in \CM A$. Since $X$ is projective, then $\Hom_{\CM A}(X, Y[1])=0$. Then the triangle splits and thus $X\in \add A$. So $\sf{Proj}(\CM A)= \add A$. Similarly, one can show $ \sf{Inj}(\CM A)=\add A$. Then by Proposition \ref{Prop:Frobeniusextri}, $\un{\CM} A$ is a triangulated category.

For the last statement, applying \cite[Corollary 2.1]{Iyama} to $\cal T=\D^{\bb}(A)$ and $\cal P= \add A$, we have $\un{\CM} A$ is triangle equivalent to $\D^{\bb}(A)/\per A$.
\end{proof}

Immediately,  we have the following.
\begin{Cor}
In Theorem \ref{Thm:properties}, 
$\D^{\bb}(A)=\per A$ if and only if $\CM A=\add A$.
\end{Cor}

Recall from \cite{Palu}, the suspension functor in $\un{\CM}A$ is given by the cone of a left $(\add A)$-approximation $X\ra P\ra \Omega^{-1} X \ra X[1]$ for $X\in \CM A$. The following result is an analogue of the well-known properties for classical Gorenstein rings.
\begin{Prop}\label{Prop:Groper}
Let $A$ be a dg algebra satisfying Assumption \ref{assumption}. Then 
\begin{enumerate}[\rm (1)]
\item There is a duality $(\ )^{*}=\RshHom_{A}(?, A):\D^{\bb}(A) \xra{\simeq} \D^{\bb}(A^{\rm op})$, which restricts to a duality $\CM A \xra{\simeq} \CM (A^{\rm op})$;
\item For $X, Y\in \CM A$ and $i>0$, we have $\Hom_{\D^{\bb}(A)}(X,Y[i])=\Hom_{\un{\CM} A}(X, \Omega^{-i}Y)$.
\end{enumerate}
\end{Prop}
\begin{proof}
(1) The functor $(\ )^{*}: \D(A)\ra \D(A^{\rm op})$ restricts to a duality  $\per A\xra{\simeq} \per (A^{\rm op})$.  For any $M\in \D^{\rm b}(A)$, it is clear that $\RshHom_{A}(M, DA)=DM\in \D^{\rm b}(A^{\rm op})$. Since $A$ is Gorenstein, then $A\in  \thick(DA)$ and moreover, $(\ )^{*}$ also induces a duality  $\D^{\bb}(A) \xra{\simeq} \D^{\bb}(A^{\rm op})$.
Since $\CM A=A[<0]^{\perp}\cap {}^{\perp}A[>0]$, it is clear that $(\ )^{*}$ restricts to a functor $(\ )^{*}: \CM A\ra \CM (A^{\rm op})$ and it is a duality. We have the following diagram.
\[ \xymatrix{ \D^{\bb}(A)  \ar[r]^{\simeq}&   \D^{\bb}(A^{\rm op})\\ \CM A \ar@{^{(}->}[u]\ar[r]^{\simeq} &  \CM (A^{\rm op}). \ar@{^{(}->}[u]
}\]

(2) Consider the following triangle  induced by the left $(\add A)$-approximation of $Y$,
\[ Y\ra Q \ra \Omega^{-1} Y \ra Y[1]. \]
Applying $\Hom_{\D^{\bb}(A)}(X, ?)$ to the triangle  above, since $\Hom_{\D^{\bb}(A)}(X, A[>0])=0$, we see
 \[\Hom_{\D^{\bb}(A)}(X, \Omega^{-1}Y[t])=\Hom_{\D^{\bb}(A)}(X, Y[t+1])\] 
 for $t \ge 1$. 
Moreover, we have the following exact sequence.
\[ \Hom_{\D^{\bb}(A)}(X, Q) \ra \Hom_{\D^{\bb}(A)}(X, \Omega^{-1}Y) \ra \Hom_{\D^{\bb}(A)}(X,Y[1]) \ra 0.\]
Since $ \Hom_{\D^{\bb}(A)}(A, Y[1])=0$, then  for $P\in \add A$, every map $P\ra \Omega^{-1}Y$ has a decomposition $P \ra Q \ra \Omega^{-1}Y$. Then \[\Hom_{\un{\CM}A}(X, \Omega^{-1}Y)=\Hom_{\D^{\bb}(A)}(X, \Omega^{-1}Y)/\Hom_{\D^{\bb}(A)}(X,Q),\]
 which is isomorphic to $\Hom_{\D^{\bb}(A)}(X, Y[1])$.
Then by induction, one can show $\Hom_{\D^{\bb}(A)}(X,Y[i])=\Hom_{\un{\CM} A}(X, \Omega^{-i}Y)$ holds for any $i\ge 1$.
\end{proof}


 The following proposition tells us that when $\CM A$ is an ordinary Frobenius category for a $d$-self-injective dg algebra $A$. 

\begin{Prop}
Assume $A$ is a $d$-self-injective dg $k$-algebra. Then $\CM A$ is a Frobenius category with $\add A$ as projective objects if and only if $d=1$ $($that is, $A$ has total cohomology concentrated in degree $0)$.
 \end{Prop}

\begin{proof}
If $A$ has total cohomology concentrated in degree $0$, then $A$ is quasi-isomorphic to $\h^{0}(A)$. 
In this case, $\h^{0}(A)$ is a Gorenstein $k$-algebra and $\CM A$ is equivalent to $\CM \h^{0}(A)$, which is Frobenius.
 
On the other hand, suppose $X$ is a non-zero object of $\CM A$. If $\CM A$ is a Frobenius category with $\add A$ as projective objects, then 
\[\Hom_{\CM A}(A, X)=\Hom_{\D^{\bb}( A)}(A, X)=\h^{0}(X)\not=0\]
which implies $X\not \in \D_{\le -1}^{\bb}$. So $\CM A \cap \D_{\le -1}^{\bb}=0$. But by Proposition \ref{Prop:cm}, $\CM A= \D_{\le 0}^{\bb}\cap \D_{\ge -d+1}^{\bb}$. Then $d=1$, which implies that $A$ has total cohomology concentrated in degree $0$.
\end{proof}

\section{Auslander-Reiten theory in $\CM A$}\label{Section:AR}

We assume that all the dg $k$-algebras considered in this section satisfy Assumption \ref{assumption}.
\subsection{Serre duality and almost split extensions}
The aim of this section is to prove the following theorem.
\begin{Thm}\label{Thm:AR}

\begin{enumerate}[\rm(1)]
\item
 $\underline{\CM}A$ admits a Serre functor  $\nu[-1]=?\ot_{A}^{\bf L}DA[-1]$;
 \item $\CM A$ admits almost split extensions.
 \end{enumerate}
\end{Thm}

We first show $\underline{\CM } A$ admits a Serre functor. We will consider it in a general setting given in \cite[Section 1.2]{Amiot}.
Let $\T$ be a $k$-linear Hom-finite triangulated category and $\cal N$ be a thick subcategory of $\T$. Assume $\cal T $ has an auto-equivalence $S$, which gives a relative Serre duality in the sense that 
$S(\cal N)\subset \cal N$ and there exists a functorial isomorphism for any $X\in \cal N$ and $Y\in \T$
\[ D\Hom_{\T}(X,Y) \simeq \Hom_{\T}(Y, SX).\]

\begin{Def}\cite[Definition 1.2]{Amiot}
Let $X$ and $Y$ be objects in $\T$. A morphism $p: P\ra X$ is called a \emph{local $\cal N$-cover of $X$ relative to $Y$} if $P$ is in $\cal N $ and it induces an exact sequence
\[  0 \ra  \Hom_{\T}(X, Y) \xra{p^{*}}\Hom_{\T}(P, Y). \]
Dually, let $Y$ and $Z$ be objects in $\T$. A morphism $q: Y\ra Q$ is called a \emph{local $\cal N$-envelop of $Y$ relative to $Z$} if $Q$ is in $\cal N$ and it induces an exact sequence
\[  0 \ra  \Hom_{\T}(Z, Y) \xra{q_{*}}\Hom_{\T}(Z, Q). \]
\end{Def}

Amiot gave the following sufficient condition for $\T/{\cal N}$ to admit a Serre functor.
\begin{Prop}\label{Amiot}\cite[Theorem 1.3]{Amiot}
Assume for any $X, Y \in \T$, there is a local $\cal N$-cover of $X$ relative to $Y$ and a local $\cal N$-envelop of $S X$ relative to $Y$. Then the quotient category $\T/\cal N$ admits a Serre functor given by $S[-1]$.
\end{Prop}

To check the condition in Proposition \ref{Amiot}, the following  lemma is useful.

\begin{Lem}\cite[Proposition 1.4]{Amiot}\label{Lem:cover}
Let $X$ and $Y$ be two objects in $\T$. If for any $P\in \cal N$ the vector space $\Hom_{\T}(P, X)$ and $\Hom_{\T}(Y, P)$ are finite-dimensional, then the existence of a local $\cal N$-cover of $X$ relative to $Y$ is equivalent to the existence of a local $\cal N$-envelop of $Y$ relative to $X$.
\end{Lem}

In our setting, to apply Proposition \ref{Amiot}, we need the following  observation.
\begin{Lem}\label{cover}
For any $X, Y\in \D^{\bb}(A)$, there exists an object $P_{X}\in \per A$ with a morphism $P_{X}\xra{p} X$ such that we have the following exact sequence.
\[  0 \ra  \Hom_{\D^{\bb}(A)}(X, Y) \xra{p^{*}}\Hom_{\D^{\bb}(A)}(P_{X}, Y). \]
\end{Lem}

\begin{proof}
Since $A$ is non-positive and $X, Y\in \D^{\bb}(A)$, by truncation, we may assume 
\begin{eqnarray*} 
&X:=[ \cdots 0 \ra X^{m} \xra{d_{m}} X^{m+1} \xra{d_{m+1}} \cdots \xra{d_{n-1}} X^{n} \ra 0 \ra \cdots ], \\
&Y:= [\cdots 0 \ra Y^{s} \xra{d_{s}} Y^{s+1} \xra{d_{s+1}} \cdots \xra{d_{t-1}} Y^{t} \ra 0 \ra \cdots ].  \end{eqnarray*}
Apply induction on $n-s$.

If $n-s<0$, then $  \Hom_{\D^{\bb}(A)}(X, Y)=0$, we can take any object in $\per A$ as $P_{X}$.

Now assume the result is true for $n-s=k$. Consider the case $n-s= k+1$.

There exists $Q_{X}\in \add A[-n]$
 and a morphism $p: Q_{X}\ra X$ such that $\h^{n}(p)$ is surjective. Then $\h^{i\geq n}(\rm{cone}(p))=0$. Let $Z=\rm{cone}(p)$. By our assumption, there exists $P_{Z}\in \per A$ with a morphism $r: P_{Z}\ra Z$ satisfies our condition. 
By the Octahedral Axiom, we have the following diagram.
 \[ \xymatrixcolsep{2pc}\xymatrix{ Q_{X} \ar@{=}[d] \ar[r]  & P_{X} \ar[r] \ar[d]& P_{Z} \ar[d]\ar[r]& Q_{X}[1] \ar@{=}[d]\\
  Q_{X} \ar[r]  & X \ar[r]  \ar[d]& Z \ar[r] \ar[d]& Q_{X}[1]  \\
  & T\ar[d] \ar@{=}[r]& T \ar[d]\\
  & P_{X}[1] \ar[r]& P_{Z}[1]&  
  }\] 
 Then it is easy to check $P_{X}$ is the cover we want.
\end{proof}

Now we are ready to prove Theorem \ref{Thm:AR}.

\begin{proof}[Proof  of Theorem \ref{Thm:AR}]
(1) Since $\per A=\thick(DA)$, then $\nu$ induces triangle equivalences $\D^{\bb}(A)\simeq \D^{\bb}(A)$ and $\per A\simeq \per A$. Moreover, $\nu$ gives a relative Serre duality by Lemma \ref{Lem:bifunc}.
We only need to show the conditions in Proposition $\ref{Amiot}$ hold in our setting. Because $\D^{\bb}(A)$ is Hom-finite by Proposition \ref{Prop:Dbfinite}, then by Lemma \ref{Lem:cover}, it suffices to check
 the existence of local $\per A$-cover. This has been proved in Lemma \ref{cover}. So the assertion is true.

(2)
By Proposition \ref{Prop:Dbfinite}, $\D^{\bb}(A)$ is Hom-finite, then $\CM A$ is Ext-finite (see Section \ref{Section:ARtheoryinex}).
 It is clear that $\CM A$ is a $k$-linear Krull-Schmidt extriangulated category. Moreover,  $\un{\CM} A$ admits a Serre functor by (1), then by 
Proposition \ref{Prop:INP}, $\CM A$ admits almost split extensions.
\end{proof}

We give the following lemma for later use.
\begin{Lem}\label{nonsplit}
Let $X$ be an non-projective indecomposable object in ${\CM }A$. Let $\tau$ be the Auslander-Reiten translation. If $\End_{{\CM}A}(X)=k$, then any non-split extension
\[ \tau(X) \xra{f} Y \xra{g} X \]
is an almost split extension.
\end{Lem}

\begin{proof}
By Proposition \ref{Prop:characAR}, it is clear.
\end{proof}

\subsection{Cohen-Macaulay approximation}

In this subsection, we show the following result.

\begin{Thm}\label{Thm:cmappro}
\begin{enumerate}[\rm (1)]
\item $\CM A$ is functorially finite in $\D^{\bb}(A)$;
\item  More precisely,  we have the following result,   where $(\ )^{*}=\RshHom_{A^{\rm op}}(?, A^{\rm op}):\D^{\bb}(A^{\rm op}) \xra{\simeq} \D^{\bb}(A).$
 \[\D^{\bb}(A)=\CM A \perp \add(\Filt A[>0])\perp \D^{\bb}_{>0}(A)=\D^{\bb}_{>0}(A^{\rm op})^{*}\perp \add(\Filt A[<0])\perp \CM A.\]
\end{enumerate}
\end{Thm}
\noindent
Immediately,  
$\CM A$ admits a property   analogous to the usual Cohen-Macaulay approximation (see \cite{AB}) in the following sense.

 \begin{Cor}\label{Prop:CMappro}
Let $M\in \D^{\bb}_{\le 0}(A)$, then there is a triangle 
\[ P\ra T\ra M \ra P[1], \]
such that $T\ra M$ is a right $(\CM A)$-approximation  of $M$ and $P\in \per A$.
\end{Cor}

To show the theorem, we consider the $t$-structures and co-$t$-structures on $\D^{\bb}(A)$ first.  Let 

\begin{eqnarray*} A_{\ge l}  =  A_{> l-1} &:=& \add\bigcup_{i\ge 0} A[-l-i]\ast \cdots \ast A[-l-1]\ast A[-l],\\
 A_{\le l}  =  A_{< l+1} &:=& \add\bigcup_{i\ge 0} A[-l]\ast A[-l+1]\ast \cdots \ast A[-l+i].
\end{eqnarray*}

There are two $t$-structures and two co-$t$-structures \cite{P} (also called weight structures \cite{Bo}) in $\D^{\bb}(A)$ induced by $A$.
\begin{Lem}\label{Lem:tcot}
 \begin{enumerate}[\rm (1)]
 \item The two paris $(A[<0]^{\perp}, A[>0]^{\perp})$ and $({}^{\perp}A[<0],{}^{\perp}A[>0])$ are $t$-structures on $\D^{\bb}(A)$;
 \item
  The two pairs $({}^{\perp}A[>0], A_{\le 0})$ and $(A_{\ge 0}, A[<0]^{\perp})$ are co-$t$-structures on $\D^{\bb}(A)$.
 \end{enumerate}
\end{Lem}

\begin{proof}
(1) By Proposition \ref{Prop:heart}, $(A[<0]^{\perp}, A[>0]^{\perp})=(\D_{\le 0}^{\bb}, \D_{\ge 0}^{\bb})$ is a  $t$-structure on $\D^{\bb}(A)$. Since the Nakayama functor $\nu$ induces a triangle equivalence $\nu: \D^{\rm b}(A)\simeq\D^{\rm b}(A)$, then applying $\nu^{-1}$ to this $t$-structure,  we get a new $t$-structure $({}^{\perp}A[<0],{}^{\perp}A[>0])$ on $\D^{\bb}(A)$.

(2) See \cite[Propsition 3.2]{IY2}.
\end{proof}

Now we show the theorem.
\begin{proof}[Proof of Theorem \ref{Thm:cmappro}]
We only show (2), since (1) is directly from (2).

By Lemma \ref{Lem:tcot}, we have $\D^{\bb}(A)=\D^{\bb}_{\le 0}\perp \D^{\bb}_{>0}$.
We claim that $\D^{\bb}_{\le 0}=\CM A \perp \Filt A[>0]$.
Let $M\in \D^{\bb}_{\le 0}=A[<0]^{\perp}$. Considering the co-$t$-structure $(^{\perp}A[>0], A_{\le 0})$, we have the following  decomposition of $M$. 
\[         T \ra M \ra S \ra T[1], \]
where $T\in {}^{\perp}A[>0]$ and $S\in A_{< 0}=\Filt A[>0]$. Applying $\Hom_{\D^{\bb}(A)}(A[<0], ?)$ to the triangle above, we have  $T\in \CM A$.
So the claim holds and  $\D^{\bb}(A)=\CM A \perp \Filt A[>0]\perp \D^{\bb}_{>0}(A)$.
By the duality $(\ )^{*}: \D^{\bb}(A^{\rm op}) \xra{\simeq} \D^{\bb}(A)$, we have $\D^{\bb}(A)=\D^{\bb}_{>0}(A^{\rm op})^{*}\perp \Filt A[<0]\perp \CM A$.
\end{proof}

We end this section by giving a result analogous to 
 famous results of Auslander and Yamagata \cite[Theorem VI.1.4]{ARS}\cite{Ya} on 
the first Brauer-Thrall theorem. For an object $X\in \D^{\rm b}(A)$, there is an integer $t$ such that $\h^{<t}(X)=0$ and $\h^{t}(X)\not=0$. By the standard truncation, we may assume $X^{i}=0$ for $i<t$ and in this case, we  have a natural inclusion $\h^{t}(X)\hookrightarrow X$. So we may regard $\soc \h^{t}(X)$ as the socle $\soc X$ of $X$.
\begin{Prop}\label{Prop:finiteall}
Let  $\cal S$ be a finite subset of $\ind (\CM A)$. Then $\cal S= \ind(\CM A)$  if  $\cal S$ is closed under successors in the AR quiver of $\CM A$ and  for any $i\ge 0$, there exists a  left $(\CM A)$-approximation  $A[i]\ra X$   in $\D^{\bb}(A)$ such that $X\in \add \cal S$. 
\end{Prop}

\begin{proof}
Notice that $A\xra{\rm{id}}A$ is the minimal left $(\CM A)$-approximation of $A$, then $A\in \add \cal S$ by our assumption. 
Let $M\in \ind(\CM A)$.
Then there exists $i\ge 0$ such that $\Hom_{\CM A}(A[i], M)\not=0$. Let $N$ be the left $(\CM A)$-approximation of $A[i]$ such that  $N\in\add\cal S$. Then $\Hom_{\CM A}(N, M)\not=0$. Let $X_{1}$ be an indecomposable direct summand of $N$ with $0\not=f\in\Hom_{\CM A}(X_{1}, M)$. If $f$ is a section, since $X_{1}, M\in \ind(\CM A)$, then $X_{1}\cong M$ and we are done. If $f$ is not a section, consider the left almost split morphism $g: X_{1}\ra Y$ (If $X_{1}\in \add A$, $g$ is given by $X_{1}\ra X_{1}/\soc X_{1}$). Then there is an $h\in \Hom_{\CM A}(Y, M)$ such that $f=h\circ g$.
So we can find an indecomposable module $X_{2}\in \add Y$, such that the composition $X_{1}\ra X_{2} \ra M$ is non-zero.
Repeat this step, we may construct a series of indecomposable modules
$ X_{1}\ra X_{2}\ra \cdots \ra M$,
such that the composition is non-zero. Since $\cal S$ is closed under successors, then $X_{i}\in \cal S$.  Since $\cal S$ is finite and $\CM A$ is Hom-finite, then $\rad (\cal S, \cal S)^{N}=0$ for big enough $N$. So there exist $n\ge1$ such that $X_{n}=M$ and $M\in \add \cal S$. Therefore $\cal S=\ind (\CM A)$.
\end{proof}

\section{Example: truncated polynomial dg algebras} \label{Section:poly}

In this section,  we give some examples. The reader may skip this section, since results here will not be used in this paper.
Consider a truncated polynomial  dg $k$-algebra.
 $$A:=k[X]/(X^{n+1}), n\ge 0,$$
with $\deg X=-d\le 0 $  and  zero differential. We determine the  indecomposable Cohen-Macaulay modules explicitly and draw the AR quiver of $\CM A$. Then we show $\un{\CM}A$ is  a $(d+1)$-cluster category by using a criterion given by Keller and Reiten \cite{Keller08}.
Let $A_{i}$ be the dg $A$-module $k[X]/(X^{i})$, $i=1, 2, \cdots, n$. We give two small examples first.
\begin{Ex}
 \begin{enumerate}[\rm (1)]
   \item Let $n=2$ and $d=2$. Then the AR quiver of $\CM A$ is as follows.
     {\tiny
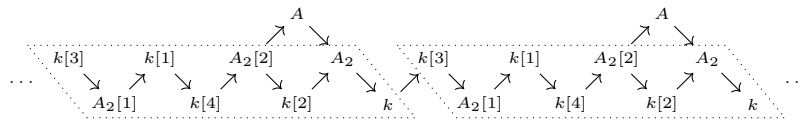
       \begin{center}
         \begin{tikzpicture}[scale=0.6]
         \draw
         node (kl) at (0,0) {$k$}
         node (k2l) at (-2,0) {$k[2]$}
         node (k4l) at (-4,0) {$k[4]$}
         node (A21l) at (-6,0) {$A_{2}[1]$}
         node (A21r) at (2,0) {$A_{2}[1]$}
         node (k4r) at (4,0) {$k[4]$}
         node (k2r) at (6,0) {$k[2]$}
         node (kr) at (8,0) {$k$}
         node (A2l) at (-1,1) {$A_{2}$}
         node (A22l) at (-3,1) {$A_{2}[2]$}
         node (k1l) at (-5,1) {$k[1]$}
         node (k3l) at (-7,1) {$k[3]$}
         node (k3r) at (1,1) {$k[3]$}
         node (k1r) at (3,1) {$k[1]$}
         node (A22r) at (5,1) {$A_{2}[2]$}
         node (A2r) at (7,1) {$A_{2}$}
         node (Ar) at (6,2) {$A$}
         node (Al) at (-2,2) {$A$}
         node at (-8, 0.5) {$\dots$}
         node at (9, 0.5) {$\dots$}
         [->] (k3l) edge (A21l) (A21l) edge (k1l) (k1l)              
         edge (k4l) (k4l) edge (A22l) (A22l) edge (k2l)
         (k2l) edge (A2l) (A2l) edge (kl) (kl) edge (k3r)
         (k3r) edge (A21r) (A21r) edge (k1r) (k1r)              
         edge (k4r) (k4r) edge (A22r) (A22r) edge (k2r)
         (k2r) edge (A2r) (A2r) edge (kr)
         (A22l) edge (Al) (Al) edge (A2l)
         (A22r) edge (Ar) (Ar) edge (A2r);
       \draw[dotted] (-7.9, 1.3)--(-0.7,1.3)--(0.6,-0.3)--(-6.6,-0.3)--(-7.9,1.3);  
              \draw[dotted] (0.2, 1.3)--(7.4,1.3)--(8.7,-0.3)--(1.5,-0.3)--(0.2,1.3);   
         \end{tikzpicture}
         \end{center}}
    \item Let $n=3$ and $d=1$. Then the AR quiver of $\CM A$ is as follows.
      {\tiny
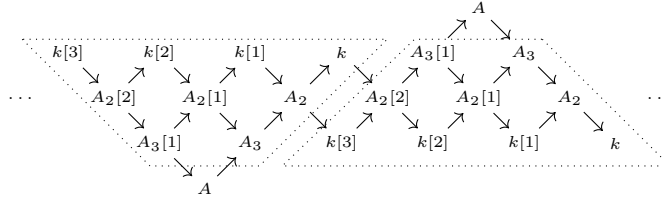
       \begin{center}
         \begin{tikzpicture}[scale=0.6]
         \draw
         node (k3r) at (0,0) {$k[3]$} 
         node (A22r) at (1,1) {$A_{2}[2]$}
         node (A31r) at (2,2) {$A_{3}[1]$}
         node (k2r) at (2,0) {$k[2]$}
         node (A21r) at (3,1) {$A_{2}[1]$}
         node (A3r) at (4,2) {$A_{3}$}
         node (k1r) at (4,0) {$k[1]$}
         node (A2r) at (5,1) {$A_{2}$}
         node (kr) at (6,0) {$k$}
         node (Ar) at (3,3) {$A$}
         node (A3l) at (-2,0) {$A_{3}$}
         node (A31l) at (-4,0) {$A_{3}[1]$}
         node (A2l) at (-1,1) {$A_{2}$}
         node (A21l) at (-3,1) {$A_{2}[1]$}
         node (A22l) at (-5,1) {$A_{2}[2]$}
         node (kl) at (0,2) {$k$}
         node (k1l) at (-2,2) {$k[1]$}
         node (k2l) at (-4,2) {$k[2]$}
         node (k3l) at (-6,2) {$k[3]$}
         node (Al) at (-3,-1) {$A$}
         node at (-7,1) {$\dots$} 
         node at (7,1) {$\dots$}
         [->] (k3l)edge(A22l) (A22l)edge(k2l)  
         (k2l)edge(A21l) (A21l)edge(k1l) 
         (k1l)edge(A2l) (A2l)edge(kl)
         (A22l)edge(A31l) (A31l)edge(A21l)
         (A21l)edge(A3l) (A3l)edge(A2l)
         (A31l)edge(Al) (Al)edge(A3l)
         (kl)edge(A22r) (A2l)edge(k3r)  
         (k3r)edge(A22r) (A22r)edge(k2r)  
         (k2r)edge(A21r) (A21r)edge(k1r) 
         (k1r)edge(A2r) (A2r)edge(kr)
         (A22r)edge(A31r) (A31r)edge(A21r)
         (A21r)edge(A3r) (A3r)edge(A2r)
         (A31r)edge(Ar) (Ar)edge(A3r);
        \draw[dotted] (-7,2.3)--(-4.2,-0.5)--(-1.8,-0.5)-- (1,2.3)--(-7,2.3)       (4.4,2.3)--(7.2,-0.5)--(-1.3, -0.5)--  (1.6,2.3)--(4.4,2.3);
        \end{tikzpicture}
       \end{center}}
  \end{enumerate}
\end{Ex}

By Proposition \ref{Prop:CMappro}, for any $A_{i}$, $1\le i \le n$, and $t\ge 0$, we have the following triangle.
\begin{eqnarray}\label{minicmappro} T_{i,t} \ra A_{i}[td] \ra P_{i,t} \ra T_{i,t}[1], \end{eqnarray}
such that  $T_{i,t}\ra A_{i}[td]$ is a  right $(\CM A)$-approximation of $A_{i}[td]$ and $P_{i,t}\in \per A$.  We assume $T_{i,t}$ is minimal. Then  $T_{i,t}$ is unique up to isomorphism and  if $A_{i}[td]\in \CM A$ (for example, $t=0\ {\rm or}\ 1$),  we have $T_{i,t}=A_{i}[td]$.    We give the first result of this section.

\begin{Thm}\label{Thm:ARquiver}
Let $A$  be the dg algebra $k[X]/(X^{n+1})$, $n\ge 0$ with $\deg X=-d\le 0 $  and  zero differential. 
\begin{enumerate}[\rm (1)]
\item Assume $d$ is even. Let $N:=\frac{(n+1)d}{2}$. Then the AR quiver of ${\CM} A$ is as follows.
      {\tiny
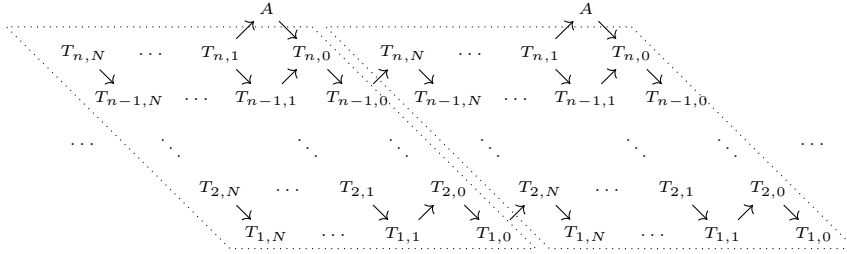
       \begin{center}
         \begin{tikzpicture}[scale=0.6]
         \draw 
         node (T10l) at (0,0) {$T_{1,0}$}
         node (T11l) at (-2,0) {$T_{1,1}$}
         node (T1Nl) at (-5,0) {$T_{1,N}$}
         node (T20l) at (-1,1) {$T_{2,0}$}
         node (T21l) at (-3,1) {$T_{2,1}$}
         node (T2Nl) at (-6,1) {$T_{2,N}$}
         node (Tn-1Nl) at (-8,3) {$T_{n-1,N}$}
         node (Tn-11l) at (-5,3) {$T_{n-1,1}$}
         node (Tn-10l) at (-3,3) {$T_{n-1,0}$}
         node (Tn0l) at (-4,4) {$T_{n,0}$}
         node (Tn1l) at (-6,4) {$T_{n,1}$}
         node (TnNl) at (-9,4) {$T_{n,N}$}         
         node (T10r) at (7,0) {$T_{1,0}$}
         node (T11r) at (5,0) {$T_{1,1}$}
         node (T1Nr) at (2,0) {$T_{1,N}$}
         node (T20r) at (6,1) {$T_{2,0}$}
         node (T21r) at (4,1) {$T_{2,1}$}
         node (T2Nr) at (1,1) {$T_{2,N}$}
         node (Tn-1Nr) at (-1,3) {$T_{n-1,N}$}
         node (Tn-11r) at (2,3) {$T_{n-1,1}$}
         node (Tn-10r) at (4,3) {$T_{n-1,0}$}
         node (Tn0r) at (3,4) {$T_{n,0}$}
         node (Tn1r) at (1,4) {$T_{n,1}$}
         node (TnNr) at (-2,4) {$T_{n,N}$}
         node at (-3.5,0) {$\dots$}
         node at (-4.5,1) {$\dots$}
         node at (-6.5,3) {$\dots$}
         node at (-7.5,4) {$\dots$}
         node at (3.5,0) {$\dots$}
         node at (2.5,1) {$\dots$}
         node at (0.5,3) {$\dots$}
         node at (-0.5,4) {$\dots$}
         node at (-7.1,2.1) {$\ddots$}
         node at (-4.1,2.1) {$\ddots$}
         node at (-2.1,2.1) {$\ddots$}
         node at (-0.1,2.1) {$\ddots$}
         node at (3.1,2.1) {$\ddots$}
         node at (5.1,2.1) {$\ddots$}
         node (Al) at (-5, 5) {$A$}
         node (Ar) at (2,5) {$A$}
         node at (-9,2) {$\dots$}
         node at (7,2) {$\dots$}
         [->] 
         (TnNl)edge(Tn-1Nl) (T2Nl)edge(T1Nl)
         (Tn1l)edge(Tn-11l)  (Tn-11l)edge(Tn0l)
         (Tn0l)edge(Tn-10l) (T21l)edge(T11l) 
         (T11l)edge(T20l) (T20l)edge(T10l) 
         (TnNr)edge(Tn-1Nr) (T2Nr)edge(T1Nr)
         (Tn1r)edge(Tn-11r)  (Tn-11r)edge(Tn0r)
         (Tn0r)edge(Tn-10r) (T21r)edge(T11r) 
         (T11r)edge(T20r) (T20r)edge(T10r)
         (Tn-10l)edge(TnNr) (T10l)edge(T2Nr)
         (Tn1l)edge(Al) (Al)edge(Tn0l)
         (Tn1r)edge(Ar) (Ar)edge(Tn0r);
        \draw[dotted] 
        (-4,4.6)--(0.9,-0.3)--(-5.8,-0.3)
         --(-10.7,4.6)--(-4,4.6)
        (3,4.6)--(7.9,-0.3)--(1.2,-0.3)
        --(-3.7,4.6)--(3,4.6); 
            \end{tikzpicture}
      \end{center}}
  \item Assume $d$ is odd. Let $N_{i}:=\frac{(n+1)d+n-2i+1}{2}$, $1\le i \le n$. Then the AR quiver of ${\CM} A$ is as follows.
      {\tiny
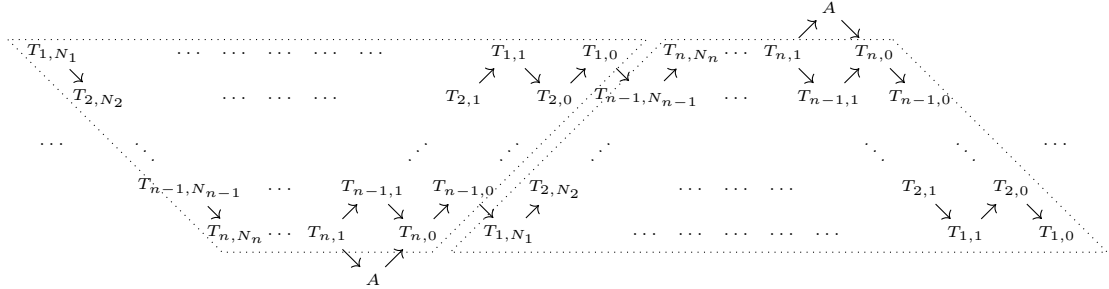
       \begin{center}
         \begin{tikzpicture}[scale=0.6]
         \draw 
          node (Tn0l) at (0,0) {$T_{n,0}$}
          node (Tn-10l) at (1,1) {$T_{n-1,0}$}
          node (T20l) at (3,3) {$T_{2,0}$}
          node (T10l) at (4,4) {$T_{1,0}$}
          node (T11l) at (2,4) {$T_{1,1}$}
          node (T21l) at (1,3) {$T_{2,1}$}
          node (Tn-11l) at (-1,1) {$T_{n-1,1}$}
          node (Tn1l) at (-2,0) {$T_{n,1}$}
          node (TnNnl) at (-4,0) {$T_{n,N_{n}}$}
          node (Tn-1Nn-1l) at (-5,1) {$T_{n-1,N_{n-1}}$}
          node (T2N2l) at (-7,3) {$T_{2,N_{2}}$}
          node (T1N1l) at (-8,4) {$T_{1,N_{1}}$}
          node at (-6,2) {$\ddots$}
          node at (0,2) {\reflectbox{$\ddots$}}
          node at (2,2) {\reflectbox{$\ddots$}}
          node at (4,2) {\reflectbox{$\ddots$}}
          node at (10,2) {$\ddots$}
          node at (12,2) {$\ddots$}
          node at (-3,0) {$\dots$}
          node at (-3,1) {$\dots$}
          node at (-2,3) {$\dots$}
          node at (-3,3) {$\dots$}
          node at (-4,3) {$\dots$}
          node at (-1,4) {$\dots$}
          node at (-2,4) {$\dots$}
          node at (-3,4) {$\dots$}
          node at (-4,4) {$\dots$}
          node at (-5,4) {$\dots$}            
          node (Al) at (-1,-1) {$A$}
          node at (-8,2) {$\dots$}
          node at (14,2) {$\dots$}
          node (Tn0r) at (10,4) {$T_{n,0}$}
          node (Tn-10r) at (11,3) {$T_{n-1,0}$}
          node (T20r) at (13,1) {$T_{2,0}$}
          node (T10r) at (14,0) {$T_{1,0}$}
          node (T11r) at (12,0) {$T_{1,1}$}
          node (T21r) at (11,1) {$T_{2,1}$}
          node (Tn-11r) at (9,3) {$T_{n-1,1}$}
          node (Tn1r) at (8,4) {$T_{n,1}$}
          node (TnNnr) at (6,4) {$T_{n,N_{n}}$}
          node (Tn-1Nn-1r) at (5,3) {$T_{n-1,N_{n-1}}$}
          node (T2N2r) at (3,1) {$T_{2,N_{2}}$}
          node (T1N1r) at (2,0) {$T_{1,N_{1}}$}
          node at (7,4) {$\dots$}
          node at (7,3) {$\dots$}
          node at (6,1) {$\dots$}
          node at (7,1) {$\dots$}
          node at (8,1) {$\dots$}
          node at (5,0) {$\dots$}
          node at (6,0) {$\dots$}
          node at (7,0) {$\dots$}
          node at (8,0) {$\dots$}
          node at (9,0) {$\dots$}            
          node (Ar) at (9,5) {$A$}
          [->]
          (T1N1l)edge(T2N2l) (Tn-1Nn-1l)edge(TnNnl)
          (Tn1l)edge(Tn-11l) (Tn-11l)edge(Tn0l)   
          (Tn0l)edge(Tn-10l) (T21l)edge(T11l)
          (T11l)edge(T20l) (T20l)edge(T10l)
          (T1N1r)edge(T2N2r) (Tn-1Nn-1r)edge(TnNnr)
          (Tn1r)edge(Tn-11r) (Tn-11r)edge(Tn0r)   
          (Tn0r)edge(Tn-10r) (T21r)edge(T11r)
          (T11r)edge(T20r) (T20r)edge(T10r)
          (T10l)edge(Tn-1Nn-1r) (Tn-10l)edge(T1N1r)
          (Tn1r)edge(Ar) (Ar)edge(Tn0r)
          (Tn1l)edge(Al) (Al)edge(Tn0l);
        \draw[dotted]
          (-9,4.3)--(-4.3,-0.4)--(0.3,-0.4)--(5,4.3)--(-9,4.3)
          (10.4,4.3)--(15.1,-0.4)--(0.7,-0.4)--(5.4,4.3)--
          (10.4,4.3);
         \end{tikzpicture}
      \end{center}}
\end{enumerate}
\end{Thm}

Before proving Theorem \ref{Thm:ARquiver}, we consider AR triangles in $\un{\CM}A$ first.
It is easy to see that $A$  is an $(nd+1)$-symmetric dg algebra. Then by Proposition \ref{Prop:cm}, $\CM A=\D^{\bb}_{\le 0}\cap \D^{\bb}_{\ge -nd}$. The  Nakayama functor $\nu: \D^{\bb}(A)\ra \D^{\bb}(A)$ is given by $\nu=?\ot_{A}^{\bf L}DA=[-nd]$.
By Theorem \ref{Thm:AR}, $\un{\CM }A$ admits a Serre functor $\nu[-1]=[-nd-1]$. Moreover, the Auslander-Reiten translation on $\un{\CM}A$ is $\tau=[-nd-2]$. 
The following lemma shows $A_{i}$ and $T_{i,t}$ are indecomposable.
\begin{Lem}\label{indec}
  Let $A_{i}, T_{i,t}$  be defined as above. Then  
 \begin{enumerate}[\rm (1)]
\item ${\rm End}_{\un{\CM }A}(A_{i})={\rm End}_{\CM A}(A_{i})=k$. Moreover, each $A_{i}$ is  indecomposable in   $\un{\CM }A$;
 \item $T_{i,t}$  is indecomposable in $\CM A$.
 \end{enumerate}
\end{Lem}	

\begin{proof}
(1) For any $A_{i}$, $1\le i\le n$, there is a natural triangle in $\D^{\bb}(A)$.
\begin{eqnarray}\label{naturtri} A_{n+1-i}[id] \ra A\ra A_{i} \ra A_{n+1-i}[id+1].\end{eqnarray}
Since $A_{n+1-i}[id]\in \D_{\le -id}$ and $A_{i}\in \D_{\ge -(i-1)d}$, then $\Hom_{\D^{\bb}(A)}(A_{n+1-i}[\ge id], A_{i})=0$. Applying  $\Hom_{\D^{\bb}(A)}(?, A_{i})$  to triangle $\eqref{naturtri}$, we have 
\[ \End_{\D^{\bb}(A)}(A_{i})\cong  \Hom_{\D^{\bb}(A)}(A, A_{i})=k.\]
So $A_{i}$ is indecomposable in $\CM A$. Since $A$ itself is indecomposable in $\CM A$ by $\End_{\CM A}(A)=k$, and $A_{i}\not=A$ by cohomology. Then $A_{i}\not\in \add A$, and $A_{i}$ is a non-zero object in $\CM A$.

(2) It is clear $T_{i,t}\cong A_{i}[td]$ in $\un{\CM} A$. So $T_{i,t}$ is indecomposable in $\un{\CM} A$ by (1). Because  if $t=0$,  $T_{i,0}=A_{i}$ and if $t>0$,  $\Hom_{\CM A}(A, A_{i}[td])=0$, then  $T_{i,t}$ does not contain $P\in\add A$ as a direct summand.  Thus $T_{i,t}$ is also indecomposable in $\CM A$.
\end{proof}

We  point out the periodicity  of $\un{\CM} A$.
\begin{Lem}\label{id}
The functor $[(n+1)d+2]: \un{\CM} A\ra\un{\CM} A$ is isomorphic to the identity functor. In particular, $\tau\cong [d]$ as functors on $\un{\CM}A$ and $\un{\CM} A$ is $(d+1)$-Calabi-Yau.
\end{Lem}

\begin{proof}
Consider the following   sequence in the category of dg $A\ot A^{\rm op}$-modules.
\[  0\ra A[(n+1)d] \xra{f} A\ot A[d] \xra{g} A\ot A \xra{h} A \ra 0 \]
where $f$ is given by $f(1):=\sum_{i=0}^{n}X^{i}\ot X^{n-i}$, $g$ is given by $g(1\ot 1):= 1\ot X -X\ot 1$ and $h$ is given by $h(1\ot 1):=1$. Since it is an exact sequence of  graded modules, we get two natural triangles in $\D(A\ot A^{\rm op})$.
\[ \ke h \xra{h}  A\ot A \ra A \ra  \ke h[1] \]
\[ A[(n+1)d] \xra{f} A\ot A[d] \xra{} \ke h \ra A[(n+1)d+1]\]
Let $M\in \D^{\bb}(A)$. Apply the functor $M\ot_{A}^{\bf L}?$ to the triangles above, we get two triangles in $\D^{\bb}(A)$:
\[ M \ot_{A}^{\bf L}\ke h\ra M\ot A \ra M \ra  M \ot_{A}^{\bf L}\ke h[1],\]
\[ M[(n+1)d] \ra M\ot A[d] \ra M \ot_{A}^{\bf L}\ke h\ra M[(n+1)d+1]. \]
Notice that $M\ot A\in\per A$, then we have natural isomorphisms $M \xra{\sim}  M \ot_{A}^{\bf L}\ke h[1]$ and $M \ot_{A}^{\bf L}\ke h \xra{\sim}   M[(n+1)d+1]$  in $\D^{\bb}(A)/\per A\cong \un{\CM}A$, which give us the desired isomorphism. 
\end{proof}

\begin{Rem}\label{Rem:smallest}
The integer $(n+1)d+2$ is the smallest natural number $r$ such that $[r]\simeq {\rm id}$ on $\un{\CM}A$. In fact one may consider the shifts of simple dg $A$-module $A_{1}=k$. It is clear that $k[i]\in \CM A$ for $0\le i\le nd$. By the triangle \eqref{naturtri}, we have $k[nd+1]\cong A_{n}$ in $\un{\CM}A$ (notice that $A_{n}[0,1,\dots,d]\in \CM A$). Then it is easy to check that $k$ can not be isomorphic to $k[i]$ in $\un{\CM}A$ for $0\le i\le (n+1)d+1$.
\end{Rem}
Now we  describe the AR-triangles in $\un{\CM} A$. 
\begin{Prop} \label{Prop:ARquiver}
Let $1\le i\le n$.
Let $\pi_{i}$ be the natural surjective map $\pi_{i}: A_{i}\ra A_{i-1}$ and let $\iota_{i}$ be the natural injective map $\iota_{i}: A_{i}[d]\ra A_{i+1}$ given by $\iota_{i}(1):= X$. Let $A_{0}=0$. Then the AR triangle in $\un{\CM} A$ ending in $A_{i}$ is given by 
\[ A_{i}[d]\xra{\big(\begin{smallmatrix} 
\pi_{i}[d] \\ -\iota_{i}[d] 
\end{smallmatrix}\big)} 
 A_{i-1}[d] \op A_{i+1} \xra{(\begin{smallmatrix} 
\iota_{i-1} & \pi_{i+1} 
\end{smallmatrix})} A_{i}. \]

\end{Prop}

\begin{proof}
First notice that the given sequence is a short exact sequence of graded modules and it gives a triangle in $\D^{\rm b}(A)$ (so in $\un{\CM} A$).
By Lemma \ref{indec}, $A_{j}[m]$ is indecomposable in $\un{\CM }A$ for any $m\in \Z$, so the given triangle can not be split. 
Since ${\rm End}_{\un{\CM}A}(A_{i})=k$, then  by Lemma \ref{nonsplit}, the triangle above is an AR triangle. 
 \end{proof}

It is easy to see that the AR quiver of $\un{\CM} A$ is of the form $\Z A_{n}/\phi$ by Proposition \ref{Prop:ARquiver} and the fact that $T_{i,t}$ are shifts of $A_{i}$ in $\un{\CM} A$. Now we  determine   the fundamental domain. Notice that by Lemma \ref{id}, $A_{i}=A_{i}[(n+1)d+2]$ and by the triangle \eqref{naturtri}, $A_{n+1-i}=A_{i}[-id-1]$ in $\un{\CM}A$. We need to find the smallest positive integer $m$ such that $A_{i}=A_{i}[md]$ or $A_{n+1-i}=A_{i}[md]$ holds.

\begin{Lem} \label{evenodd}
 \begin{enumerate}[\rm (1)]
  \item Assume $d$ is even. Let $N:=\frac{(n+1)d}{2}$. Then $N$ is the smallest positive integer such that $A_{i}=A_{i}[(N+1)d]$;
  \item Assume $d$ is odd. Let $N_{i}:=\frac{(n+1)d+n-2i+1}{2}$. Then $N_{i}$ is the smallest positive integer such that $A_{n+1-i}=A_{i}[(N_{i}+1)d]$.
 \end{enumerate}
\end{Lem}

\begin{proof}
 (1) It is obvious $A_{i}=A_{i}[(N+1)d]$ by Lemma \ref{id}.  By Remark \ref{Rem:smallest}, we know $(n+1)d+2$ is the smallest natural number $r$ such that $[r]\simeq {\rm id}$ on $\un{\CM}A$.  Let $d=2e$. If $l> 0$ satisfies $A_{i}=A_{i}[ld]$, then $(n+1)d+2 \mid ld$, that is $(n+1)e+1\mid le$.
 Since $(n+1)e+1$ and $e$ are coprime, then $(n+1)e+1\mid l$ and $l\ge (n+1)e+1=N+1$.
 
 (2) Assume positive integer $s$ satisfies $A_{n+1-i}=A_{i}[sd]$. Then by the fact that $A_{n+1-i}=A_{i}[-id-1]$, we have $(n+1)d+2 \mid sd+id+1$.  
 Since $sd+id+1=\frac{d+1}{2}((n+1)d+2)+ (s- \frac{(n+1)d+n+3-2i}{2})$, then we need $(n+1)d+2\mid s-\frac{(n+1)d+n+3-2i}{2}$. So the smallest $s$ is $N_{i}+1$.
\end{proof}

We can prove Theorem \ref{Thm:ARquiver} now.

\begin{proof}[Proof of Theorem \ref{Thm:ARquiver}]
The AR triangle given in Proposition \ref{Prop:ARquiver} 
 is induced by some conflation 
\begin{eqnarray}\label{conflation}
 T_{i,1} \ra T_{i-1,1}\op T_{i+1,0} \ra T_{i,0}\end{eqnarray}
up to taking projective direct sums in $\CM A$ (see \cite{Palu}). Notice that for the projective-injective object $A$, the only  right almost split morphism  is given by the natural injection $T_{n,1}=A_{n}[d]\ra A$ and the only  left almost split morphism is given by the natural surjection $A \ra A_{n}=T_{n,0}$. Then  the extension \eqref{conflation}  is an almost split extension in $\CM A$ for $i\not= n$.  Then by Lemma \ref{evenodd}, the AR sub-quiver $\cal{S}$ of $\CM A$ consisting of $T_{i,d}$ is given as in Theorem \ref{Thm:ARquiver}. We only need to show $T_{i,d}$ gives all indecomposable CM $A$-modules.

By Proposition \ref{Prop:finiteall}, it suffices to show that every left $(\CM A)$-approximation of  $A[p]$, $p\ge 0$, belongs to $\add \cal S$. First notice that if $p>nd$ or $p=0$, the assertion is obvious (because for the case  $p>nd$, we have a natural approximation $A[p]\ra 0$ and for the case $p=0$, the approximation is given by $A\xra{\rm id}A$). We only consider other cases and we may assume $td<p\le (t+1)d$ for some $0\le t\le n-1$.
Then the  left $(\CM A)$-approximation of $A[p]$ is given by the natural map   $A[p]\ra A_{n-t}[p]$.
To show $A_{n-t}[p]\in\add \cal S$, it is enough to show $A_{n-t}[p]=T_{i,s}=A_{i}[sd]$ in $\un{\CM}A$ for some $1\le i\le n$ and $s\ge 0$.
It can be shown in the following way.
\begin{enumerate}[\rm (1)]
\item[$\bullet$] Assume $(t+1)d-p$ is even and $q=((t+1)d-p)/2$. Then $A_{n-t}[p]=A_{n-t}[((n+1)d+2)q+p]=A_{n-t}[sd]$, where $s=(n+1)q+t+1$. So $A_{n-t}[p]=T_{n-t, s}\in \add \cal S$.
\item[$\bullet$] Assume $(t+1)d-p$ is odd and $q=((t+1)d-p-1)/2$. By triangle \eqref{naturtri}, we have $A_{n-t}[p]=A_{t+1}[(n-t)d+p+1]$ in $\un{\CM}A$.
Then 
\[ A_{n-t}[p]=A_{n-t}[((n+1)d+2)q+p]=A_{t+1}[sd],\]
where $s=(n+1)(q+1)$. So $A_{n-t}[p]=T_{t+1,s}\in \add \cal S$. 
\end{enumerate}
Then the result is true.
\end{proof}

Before giving the second result of this section, we give another description of $T_{i,t}$, which will be used later.

\begin{Prop}\label{Prop:gap}
For any $1\le i\le n$ and $t\ge 0$, there exist $1\le j\le n$ and $0\le s\le (n+1-j)d$ such that $T_{i,t}=A_{j}[s]$ in $\CM A$.
\end{Prop} 
\begin{proof}
First notice that $T_{i, 0}=A_{i}$ and $T_{i,t}$ can be defined by applying  induction on $t$ by the following triangle in $\D^{\rm b}(A)$.
\[ T_{i,t+1} \xra{f_{t}} T_{i, t}[d]\ra P_{i,t} \ra T_{i,t+1}[1],\]
where $f_{t}$ is the minimal right $(\CM A)$-approximation of $T_{i,t}[d]$ and $P_{i,t}\in \per A$.
Also notice that for any $A_{j}[s]\in \CM A$, the minimal right $(\CM A)$-approximation of $A_{j}[s+d]$ also has the form of $A_{j'}[s']$ for some $1\le j'\le n$ and $0\le s'\le (n+1-j')d$. Then the assertion can be shown inductively.
\end{proof}

Theorem \ref{Thm:ARquiver} implies that $\un{\CM}A$ is a cluster category. In fact we have the following result.

\begin{Thm}\label{main}
The stable category  $\un{\CM} A$ is triangle equivalent to  $\cal C_{d+1}(A_{n})$.\end{Thm}

The key ingredient of the proof is  Keller and Reiten's result \cite{Keller08}.
We first show $\un{\CM} A$ admits a $(d+1)$-cluster tilting object.
\begin{Prop}\label{cluster}
Let $T:=\bop_{i=1}^{n}A_{i}$, then $T$ is a $(d+1)$-cluster-tilting object in $\un{\CM}A$, that is, $\add T$ is functorially finite in $\un{\CM} A$ and $X\in \add T$ if and only if $\Hom_{\un{\CM}A}(T, X[m])=0$ for all $1\le m\le d$.
\end{Prop}

We show the following lemma first.
\begin{Lem}\label{11}
\begin{enumerate}[\rm (1)]
\item $T$ is a $(d+1)$-rigid object in $\un{\CM}A$, $i.e.$  $\Hom_{\un{\CM}A}(A_{i}, A_{j}[s])=0$ for any $1 \le i, j \le n$ and $1\le s \le d$;

\item  $\Hom_{\un{\CM}A}(A_{i}, A_{j}[-s])=0$ for any $1 \le i, j \le n$ and $1\le s \le d-1$;

\item Let $M\in \CM A$. Assume  $\Hom_{\un{\CM}A}(A_{i}, M[m])=0$ for any $1 \le i \le n$ and $1\le m \le d$. If $\h^{0}(M)=0$, then $M= 0$.
\end{enumerate}
\end{Lem}

\begin{proof}
By Proposition \ref{Prop:Groper}, we have $\Hom_{\un{\CM}A}(A_{i}, A_{j}[s])=\Hom_{\D^{\bb}(A)}(A_{i}, A_{j}[s])$.
Consider the following two triangles.
\begin{equation}\label{5.1}
 A_{n-i+1}[id]\lra A \lra A_{i} \lra A_{n-i+1}[id+1] \lra A[1], \end{equation}
\begin{equation}\label{5.2}
 A_{i}[(n-i+1)d]\lra A \lra A_{n-i+1} \lra A_{i}[(n-i+1)d+1]  \lra A[1].  \end{equation}
If $s>1$, by applying the functor $\Hom_{\D^{\bb}(A)}(? , A_{j}[s])$ to triangle \eqref{5.1}, we have 
\[ \Hom_{\D^{\bb}(A)}(A_{i}, A_{j}[s])=\Hom_{\D^{\bb}(A)}(A_{n-i+1}[id+1], A_{j}[s]).\]
Apply the functor $\Hom_{\D^{\bb}(A)}(? , A_{j}[s-id-1])$ to the triangle \eqref{5.2}. Since 
\[ \Hom_{\D^{\bb}(A)}(A_{i}[(n-i+1)d+1], A_{j}[s-id-1])=0, \]
then
\[ \Hom_{\D^{\bb}(A)}(A_{n-i+1}, A_{j}[s-id-1])=0.\]
Thus $\Hom_{\D^{\bb}(A)}(A_{i}, A_{j}[s])=0$.

If $s=1$. Apply the functor $\Hom_{\D^{\bb}(A)}(? , A_{j}[1])$ to triangle \eqref{5.1}. Notice that the induced map 
\[ \Hom_{\D^{\bb}(A)}(A[1], A_{j}[1]) \ra \Hom_{\D^{\bb}(A)}(A_{n-i+1}[id+1], A_{j}[1])\]
is surjective. Then $\Hom_{\D^{\bb}(A)}(A_{i}, A_{j}[1])=0$. So (1) is true. The proof of statement (2) is similar to (1).

For (3), let $M\in \CM A$. If $M\not=0$ in $\un{\CM} A$, let $t:={\rm min}\{ s\in \Z \mid \h^{s}(M)\not=0\}$.  We may assume $-id\le t <-(i-1)d$. We will show $\Hom_{\un{\CM}A}(A_{i}, M[t+id+1])\not=0$, which is a contradiction.

Since $\h^{\ge 0}M=0$, then $\h^{0}M[id+t]=0$. Apply the functor $\Hom_{\D^{\bb}(A)}(? , M[id+t+1])$ to \eqref{5.1}, we have 
\[ \Hom_{\D^{\bb}(A)}(A_{i}, M[id+t+1])= \Hom_{\D^{\bb}(A)}(A_{n-i+1}[id+1], M[id+t+1]). \]
Apply the functor $\Hom_{\D^{\bb}(A)}(? , M[t])$ to \eqref{5.2}, then 
\[  \Hom_{\D^{\bb}(A)}(A_{n-i+1}, M[t])= \Hom_{\D^{\bb}(A)}(A, M[t])=\h^{0}(M[t]).\]
Then $\Hom_{\un{\CM}A}(A_{i}, M[id+t+1])=\Hom_{\D^{\bb}(A)}(A_{i}, M[id+t+1])=\h^{0}(M[t])\not=0$. It is contradictory to our assumption. So $M=0 \in \un{\CM}A$. Since $M\not\in \add A$ by  $\h^{0}(M)=0$, then $M=0$.
\end{proof}

\begin{proof}[Proof of Proposition \ref{cluster}]
Since $\un{\CM} A$ is Hom-finite, then $\add T$ is a functorially finite subcategory of $\un{\CM} A$. Let $0\not=M\in \ind \CM A$. 
Then $M=A$ or $M=A_{j}[s]$ for some $1\le j\le n$ and $0\le s\le (n+1-j)d$ by Theorem \ref{Thm:ARquiver} and Proposition \ref{Prop:gap}.
Assume  $\Hom_{\un{\CM}A}(A_{i}, M[m])=0$ for any $1 \le i \le n$ and $1\le m \le d$. Then by Lemma \ref{11}, we have $\h^{0}(M)\not=0$. 
So $M\in \add T$. Then $T$ is a $(d+1)$-cluster-tilting object in $\un{\CM}A$.
\end{proof}

Now we are ready to prove Theorem \ref{main}.
\begin{proof}[Proof of Theorem \ref{main}]
 Since $\dim \Hom_{\un{\CM}A}(A_{i}, A_{j})= 0$ for $i>j$ and  $\dim \Hom_{\un{\CM}A}(A_{i}, A_{j})= 1$ for $i\le j$, the endomorphism algebra ${\rm End}_{\un{\CM}A}(T) $ is isomorphic to a path algebra $kQ$, where $Q$ is the quiver of type $A_{n}$ with linear orientation.
By Proposition \ref{cluster}, $T$ is a $(d+1)$-cluster-tilting object in $\un{\CM}A$. Moreover, by Lemma \ref{11},   $\Hom_{\un{\CM}A}(A_{i}, A_{j}[-k])=0$ for any $1 \le i, j \le n$ and $1\le k \le d-1$. Then by \cite[Theorem 4.2]{Keller08}, there is a triangle equivalence $\un{\CM }A \xra{\sim} \cal C_{d+1}(A_{n})$.
\end{proof}

\section{Negative Calabi-Yau configurations and combinatorial configurations}\label{Section:confi}

\subsection{Negative Calabi-Yau configurations}\label{Section:CYconfiguration}
In this subsection, we introduce negative Calabi-Yau configurations in the categorical framework. 
\begin{Def}\label{Def:configuration}
Let $\T$ be a $k$-linear Hom-finite Krull-Schmidt triangulated category and let $C$ be a  set of indecomposable objects of $\T$. For $d\ge 1$, we call $C$ a \emph{$(-d)$-Calabi-Yau configuration} (or \emph{$(-d)$-CY configuration}) if the following conditions hold.
 \begin{enumerate}[\rm(1)]
\item $\dim_{k} \Hom_{\T}(X, Y)=\delta_{X,Y}$ for $X,Y \in C$;
\item $\Hom_{\T}(X, Y[-j])=0$ for any two objects $X, Y$ in $ C$ and $0< j \le d-1$;

\item For any indecomposable object $M$ in $
\T$, there exists $X\in C$ and $0\le j \le d-1$, such that $\Hom_{\T}(X, M[-j])\not=0$.
\end{enumerate}
\end{Def}

If $\T$ admits a Serre functor $\SSS$, then by Serre duality, $(3)$ is equivalent to the following condition.
\begin{enumerate}[\rm (1)]
 \item[$(3^{\rm op})$]   For any indecomposable object $N$ in $\T$, there exists $X\in C$ and $0\le j \le d-1$, such that $\Hom_{\T}(N, X[-j])\not=0$.
\end{enumerate}

\begin{Def}\label{Def:SMS}
As in Definition \ref{Def:configuration}, 
we call $C$ a \emph{pre-$d$-simple-minded system} (or \emph{pre-$d$-SMS}), if it only satisfies (1), (2) above, and call $C$ a \emph{$d$-simple-minded system} (or \emph{$d$-SMS}) if it satisfies (1), (2) above and

\noindent
(3$'$) $\T=\add \Filt\{C, C[1], \cdots, C[d-1]\}$.
\end{Def}

It is easy to see that ``$d$-SMS'' implies ``$(-d)$-CY configuration''.
We show that if $\T$ admits a Serre functor $\SSS$, then any $(-d)$-CY configuration in $\T$ is preserved by the functor $\SSS[d]$. This property motivates the name ``$(-d)$-CY configuration''.
\begin{Thm}\label{Thm:closed}
Let $\T$ be a $k$-linear Hom-finite Krull-Schmidt triangulated category with Serre functor $\SSS$.
Let $C$ be a $(-d)$-CY configuration in $\T$, then $\SSS C[d]= C$.
\end{Thm}

\begin{Rem}\label{Rem:closed}
For the case $d=1$, Riedtmann showed the periodicity of configurations for type $A_{n}$ and $D_{n}$ in combinatorial setting (see \cite{Riedtmann, Riedtmann4}).
\end{Rem}

To prove the theorem above, we need the following well-known property.

\begin{Lem}\label{Lem:section}
Let $\T$ be a   $k$-linear Hom-finite triangulated category with Serre functor $\SSS$. Let $X\in \T$ with  $\End_{\T}(X)=k$ and $f\in \Hom_{\T}(X, \SSS X)$. Then for any $Y\in \T$ and $g\in \Hom_{\T}(\SSS X,Y)$ which is not a section, we have $g\circ f=0$.
 \end{Lem}
 
 \begin{proof}
 We have the following commutative diagram.
 \[ \xymatrixcolsep{6pc}\xymatrix{ \Hom_{\T}(X, \SSS X) \ar[r]^{\Hom_{\T}(X, g)} \ar[d]^{\simeq}& \Hom_{\T}(X, Y) \ar[d]^{\simeq} \\
 D\Hom_{\T}(\SSS X, \SSS X)  & D\Hom_{\T}(Y, \SSS X) \ar[l]_{D\Hom_{\T}(g, \SSS X)}
 }.\]
 Since the lower map is zero by $\End_{\T}(\SSS X)=k$ and the assumption that $g$ is not a section, so is the upper one.
 \end{proof}
  
\begin{proof}[Proof of Theorem \ref{Thm:closed}]
The proof falls into two parts. 

(a) We first prove $\SSS[d]C\subset C$.
For any $X\in C$, we only need to show $\SSS[d]X\in C$. By condition (3), there exist $Y\in C $ and $0\le i\le d-1$ such that $\Hom_{\T}(Y, \SSS[d]X[-i])\not =0$. 
Since  
\[ \Hom_{\T}(Y, \SSS[d]X[-i])=D\Hom_{\T}(X, Y[-d+i])\]
 If $0<i\le d-1$, it is zero by condition (2). So we must have $i=0$.  Let $f: Y\ra \SSS[d]X$ be a non-zero morphism and consider the triangle extended by $f$,
 \[\SSS X[d-1] \xra{h} N \xra{g} Y \xra{f} \SSS X[d].\]
  We claim that $N=0$.

If $N\not=0$, then there exist $Z\in C$ and $0 \le j \le d-1$, such that $\Hom_{\T}(Z, N[-j])\not =0$. Let $p\in \Hom_{\T}(Z[j], N)$ be a non-zero morphism. 
  If $g\circ p\not=0$, then  $j=0$ and $g\circ p$ is an isomorphism by Definition \ref{Def:configuration}(1)(2). Thus $g$ is a retraction and $f=0$,  a contradiction.
   So $g\circ p=0$. Then there exists a morphism $q: Z[j]\ra \SSS X[d-1]$, such that $p=h\circ q$. 
    \[ \xymatrix{   & Z[j] \ar[d]^{p} \ar[dl]_{q} & & \\
 \SSS X[d-1] \ar[r]^{h} & N \ar[r]^{g}  & Y \ar[r]^<<<<{f} & \SSS X[d]   }\]

Since $p\not =0$, then $q\not =0$, which implies that $j=d-1$ and $Z\cong X$ by the fact  $\Hom_{\T}(Z[j],\SSS X[d-1])=D\Hom_{\T}(X, Z[j-d+1])$ and Definition \ref{Def:configuration}(1)(2). 
 Then by Lemma \ref{Lem:section}, we know $h$ is a section. Thus $f=0$, a contradiction. So $N=0$ and $\SSS X[d]\cong Y\in C$. 

(b) We prove $\SSS[d]C\supset C$. By considering conditions (1), (2), and (3$^{\rm op}$), one can show the statement easily, which is similar to the proof in part (a). We leave it to the reader. 
\end{proof}

\subsection{CM dg modules and CY configurations}

In this subsection, we study CY configurations in  the stable categories of Cohen-Macaulay dg modules over  $d$-self-injective dg algebras. In this case, we show that the set of simple dg  $A$-modules forms a $(-d)$-CY configuration in $\un{\CM}A$, which generalizes Riedtmann's result \cite[Proposition 2.4]{Riedtmann}.

Recall from Section \ref{Section:nonpositivedg},  for a non-positive dg $k$-algebra $A$ with $A^{>0}=0$, we  may regard $\h^{0}(A)$-modules as dg $A$-modules via the homomorphism $A\ra \h^{0}(A)$.  Let $\{ S_{1}, \dots, S_{r} \}$ be the set of simple $\h^{0}(A)$-modules. We also regard them as simple dg $A$-modules (when we talk about simple modules, we always assume they are concentrated in degree zero part). Recall  that if $A$ is a $d$-self-injective dg algebra, then $\CM A= \D_{\le 0}^{\bb}\cap \D_{\ge -d+1}^{\bb}$ (see Proposition \ref{Prop:cm}).
 
The main result in this subsection is the following.
\begin{Thm}\label{Thm:configuration}
Let $A$ be a  $d$-self-injective dg $k$-algebra with $d\ge 0$. Then
the set of simple modules $\{  S_{i}\mid 1 \le i \le r \}$ is a $d$-SMS of $\un{\CM} A$, and hence a  $(-d)$-CY configuration of $\un{\CM}A$. 
\end{Thm}

To prove this theorem, we start with the following lemma first.
\begin{Lem}\label{12}
Let $M\in \CM A$. If $d>1$, then for $1\le i \le r $ and $0\le t \le d-2$, we have
\begin{enumerate}[\rm(1)]
\item $\Hom_{\CM A}(S_{i}[t], A)=0$
\item $\Hom_{\un{\CM}A}(S_{i}[t], M)=\Hom_{\CM A}(S_{i}[t], M)$. 
\end{enumerate}
\end{Lem}
\begin{proof}
We only prove (1), since (2) is immediately from (1).
Since $DA= A[-d+1]$ in $\D A$, then
\[ \Hom_{\CM A}(S_{i}[t], A)=\Hom_{\CM A}(S_{i}[t-d+1], DA)=D\h^{t-d+1}(S_{i})=0 \]
for  $0\le t \le d-2.$ \end{proof}



Now we prove  Theorem \ref{Thm:configuration}.

\begin{proof}[Proof of Theorem \ref{Thm:configuration}]
If $d=1$, then $A$ is an ordinary self-injective $k$-algebra and the assertion is known. We only prove it for $d>1$.
By Lemma \ref{12} and Proposition \ref{Prop:heart}, we have \[\Hom_{\un{\CM} A}(S_{i}, S_{j})=\Hom_{\CM A}(S_{i}, S_{j})=\Hom_{\h^{0}(A)}(S_{i},S_{j}),\] 
where $1\le i, j\le r$.
So the condition (1) in Definition \ref{Def:configuration} holds.
Since $\CM A= \D_{\le 0}^{\bb}\cap \D_{\ge -d+1}^{\bb}$ and $\Hom_{\D^{\bb}(A)}(S_{i}[>0], S_{j})=0$, then we have $S_{i}[t]\in \CM A$ for $1\le t\le d-1$ and $\Hom_{\CM A}(S_{i}[t], S_{j})=0$. Thus $\Hom_{\un{\CM}A}(S_{i}[t], S_{j})=0$ and  the condition (2) in Definition \ref{Def:configuration} is true.
Now we show the condition (3$'$) holds.
By Propositions \ref{Prop:cm} and \ref{Prop:heart}, we have 
\[\CM A =\Filt S[d-1]\ast \Filt S[d-2]\ast \cdots \ast\Filt S,\] where $S=\bop_{i=1}^{n}S_{i}$.
Thus $\un{\CM}A=\add \Filt\{S, S[1], \cdots, S[d-1]\}$ and (3$'$) is true. Then the set of simples forms a $d$-SMS of  $\un{\CM}A$.
\end{proof}

To recover the AR quiver  $\mathfrak A (\CM A)$ of $\CM A$ from  $\mathfrak A(\un{\CM}A)$, we need the following result.
\begin{Prop}
Let $A$ be a $d$-self-injective dg $k$-algebra. Let $P\in \add A$ be an indecomposable  dg $A$-module. Then $\rad P$ is an indecomposable object in $\CM A$ and it does not belong to $\add A$.  
\end{Prop}

\begin{proof}
Since $A$ is $d$-self-injective, then $\CM A=\D^{\bb}_{\le 0}(A)\cap \D^{\bb}_{\ge -d+1}(A)$. So $\rad P \in \CM A$.
We have a natural functor $\h: \D(A)\ra \mathsf{ Mod} \, \h(A)$ by taking cohomology, where  $\mathsf{Mod}\, \h(A)$ is the category of graded $\h(A)$-modules.
Notice that $\h(A)$ is a self-injective graded algebra and $\h(P)$ is an indecomposable projective-injective graded $\h(A)$-module.
So $\h(\rad P)=\rad(\h(P))$ is an indecomposable graded $\h(A)$-module and it does not belong to $\add \h(A)$. 
Then the assertion is true.
\end{proof}

\subsection{Combinatorial  configurations}
\label{Section:combi}

We give a combinatorial interpretation of Calabi-Yau configurations of Dynkin type in  our combinatorial framework. 

Let $\Delta$ be a Dynkin diagram. Recall from \cite{Gabriel} that a \emph{slice} of $\Z \Delta$ (see Section \ref{Section:quiver} for the definition of $\Z\Delta$) is a connected full subquiver which contains a unique representatives of the vertices $(r, q)$, $r\in \Z$ for each $q\in \Delta_{0}$. For each vertex $x=(p, q)$ of $\Z\Delta$, there is a unique slice admitting $x$ as its unique source. We call this slice  \emph{the slice starting at $x$}. 	
 An integer-valued function $f$ on the vertices of $\Z\Delta$ is \emph{additive} if it satisfies the equation $f(x)+f(\tau x)=\sum_{y\ra x \in (\Z\Delta)_{1}}f(y)$. It is easy to see that $f$ is determined by its value on a slice. Now we define $f_{x}$ as the additive function which has value $1$ on the slice starting at $x$ for each vertex $x$.  Let $Q_{x}$ be the connected component of the full subquiver $\{y\in (\Z\Delta)_{0}\mid f_{x}(y)>0\}$ of  $\Z\Delta$ containing $x$. We define a map $h_{x}$ by 
\[ h_{x}(y)= \begin{cases}    f_{x}(y) & \text{if } y\in (Q_{x})_{0}; \\ 0 & \text{otherwise}.
\end{cases}  \]
Notice that $h_{x}$ is no longer an additive map.
Let us see an example of type $D$.
\begin{Ex}
Let $x$ be the marked vertex in $(\Z D_{4})_{0}$. Then the value of $h_{x}$ is given as follows
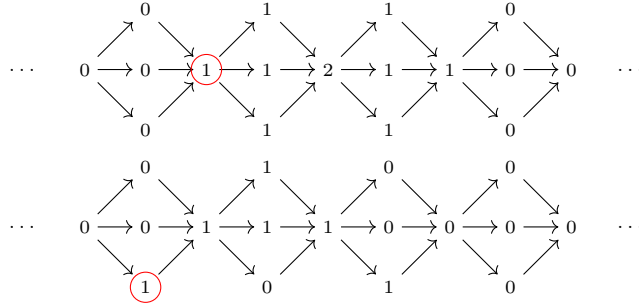
\begin{center}
\scriptsize{
\begin{tikzpicture}[scale=.8]
 \draw 
 node at (-3,0){$\dots$} node at (7,0){$\dots$}
 node (a0) at (0,0) {$1$}  node (a1) at (1,0) {$1$}  node (a2) at (2,0) {$2$}  node (a3) at (3,0) {$1$}  node (a4) at (4,0) {$1$}  node (a5) at (5,0) {$0$}  node (a-1) at (-1,0) {$0$} node(a-2) at (-2,0){$0$}
 node (a6) at (6,0){$0$}
  node (b1) at (1,1) {$1$}  node (b3) at (3,1) {$1$}  node (b5) at (5,1) {$0$}  node (b-1) at (-1,1) {$0$}
  node (c1) at (1,-1) {$1$}  node (c3) at (3,-1) {$1$}  node (c5) at (5,-1) {$0$}  node (c-1) at (-1,-1) {$0$};
  \draw[->](a-2)--(b-1);\draw[->](a-2)--(a-1);\draw[->](a-2)--(c-1);
  \draw[->] (b-1)--(a0); \draw[->] (a-1)--(a0); \draw[->] (c-1)--(a0); \draw[->] (a0)--(b1);\draw[->] (a0)--(a1);\draw[->] (a0)--(c1); \draw[->] (b1)--(a2); \draw[->] (a1)--(a2); \draw[->] (c1)--(a2);  \draw[->] (a2)--(b3); \draw[->] (a2)--(a3);\draw[->] (a2)--(c3); \draw[->] (c3)--(a4); \draw[->] (a3)--(a4);  \draw[->] (b3)--(a4); \draw[->] (a4)--(b5); \draw[->] (a4)--(a5); \draw[->] (a4)--(c5); \draw[->] (b5)--(a6);\draw[->] (a5)--(a6); \draw[->] (c5)--(a6);
  \draw[red] (a0) circle [radius=0.25];
  \end{tikzpicture} } \end{center}
\begin{center}
\scriptsize{
\begin{tikzpicture}[scale=.8]
 \draw 
 node at (-3,0){$\dots$} node at (7,0){$\dots$}
 node (a0) at (0,0) {$1$}  node (a1) at (1,0) {$1$}  node (a2) at (2,0) {$1$}  node (a3) at (3,0) {$0$}  node (a4) at (4,0) {$0$}  node (a5) at (5,0) {$0$}  node (a-1) at (-1,0) {$0$} node(a-2) at (-2,0){$0$}
 node (a6) at (6,0){$0$}
  node (b1) at (1,1) {$1$}  node (b3) at (3,1) {$0$}  node (b5) at (5,1) {$0$}  node (b-1) at (-1,1) {$0$}
  node (c1) at (1,-1) {$0$}  node (c3) at (3,-1) {$1$}  node (c5) at (5,-1) {$0$}  node (c-1) at (-1,-1) {$1$};
  \draw[->](a-2)--(b-1);\draw[->](a-2)--(a-1);\draw[->](a-2)--(c-1);
  \draw[->] (b-1)--(a0); \draw[->] (a-1)--(a0); \draw[->] (c-1)--(a0); \draw[->] (a0)--(b1);\draw[->] (a0)--(a1);\draw[->] (a0)--(c1); \draw[->] (b1)--(a2); \draw[->] (a1)--(a2); \draw[->] (c1)--(a2);  \draw[->] (a2)--(b3); \draw[->] (a2)--(a3);\draw[->] (a2)--(c3); \draw[->] (c3)--(a4); \draw[->] (a3)--(a4);  \draw[->] (b3)--(a4); \draw[->] (a4)--(b5); \draw[->] (a4)--(a5); \draw[->] (a4)--(c5); \draw[->] (b5)--(a6);\draw[->] (a5)--(a6); \draw[->] (c5)--(a6);
  \draw[red] (c-1) circle [radius=0.25];
  \end{tikzpicture} } \end{center}
\end{Ex}

Let $\phi$ be a weakly admissible automorphism (see Section \ref{Section:quiver}) of $\Z \Delta$. Let $\pi: \Z \Delta \ra \Z\Delta/\phi$ be the natural projection. For $x\in \Z\Delta$, we define $h^{\phi}_{x}$ as follows
\[  h^{\phi}_{x}(y)=\sum_{\pi(z)=y} h_{x}(z) \mbox{ for } y\in (\Z\Delta/\phi)_{0}   \]
If $\phi$ is identity, then $h^{\phi}$ is exactly $h$.
Recall we have defined the ``shift permutation'' $[1]$  in section \ref{Section:quiver}. Now we use $h^{\phi}_{x}$ and $[1]$ to define combinatorial configurations.
 \begin{Def}\label{Def:combi}
Let $\Delta$ be a Dynkin diagram and let $\phi$ be a weakly admissible group. Let $C$ be a subset of $(\Z\Delta/\phi)_{0}$.  For $d\ge 1$, if the following conditions hold
 \begin{enumerate}[\rm $\bullet$]
   \item $h^{\phi}_{x}(y)=\delta_{x, y}$ for $x,y \in C$;
   \item $h^{\phi}_{x}(y[-j])=0$ for $x, y \in C$ and $0< j \le d-1$;
   \item For any  vertex $z$ in $(\Z\Delta/\phi)_{0}$, there exists $x\in C$ and $0\le j \le d-1$, such that $h^{\phi}_{x}(z[-j])\not=0$. \end{enumerate}
 we call $C$ a \emph{$(-d)$-combinatorial configuration}.
\end{Def}

The connection between  configurations of $\Z\Delta$ and configurations of $\Z\Delta/\phi$ is given as follows.
\begin{Prop}
Let $C$ be a subset of $(\Z\Delta/\phi)_{0}$. Then $C$ is a $(-d)$-combinatorial configuration of $\Z\Delta/\phi$ if and only if $\pi^{-1}(C)$ is a $(-d)$-combinatorial configuration of $\Z\Delta$.
\end{Prop}

\begin{proof}
Using the definition $ h^{G}_{x}(y)=\sum_{\pi(z)=y} h_{x}(z) $, it is easy to show the statement.
\end{proof}

Here is a simple example:
\begin{Ex}
We consider the quiver $\Z A_{2}/\SSS[2]$,
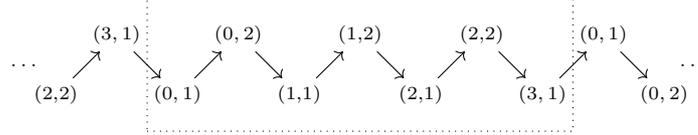
\begin{center}
\scriptsize{
\begin{tikzpicture}[scale=.8]

\draw[yshift=0cm, xshift=0cm] 
node at (-2.5,0.5){$\dots$} 
node at (8.5,0.5){$\dots$}
node (x-1)  at (-2,0) {(2,2)}  node (x) at (0,0) {$(0,1)$}  node (x1) at (2,0) {(1,1)}
   node (x2) at (4,0) {(2,1)} node (x3) at (6,0) {$(3,1)$} node (x4) at (8,0) {$(0,2)$}
   node (y-1) at (-1,1) {$(3,1)$}  node (y) at (1,1) {$(0,2)$}  node (y1) at (3,1) {(1,2)}
   node (y2) at (5,1) {(2,2)} node (y3) at (7,1) {$(0,1)$}   
   [dotted](-0.5,-0.6) rectangle (6.5,1.6) ;
  \draw[->] (x-1)--(y-1); \draw[->](y-1)--(x);
  \draw[->](x)--(y);  \draw[->] (y)--(x1);
  \draw[->](x1)--(y1); \draw[->](y1)--(x2); \draw[->](x2)--(y2);\draw[->](y2)--(x3);\draw[->](x3)--(y3);\draw[->](y3)--(x4) ;
\end{tikzpicture} }
\end{center}

One can check that there only exist  seven $(-2)$-combinatorial  configurations. We give all of them

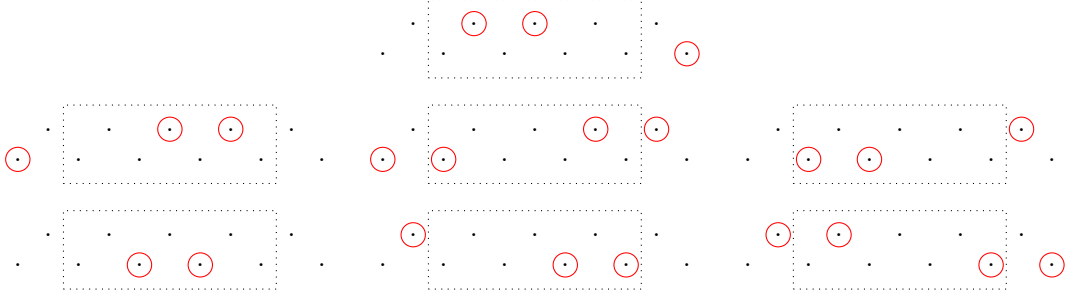
\begin{center}
\begin{tikzpicture}[scale=.4]

\draw[yshift=3.5cm, xshift=12cm] node (x-1) at (-2,0) {.}  node (x) at (0,0) {$.$}  node (x1) at (2,0) {.}
   node (x2) at (4,0) {.} node (x3) at (6,0) {.} node (x4) at (8,0) {$.$}
   node  at (-1,1) {.}  node (y) at (1,1) {$.$}  node (y1) at (3,1) {.}
   node (y2) at (5,1) {.} node (y3) at (7,1) {$.$}   
  [dotted] (-0.5,-0.8) rectangle (6.5,1.8); 
    \draw[red] (y1) circle [radius=0.4]   (y) circle [radius=0.4]  (x4) circle [radius=0.4];

\draw node (x-1) at (-2,0) {.}  node (x) at (0,0) {$.$}  node (x1) at (2,0) {.}
   node (x2) at (4,0) {.} node (x3) at (6,0) {.} node (x4) at (8,0) {$.$}
   node  at (-1,1) {.}  node (y) at (1,1) {$.$}  node (y1) at (3,1) {.}
   node (y2) at (5,1) {.} node (y3) at (7,1) {$.$}   
  [dotted] (-0.5,-0.8) rectangle (6.5,1.8); 
    \draw[red] (y1) circle [radius=0.4]   (y2) circle [radius=0.4] (x-1) circle [radius=0.4];
    
\draw[xshift=12cm] node  (x-1) at (-2,0) {.}  node (x) at (0,0) {$.$}  node (x1) at (2,0) {.}
   node (x2) at (4,0) {.} node (x3) at (6,0) {.} node (x4) at (8,0) {$.$}
   node  at (-1,1) {.}  node (y) at (1,1) {$.$}  node (y1) at (3,1) {.}
   node (y2) at (5,1) {.} node (y3) at (7,1) {$.$}   
   [dotted](-0.5,-0.8) rectangle (6.5,1.8); 
    \draw[red] (y2) circle [radius=0.4]   (x) circle [radius=0.4] (x-1) circle [radius=0.4] (y3) circle [radius=0.4];

\draw[xshift=24cm] node  at (-2,0) {.}  node (x) at (0,0) {$.$}  node (x1) at (2,0) {.}
   node (x2) at (4,0) {.} node (x3) at (6,0) {.} node (x4) at (8,0) {$.$}
   node  at (-1,1) {.}  node (y) at (1,1) {$.$}  node (y1) at (3,1) {.}
   node (y2) at (5,1) {.} node (y3) at (7,1) {$.$}   
   [dotted](-0.5,-0.8) rectangle (6.5,1.8); 
    \draw[red] (x) circle [radius=0.4]   (x1) circle [radius=0.4] (y3) circle [radius=0.4];

\draw[yshift=-3.5cm, xshift=0cm] node  at (-2,0) {.}  node (x) at (0,0) {$.$}  node (x1) at (2,0) {.}
   node (x2) at (4,0) {.} node (x3) at (6,0) {.} node (x4) at (8,0) {$.$}
   node  at (-1,1) {.}  node (y) at (1,1) {$.$}  node (y1) at (3,1) {.}
   node (y2) at (5,1) {.} node (y3) at (7,1) {$.$}   
  [dotted] (-0.5,-0.8) rectangle (6.5,1.8); 
    \draw[red] (x1) circle [radius=0.4]   (x2) circle [radius=0.4];
    
\draw[yshift=-3.5cm, xshift=12cm] node  at (-2,0) {.}  node (x) at (0,0) {$.$}  node (x1) at (2,0) {.}
   node (x2) at (4,0) {.} node (x3) at (6,0) {.} node (x4) at (8,0) {$.$}
   node (y-1)  at (-1,1) {.}  node (y) at (1,1) {$.$}  node (y1) at (3,1) {.}
   node (y2) at (5,1) {.} node (y3) at (7,1) {$.$}   
  [dotted] (-0.5,-0.8) rectangle (6.5,1.8); 
    \draw[red] (x2) circle [radius=0.4]   (x3) circle [radius=0.4] (y-1) circle [radius=0.4];

\draw[yshift=-3.5cm, xshift=24cm] node  at (-2,0) {.}  node (x) at (0,0) {$.$}  node (x1) at (2,0) {.}
   node (x2) at (4,0) {.} node (x3) at (6,0) {.} node (x4) at (8,0) {$.$}
   node (y-1) at (-1,1) {.}  node (y) at (1,1) {$.$}  node (y1) at (3,1) {.}
   node (y2) at (5,1) {.} node (y3) at (7,1) {$.$}   
  [dotted] (-0.5,-0.8) rectangle (6.5,1.8); 
    \draw[red] (x3) circle [radius=0.4]   (y) circle [radius=0.4] (y-1) circle [radius=0.4] (x4) circle [radius=0.4];
\end{tikzpicture}
 \end{center}
\end{Ex}

\subsection{Calabi-Yau configurations VS. combinatorial configurations}

In this section we study the connection between Calabi-Yau configurations and combinatorial configurations.
Let $\T$ be a Hom-finite Krull-Schmidt  triangulated category with  AR quiver  isomorphic to $\Z \Delta/\SSS[d]$. We identify the elements in $\ind \T$ as the vertices in $\Z\Delta/\SSS[d]$.  Let $\pi: \Z\Delta \ra \Z\Delta/\SSS[d]$ be the natural surjection.  We denote by $\bar{h}$ for the map $h^{\SSS[d]}$.
We first show that

\begin{Prop}\label{Prop:vs}
For any $X, Y\in \ind \T$, we have $\dim \Hom_{\T}(X, Y)= \bar{h}_{X}(Y)$.
\end{Prop}
To prove this, we consider the free Abelian monoid $\N_{\ge 0}(\Z\Delta)$ generated by $(\Z\Delta)_{0}$. For any $n\in \N_{\ge 0}$ and $x\in (\Z\Delta)_{0}$, we define a map $f_{n}(x): \N_{\ge 0}(\Z\Delta) \ra \N_{\ge 0}(\Z\Delta)$ by 

\[ f_{n}(x)=  \begin{cases}
 x & \text{if } n = 0;\\
 \sum_{x\ra y\in (\Z\Delta)_{0}}y & \text{if } n=1;\\
 f_{1}(f_{n-1}(x))-\tau^{-1}(f_{n-2}(x)) & \text{if }  n\ge 2.
 \end{cases}  \]
By the definition, we have the following lemma immediately.
\begin{Lem}
For any vertices $x, y$ in $(\Z\Delta)_{0}$, the multiplicity of $y$ in $\bigcup_{i\ge 0} {\rm supp}f_{i}(x)$ is $h_{x}(y)$.
\end{Lem}
For any module $M\cong \bigoplus_{i=1}^{l}M_{i}^{t_{i}}$ in $\T$, we identify it as the element $\sum_{i=1}^{l}t_{i}M_{i}$ in $\N_{\ge 0}\Z\Delta$, and vice versa. 

\begin{Prop}\cite[Theorems 4.1 and 7.1]{Iyama2}
Let $X\in \ind \T$, then we have a surjective morphism
\[  \Hom_{\T}(f_{n}(X), ?) \ra  \rad^{n}_{\T}(X, ?) \]
of functors which induces an isomorphism
\[    \Hom_{\T}(f_{n}(X), ?)/\rad_{\T}(f_{n}(X), ?)       \cong \rad^{n}_{\T}(X,?)/  \rad^{n+1}_{\T}(X,?).      \]
\end{Prop}

\begin{proof}[Proof of Proposition \ref{Prop:vs}]
Since $\T$ is representation-finite, then $\rad^{n}_{\T}(X,?)	=0$ for $n$ large enough. For any $X, Y \in \ind \T$, we have
 \begin{IEEEeqnarray*}{+rCl+x*}
 \dim_{k}\Hom_{\T}(X, Y)& = & 	\sum_{i\ge 0}\dim_{k}(\rad^{i}_{\T}(X,Y)/  \rad^{i+1}_{\T}(X,Y)) \\ & =& \sum_{i\ge 0}\dim_{k}(\Hom_{\T}(f_{i}(X), Y)/\rad_{\T}(f_{i}(X), Y) )   	 \\
 & = &  \sum_{\pi(y)=Y}	\sum_{i\ge 0}(\text{multiplicity of $y$ in } f_{i}(X)) 
 =\sum_{\pi(y)=Y} h_{X}(y) 
 = \bar{h}_{X}(Y) &\qedhere
 \end{IEEEeqnarray*}
 \end{proof}

The following theorem shows that Calabi-Yau configurations in $\T$  coincide with combinatorial configurations.
\begin{Thm}\label{Thm:ComCY}
 Let $ C \subset \ind\T$ be a subset. Then the following are equivalent:
 \begin{enumerate}[\rm (1)]
 \item $C$ is a $(-d)$-CY configuration in $\T$;
 \item $C$ is a $(-d)$-combinatorial configuration in $\Z\Delta/\SSS[d+1]$.
  \end{enumerate}
\end{Thm}

\begin{proof}
 It  directly follows from Proposition \ref{Prop:vs}.
\end{proof}

Thanks to the theorem above, by abuse of notation, we may use the name ``Calabi-Yau configuration'' even in the combinatorial context.


\section{Trivial extension dg algebras}\label{Section:trivial}

In this section, we consider a class of self-injective dg algebras given by trivial extension.
Some results here will be used to prove a certain converse of Theorem \ref{Thm:configuration} (see Theorem \ref{Thm:main}).
 Let $B$ be a non-positive  proper dg $k$-algebra. Let $\inf(B)$ be the smallest integer $i$ such that $\h^{i}(B)\not=0$. Clearly, $\inf(B)\le 0$. For $d\in \Z$, we consider the complex $A:=B\op DB[d]$. We regard $A$ as a dg $k$-algebra whose  multiplication is given by 
\[ (a, f)(b, g):= (ab, ag+ fb) \]
where $a,b \in B$ and $f,g \in DB$, and the differential of $A$ inherits from $B$ and $DB$. If $d\ge -\inf(B)$, then $A$ is  non-positive. Moreover, we have an isomorphism $DA\simeq A[-d]$ in $\D A^{\rm e}$. If $\inf(B)=0$ and $d=0$, $A$ is the usual trivial extension.

We give a result analogous to \cite[Theorem 3.1]{Rickard}.
\begin{Prop} \label{Prop:Ric}
Let $B$ be a non-positive proper dg $k$-algebra and let $X$ be a silting object in $\per B$. Let $B':=\shEnd_{B}(X)$. Consider the  trivial extension dg algebras $A=B\op DB[d]$ and $A'=B'\op DB'[d]$, then $\per A$ is triangle equivalent to $\per A'$. \end{Prop}

\begin{proof}
We may regard $A$ as a dg $B$-module through the injection $B\hookrightarrow A$. Consider the functor
\[ ?\ot_{B}^{\mathsf L}A:  \per B \lra \per A. \]
It sends $B$ to $A$. Since $\thick_{B}(X)=\per B$, then $\thick_{A}(X\ot_{B}A)=\per A$.  
Then $X\ot_{B}A$ is a compact generator of $\D A$ and
 we have a triangle equivalence between $\per \shEnd(X\ot_{B} A)$ and $\per A$ (see for example \cite[Lemma 4.2]{Keller94}) .
Next we consider the dg algebra  $\shEnd(X\ot_{B} A)$.
Notice that, as $k$-complexes, we have the following isomorphisms.
\[ \shHom_{A}(X\ot_{B} A, X\ot_{B} A) \simeq \shHom_{B}(X, X\op(X\ot_{B} DB[d]))\simeq \shEnd_{B}(X) \op D\shEnd_{B}(X)[d]. \]
In fact these isomorphisms also induce an isomorphism between dg algebras $\shEnd(X\ot_{B} A)$ and $\shEnd_{B}(X) \op D\shEnd_{B}(X)[d]$. 
Then  $\shEnd_{A}(X\ot_{B}A)$ is isomorphic to 
  $A'=B'\op DB'[d]$.
So $\per A$ is triangle equivalent to $\per A'$.
\end{proof}


In the sequel, we only consider the special case that $\inf(B)=0$ and $B$ has finite global dimension.  In this case, $A$ is a Gorenstein proper dg $k$-algebra.
 If $A'$ considered in Proposition \ref{Prop:Ric}
is also Gorenstein proper, 
then we have $\un{\CM}A \simeq \un{\CM}A'$ by the result above. 
We show  $\un{\CM}A$ is a cluster category in the following sense. For the details of orbit category, we refer to \cite{Keller05}.

\begin{Def}
Let $B$ be a finite dimensional hereditary $k$- algebra. The \emph{$(-d)$-cluster category $\cal C_{-d}(B)$} is defined as  the orbit category $\D^{\rm b}({\mod} B)/\nu[d]$, where $\nu$ is the Nakayama functor.
\end{Def}

Keller proved the following result.
\begin{Prop}\cite[Theorem 2]{Keller05}\cite{Kellercorr} 
Let $B$ be a finite-dimensional hereditary  $k$-algebra.  Let $A=B\op DB[d-1]$ be the trivial extension dg algebra. Then 
\begin{enumerate}[\rm(1)]
\item $\cal C_{-d}(B)$ has a structure of triangulated category;
\item $\cal C_{-d}(B)$ is triangle equivalent to $\D_{\sg}(A)$.
\end{enumerate}
\end{Prop}
Notice that in this case $\D_{\sg}(A)=\thick_{A}(B)/\per A$  holds by the fact $\thick_{A}(B)=\D^{\rm b}(A)$ (see Proposition \ref{Prop:Dbfinite}).  By using this proposition, we have a useful observation, where we denote by $\mathfrak A (\un{\CM}A)$ the AR quiver of $\un{\CM}A$.
\begin{Cor}\label{Cor:cluster}
Let $B$ be a finite-dimensional hereditary $k$-algebra and let $A=B\op DB[d-1]$ for $d\ge 1$. Then
\begin{enumerate}[\rm(1)]
\item
The stable category $\un{\CM }A$ is triangle equivalent to $\cal C_{-d}(B)$;
\item We have $\mathfrak A (\un{\CM}A)=\mathfrak A(\D^{\bb}(\mod B)/\nu[d])=\Z \Delta/\nu[d]$. In particular, $\mathfrak A (\mod B)$ is a full sub-quiver of $\mathfrak A(\un{\CM}A)$.
\end{enumerate}
\end{Cor}


We end this section with a concrete example.
\begin{Ex}
Let $B$ be the $k$-algebra given by the quiver $1\ra 2 \la 3$. Let $A=B\op DB[1]$ be the trivial extension dg $k$-algebra. Then we may regard $A$ as the dg $k$-algebra   given by \xymatrix{1\ar@/^-0.6pc/[r]_{\alpha_{1}}  & 2 \ar@/^0.6pc/[r]^{\beta_{2}} \ar@/^-0.6pc/[l]_{\alpha_{2}}&3 \ar@/^0.6pc/[l]^{\beta_{1}}},
 with relations $\{\alpha_{1}\alpha_{2}\alpha_{1}, \beta_{1}\beta_{2}\beta_{1}, \alpha_{1}\beta_{2}, \beta_{1}\alpha_{2}, \alpha_{2}\alpha_{1}-\beta_{2}\beta_{1} \}$ and $0$ differential. Further, the degrees of $A$ are induced by $\deg \alpha_{1}=0=\deg \beta_{1}$, 
 $\deg \alpha_{2}=-1=\deg \beta_{2}$.
 By Corollary \ref{Cor:cluster}, we know the $\un{\CM}A$ is triangle equivalent to $\D^{\bb}(\mod B)/\nu[2]$. We describe  $\mathfrak A (\un{\CM}A)$ as follows.
 
  {\small
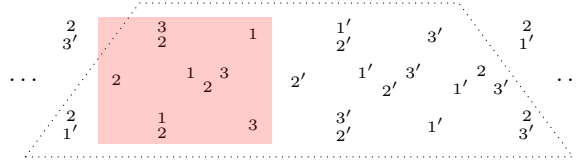
       \begin{center}
         \begin{tikzpicture}[scale=0.6]
         \draw
         node (00) at (0,0) {$\begin{smallmatrix}&2&\end{smallmatrix}$}
         node (20) at (2,0) {$\begin{smallmatrix} 1& &3\\ &2& \end{smallmatrix}$}
         node (40) at (4,0) {$\begin{smallmatrix}&2' &\end{smallmatrix}$}
         node (60) at (6,0) {$\begin{smallmatrix}1'&&3'\\&2'&\end{smallmatrix}$}
         node (80) at (8,0) {$\begin{smallmatrix}&2&\\ 1'&&3' \end{smallmatrix}$}
         node at (-2,0) {$\dots$}
       node at (10,0){$\dots$}
         
         node (31) at (3,1) {$\begin{smallmatrix}1\end{smallmatrix}$}
         node (51) at (5,1) {$\begin{smallmatrix}1'\\ 2'\end{smallmatrix}$}
         node (71) at (7,1) {$\begin{smallmatrix}3'\end{smallmatrix}$}
         node (91) at (9,1) {$\begin{smallmatrix}2\\ 1'\end{smallmatrix}$}
         node (11) at (1,1) {$\begin{smallmatrix}3\\2\end{smallmatrix}$}

         node (-11) at (-1,1) {$\begin{smallmatrix}2\\ 3'\end{smallmatrix}$}
         node (-1-1) at (-1,-1) {$\begin{smallmatrix}2\\ 1'\end{smallmatrix}$}
        
         node (1-1) at (1,-1) {$\begin{smallmatrix}1\\2\end{smallmatrix}$}
         node (3-1) at (3,-1) {$\begin{smallmatrix}3\end{smallmatrix}$}
         node (5-1) at (5,-1) {$\begin{smallmatrix}3'\\ 2'\end{smallmatrix}$}
         node (7-1) at (7,-1) {$\begin{smallmatrix}1'\end{smallmatrix}$}
         node (9-1) at (9,-1) {$\begin{smallmatrix}2\\ 3'\end{smallmatrix}$}
         [dotted] (-2,-1.7)--(0.5,1.7)--(7.5,1.7)--(10,-1.7)--(-2,-1.7);  
         \fill[opacity=0.2, red] (-0.4,1.4)--(-0.4,-1.4)--(3.4,-1.4)--(3.4,1.4)--(-0.4,1.4);
        
                \end{tikzpicture} 
     \end{center}}  
\noindent where the arrows are omitted and a fundamental domain is outlined in dotted line. 
Notice that the shaded part is exactly the AR quiver of $\mod B$.
\end{Ex}


\section{CY configurations and symmetric dg algebras}\label{Section:mainsection}



The aim of this section is to show the following theorem, which is the converse of Theorem \ref{Thm:configuration}.


\begin{Thm}\label{Thm:main}
Let $\Delta$ be a Dynkin diagram and $d\ge 1$. 
Let $C$ be a subset of vertices of $\Z \Delta/\SSS[d]$. The following are equivalent.
\begin{enumerate}[\rm (1)]
 \item $C$ is a $(-d)$-CY configuration;
 \item There exists a $d$-symmetric   dg $k$-algebra $A$ with the AR quiver of ${\CM}A$ is $(\Z \Delta)_{C}/\SSS[d]$.
\end{enumerate}
\end{Thm}

To  prove  this, we need some preparations. We first study the connection between simple-minded collections (SMCs) in $\D^{\bb}(A)$ and $(-d)$-CY configurations in $\un{\CM}A$. 

\subsection{$(-d)$-CY configurations are given by SMCs}
To prove Theorem \ref{Thm:main}, we need the following notion.
\begin{Def}
Let $\T$ be a $k$-linear Hom-finite Krull-Schmidt triangulated category and let $R$ be a  set of  objects of $\T$. 
We call $R$ a \emph{pre-simple-minded collection} (\emph{pre-SMC}) of $\T$,  if for any $X, Y\in R$, the following conditions hold.
 \begin{enumerate}[\rm $\bullet$]
 \item   $\dim_{k}\Hom_{\T}(X, Y)=\delta_{X,Y}$. 
\item $\Hom_{\T}(X, Y[< 0])=0$.
\end{enumerate}
Moreover, if $\T=\thick_{\T}(R)$ also holds, we call $R$ a \emph{simple-minded collection} (\emph{SMC}) of $\T$.
\end{Def}

In this subsection, our aim is to show the following result, which plays a key role in the proof of Theorem \ref{Thm:main}.

\begin{Thm}\label{Thm:SMCsurj}
Let $\Delta$ be a Dynkin diagram. 
Let $A=k\Delta \op D(k\Delta)[d-1]$ be the trivial extension dg $k$-algebra.
 Then
the quotient functor $\D^{\bb}(A)\ra \D_{\sg}(A) \cong \un{\CM} A$ induces a surjective map 
$$\{\text{SMCs in $\D^{\bb}(A)$}\} \longrightarrow \{ \text{$(-d)$-CY configurations in $\un{\CM}A$}  \}.  $$
\end{Thm}



Notice  that in this setting, the notion `$d$-SMS' coincides with `$(-d)$-CY configuration' by \cite[Proposition 2.13]{CSP}, since $\CM A$ has only finitely many indecomposable objects.
 So we may regard $(-d)$-CY configuration and $d$-SMS as the same  notion later in this section.
 
Let $R$ be a pre-SMC of $\D^{\bb}(A)$. The \emph{SMC reduction $\cal U$ of $\D^{\bb}(A)$ with respect to $R$} is defined as the Verdier quotient (see \cite[Section 3.1]{Jin})
 $$\cal U:=\D^{\bb}(A)/\thick(R).$$
 Let $E$ be a pre-$d$-SMS of $\D_{\sg}(A)$.
 The \emph{SMS reduction} $(\D_{\sg}(A))_{E}$  of  $\D_{\sg}(A)$ with respect to $E$ is defined as the subcategory (see \cite[Section 6]{CSP})
$$ (\D_{\sg}(A))_{E}:=\{ X\in \D_{\sg} \mid \Hom_{\D_{\sg}(A)}(X, Y[-i])=0 \text{ for any $i=0,1, \cdots, d$ and $Y\in E$}\}. $$
The following proposition shows some important properties of SMC and SMS reduction.
\begin{Prop}\label{Prop:reduction}
\begin{enumerate}[\rm(1)]
\item \cite[Theorem 3.1]{Jin}
Let $R$  be a pre-SMC of $\D^{\bb}(A)$. Assume $\Filt(R)$ is functorially finite in $\D^{\bb}(A)$. Then
the natural functor $\D^{\bb}(A)\ra \cal U$ induces a bijection
\[ \{\text{SMCs in $\D^{\bb}(A)$ containing $R$}\} 
\longleftrightarrow \{\text{SMCs in $\cal U$}\}; \]
\item \cite[Theorem 6.6]{CSP}
Let  $E$ be a  pre-$d$-SMS of  $\D_{\sg}(A)$. Assume  $\Filt(E)$ is functorially finite in  $\D_{\sg}(A)$. Then the SMS reduction $(\D_{\sg}(A))_{E}$ has a structure of triangulated category and there is a bijection 
\[\{\text{$d$-SMSs in $\D_{\sg}(A)$ containing $E$}\} \longleftrightarrow \{ \text{$d$-SMSs in $(\D_{\sg}(A))_{E}$}\};\]
\item \cite[Corollary 4.15  and Proposition 5.1]{Jin} Let $A$ be a representation finite  $d$-self-injective dg algebra. Then the quotient functor $\D^{\bb}(A)\ra \D_{\sg}(A) \cong \un{\CM} A$ induces a well-defined  map 
\[\{\text{SMCs in $\D^{\bb}(A)$}\} \longrightarrow \{ \text{$d$-SMSs in $\un{\CM}A$}  \}.  \]
\end{enumerate}
\end{Prop}

Now we are ready to show Theorem \ref{Thm:SMCsurj}.

\begin{proof}[Proof of Theorem \ref{Thm:SMCsurj}]

We first mention  that $\un{\CM}A$ is triangle equivalent to the orbit category $\D^{\bb}(\mod k\Delta)/\nu[d]$ and the AR quiver of $\un{\CM}A$ is $\Z\Delta/\SSS[d]$   by Corollary \ref{Cor:cluster}. Moreover, we may assume $\Delta$ has  an alternating orientation (that is,  each vertex in $\Delta$ is either sink  or source).
We apply  induction on the number $n$ of  vertices of $\Delta$. If $n=1$, then $\Delta=A_{1}$. In this case,  any indecomposable object in $\un{\CM}A$ is a shift of the simple object and the argument is clearly true. 

For the general case, let $C$ be a $(-d)$-CY configuration in $\un{\CM}A$. 
Since the quiver $\Delta$ has
an alternating orientation, each $\tau$-orbit in the AR quiver of $\mod k\Delta$ contains a simple $k\Delta$-module.
Then  by Corollary \ref{Cor:cluster} (2), there exists a simple dg $A$-module $S$  and an integer $s$, such that $\tau^{s}S\in C$ in $\un{\CM}A$. Notice that $\tau=\SSS[-1]$ and $\SSS=[-d]$, so we have $S[t]\in C$, where $t=-sd$. 
Without loss of generality, in the following we may assume $S\in C$.

It is clear that $\{S\}$ is a pre-SMC in $\D^{\bb}(A)$, and  the SMC reduction $\D^{\bb}(A)/\thick(S)$ is triangle equivalent to $\D^{\bb}(eAe)$ by \cite[Proposition 3.9]{Jin}, where $e\in k\Delta$ is an idempotent such that $\Top(1-e)A=S$. Considering  the SMS reduction $(\D_{\sg}(A))_{S}$ of $\D_{\sg}(A)$ with respect to $S$,  then by \cite[Theorem 6.4]{Jin}, we have 
a triangle equivalence $\D_{\sg}(eAe)\simeq (\D_{\sg}(A))_{S}$ and 
the following commutative diagram.
 {\small \[\xymatrixcolsep{3pc}\xymatrixrowsep{4pc}\xymatrix{  \{ \text{SMCs in $\D^{\bb}(A)$ containing $S$} \}
  \ar[r]  \ar[d]_{\simeq}& 
  \{ \text{$d$-SMSs in $\D_{\sg}(A)$ containing $S$} \}
   \ar[d]^{\simeq}\\ \{ \text{SMCs in $\D^{\bb}(eAe)$}\} \ar[r] & \{\text{$d$-SMSs in $\D_{\sg}(eAe)$} \}
}.\]}

 \noindent Notice that $eAe\cong k\Delta'\op D(k\Delta')[d-1]$, where $\Delta'$ is obtained from $\Delta$ by deleting the vertex $i$, which corresponds to $S$, and the arrows connected to $i$. Then $\Delta'$ also has an alternating orientation and has  $n-1$ vertices. 
By induction, we may assume the lower map is surjective.  
Since $C$ is a $(-d)$-CY configuration in $\D_{\sg}(A)$ containing $S$, then $C\backslash \{ S\}$ is a $(-d)$-CY configuration in $(\D_{\sg}(A))_{S}$ by Proposition \ref{Prop:reduction} (2). 
So there is a SMC $H$ of $\D^{\bb}(eAe)$ such that $H\cong C\backslash \{ S\}$ in $\D_{\sg}(eAe)$. We may regard $H$ as a subset of  $\D^{\bb}(A)/\thick(S)$ through $\D^{\bb}(eAe)\simeq \D^{\bb}(A)/\thick(S)$. Then $H\cup \{S\}$ is a SMC of $\D^{\bb}(A)$ by Proposition \ref{Prop:reduction} (1). By the diagram above, we have that $H\cup \{S\}$ sends to $C$ in $\D_{\sg}(A)$. 
 By induction, the assertion is true.
\end{proof}

\subsection{Proof of Theorem \ref{Thm:main}}
To prove our main theorem, we need the following 
generalization of silting-SMC correspondence  \cite[Theorem 6.1]{KoY} due to \cite{SY}.

\begin{Prop}\label{Prop:KoY}
\cite[Theorem 1.1]{SY}
Let $A$ be a non-positive proper dg $k$-algebra.
Then 
\begin{enumerate}[\rm(1)]
\item There is a bijection 
$$\{\text{SMCs of $\D^{\bb}(A)$}\} \longleftrightarrow \{\text{silting objects of $\per A$} \};$$
\item Let $\{ X_{1}, \cdots, X_{n}\}$ be a SMC of $\D^{\bb}(A)$ which corresponds to a silting $P\in \per A$. Let $B:=\shEnd_{A}(P)$ be the endmoprhism dg $k$-algebra. Then the natural equivalence   $?\ot_{B}^{\bf L}P: \D^{\bb}(B)\ra \D^{\bb}(A)$ sends simple $B$-modules to $\{ X_{1}, \cdots, X_{n}\}$.
\end{enumerate}
\end{Prop}

We also need the following observation.

\begin{Lem}\label{Lem:siltGoren}
Let $A$ be a proper dg $k$-algebra and let $P$ be a silting dg $A$-modules. Let $B:=\shEnd_{A}(P)$. Then if $A$ is $d$-symmetric, so is $B$. 
\end{Lem}

\begin{proof}
Let $A$ be a $d$-symmetric dg $k$-algebra, that is $\add A[-d+1]=\add DA$ in $\D(A^{\rm op}\ot_{k} A)$. We may assume $\D A\cong A[-d+1]$ in $\D(A^{\rm op}\ot_{k} A)$.
 We may also assume $DA$ and $A$  have property (P) as $(A^{\rm op}\ot_{k} A)$-modules and $P$ has property (P) as $(B^{\rm op}\ot_{k} A)$-module (see \cite[Section 3.1]{Keller94}). 
Then there is a quasi-isomorphim $f:A\ra DA[d-1]$ of $(A^{\rm op}\ot_{k} A)$-modules, which induces a quasi-isomorphism $1\ot f: P\ra P\ot_{A}DA[d-1]$ of $(B^{\rm op}\ot_{k} A)$-modules.

Applying the functor $\shHom_{A}(P, ?)$, we have a quasi-isomorphism $B\ra \shHom_{A}(P, P\ot_{A}DA[d-1])$ of $B^{\rm op}\ot_{k}B$-modules.
Notice that we have an isomorphism 
\[ \shHom_{A}(P, P\ot_{A}DA[d-1]) \cong D\shHom_{A}(P, P)[d-1]=DB[d-1]\]
of $(B^{\rm op}\ot_{k} B)$-modules, then $B\cong DB[d-1]$ in $\D(B)$. Thus $B$ is also $d$-symmetric.
\end{proof}

Now we are ready to prove Theorem \ref{Thm:main}.

\begin{proof}[Proof of Theorem \ref{Thm:main}]
Let $C$ be a $(-d)$-CY configuration of $\Z\Delta/\SSS[d]$.
Let $A=k\Delta\op D(k\Delta)[d-1]$ be the trivial extension dg algebra. We may regard $C$ as a $(-d)$-CY configuration in $\un{\CM}A$.
 There is a SMC $\{X_{1}, \cdots, X_{n}\}$ in $\D^{\bb}(A)$ which is sent to $C$ in $\un{\CM}A$ by Theorem \ref{Thm:SMCsurj} and there exists a silting object $P$ in $\per A$ corresponds to $\{X_{1}, \cdots, X_{n}\}$ by Proposition \ref{Prop:KoY}.
  Let $B=\shEnd_{A}(P)$.
Then $B$ is a $d$-symmetric dg $k$-algebra by Lemma \ref{Lem:siltGoren} and moreover, $\un{\CM}B$ is triangle equivalent to $\un{\CM}A$. By Proposition \ref{Prop:KoY}, the simple modules of $B$ corresponds to $\{X_{1}, \cdots, X_{n}\}$. Then the AR quiver of $\un{\CM}B$ is isomorphic to $(\Z \Delta)_{C}/\SSS[d]$.
 \end{proof}

\section{Maximal $d$-Brauer relations and Brauer tree dg algebras} \label{Section:maximal}
 
In the section, we give a combinatorial 
proof of Theorem \ref{Thm:main} for the case $\Delta=A_{n}$. We will see in $A_{n}$ case, there is a very nice description of $(-d)$-CY configurations by  \emph{maximal $d$-Brauer relations}. 
We develop some technical concepts and results on them.
Then we introduce \emph{Brauer tree dg algebras} from maximal $d$-Brauer relations and we show the simples of such dg algebras correspond to the given CY configurations.

\subsection{Maximal $d$-Brauer relations}
\label{maximalBrauer}

We start with the following definition.

\begin{Def}Let $d\ge 1$ and $n>0$ be two integers and let $N:=(d+1)n+d-1$.
Let $\Pi$ be an $N$-gon with vertices numbered clockwise from $1$ to $N$.
\begin{enumerate}[\rm (1)]
\item A \emph{diagonal} in $\Pi$ is a straight line segment that joins two of the vertices and goes through  the interior of $\Pi$.  The diagonal which joins two vertices  $i$ and $j$ is denoted by $(i,j)=(j,i)$.

\item A \emph{$d$-diagonal} in $\Pi$ is a diagonal of the form $(i, i+d+j(d+1))$, where $0\le j\le n-1$.
\end{enumerate}
\end{Def}

The definition of maximal $d$-Brauer relation is as follows. It is  some special kind of $2$-Brauer relation in the sense of  \cite[Definition 6.1]{Luo}. 

\begin{Def}\label{Def:Brauer relation}
Let $B$ be a set of $d$-diagonals in $\Pi$. We call $B$ a \emph{$d$-Brauer relation of $\Pi$} if any two $d$-diagonals in $B$  are disjoint. We call a $d$-Brauer relation $B$ \emph{maximal}, if it is maximal with respect to inclusions.
\end{Def}

We denote by ${\bf B}$ the set of  maximal $d$-Brauer relations on $\Pi$. Let $\theta$  be the clockwise rotation by $2\pi /N$. If $I=(i_{1},i_{2})$ is a diagonal, then $\theta^{t}(I)=(i_{1}+t, i_{2}+t)$ gives us a new diagonal. For any $B, B'\in \B$, if there exists $n\in \Z$ such that $B=\theta^{n}(B')$, we say $B$
 and $B'$ are \emph{equivalent up to rotation}, denoting by $B\sim B'$. It gives rise to  an equivalence relation on $\B$. We denote by $\un{\bf B}$  the set of equivalence classes of $\B$.
 We give two simple examples to show what the maximal $d$-Brauer relations look like.

\begin{Ex}\label{Ex:d=1&n=2}
Let $d=2$ and $n=2$. Then $N=7$ and $\B$ consists of the following and $\# \un{\B}=1$.
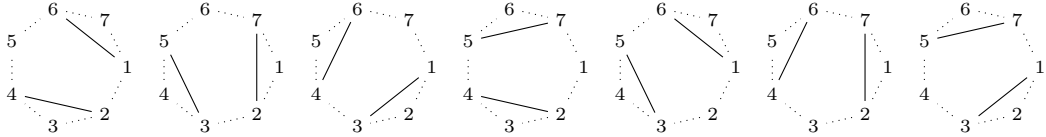
\newdimen \R \R=0.8cm

 \begin{center}
\scriptsize{
   \begin{tikzpicture}
 \draw node (1) at (0:\R){1}
       node (2) at (-360/7:\R){2} 
       node (3) at (-360*2/7:\R){3}
       node (4) at (-360*3/7:\R){4}
       node (5) at (-360*4/7:\R){5}
       node (6) at (-360*5/7:\R){6}
       node (7) at (-360*6/7:\R){7}
       [dotted] (1)--(2)--(3)--(4)--(5)--(6)--(7)--(1);
     \draw (1)--(6) (2)--(4);
 
  \draw[xshift=2.5\R] node (1) at (0:\R){1}
       node (2) at (-360/7:\R){2} 
       node (3) at (-360*2/7:\R){3}
       node (4) at (-360*3/7:\R){4}
       node (5) at (-360*4/7:\R){5}
       node (6) at (-360*5/7:\R){6}
       node (7) at (-360*6/7:\R){7}
       [dotted] (1)--(2)--(3)--(4)--(5)--(6)--(7)--(1);
     \draw (2)--(7) (3)--(5);
     
 \draw[xshift=5\R] node (1) at (0:\R){1}
       node (2) at (-360/7:\R){2} 
       node (3) at (-360*2/7:\R){3}
       node (4) at (-360*3/7:\R){4}
       node (5) at (-360*4/7:\R){5}
       node (6) at (-360*5/7:\R){6}
       node (7) at (-360*6/7:\R){7}
       [dotted] (1)--(2)--(3)--(4)--(5)--(6)--(7)--(1);
     \draw (3)--(1) (4)--(6);
     
  \draw[xshift=7.5\R] node (1) at (0:\R){1}
       node (2) at (-360/7:\R){2} 
       node (3) at (-360*2/7:\R){3}
       node (4) at (-360*3/7:\R){4}
       node (5) at (-360*4/7:\R){5}
       node (6) at (-360*5/7:\R){6}
       node (7) at (-360*6/7:\R){7}
       [dotted] (1)--(2)--(3)--(4)--(5)--(6)--(7)--(1);
     \draw (4)--(2) (5)--(7);

 \draw[xshift=10\R] node (1) at (0:\R){1}
       node (2) at (-360/7:\R){2} 
       node (3) at (-360*2/7:\R){3}
       node (4) at (-360*3/7:\R){4}
       node (5) at (-360*4/7:\R){5}
       node (6) at (-360*5/7:\R){6}
       node (7) at (-360*6/7:\R){7}
       [dotted] (1)--(2)--(3)--(4)--(5)--(6)--(7)--(1);
     \draw (5)--(3) (6)--(1);

 \draw[xshift=12.5\R] node (1) at (0:\R){1}
       node (2) at (-360/7:\R){2} 
       node (3) at (-360*2/7:\R){3}
       node (4) at (-360*3/7:\R){4}
       node (5) at (-360*4/7:\R){5}
       node (6) at (-360*5/7:\R){6}
       node (7) at (-360*6/7:\R){7}
       [dotted] (1)--(2)--(3)--(4)--(5)--(6)--(7)--(1);
     \draw (6)--(4) (7)--(2);

 \draw[xshift=15\R] node (1) at (0:\R){1}
       node (2) at (-360/7:\R){2} 
       node (3) at (-360*2/7:\R){3}
       node (4) at (-360*3/7:\R){4}
       node (5) at (-360*4/7:\R){5}
       node (6) at (-360*5/7:\R){6}
       node (7) at (-360*6/7:\R){7}
       [dotted] (1)--(2)--(3)--(4)--(5)--(6)--(7)--(1);
     \draw (7)--(5) (1)--(3);
\end{tikzpicture} }
\end{center} 
\end{Ex}

\begin{Ex}\label{Ex:d=1&n=3}
Let $d=2$ and $n=3$. Then $N=10$, $\# \B=30$ and $\un{\B}$ consists  of the following.\begin{center}
\newdimen \R \R=1cm
\scriptsize
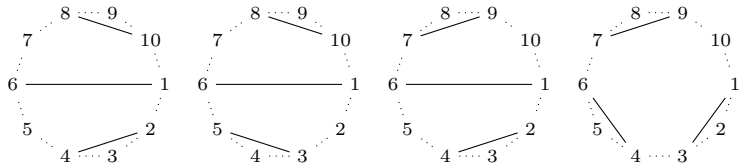
{
\begin{tikzpicture}
 \draw node (1) at (0:\R){1}
       node (2) at (-360/10:\R){2} 
       node (3) at (-360*2/10:\R){3}
       node (4) at (-360*3/10:\R){4}
       node (5) at (-360*4/10:\R){5}
       node (6) at (-360*5/10:\R){6}
       node (7) at (-360*6/10:\R){7}
       node (8) at (-360*7/10:\R){8}
       node (9) at (-360*8/10:\R){9}
       node (10) at (-360*9/10:\R){10}
       [dotted] (1)--(2)--(3)--(4)--(5)--(6)--(7)--(8)--(9)--(10)--(1);
  \draw (1)--(6) (2)--(4) (8)--(10);

\draw[xshift=2.5\R] node (1) at (0:\R){1}
       node (2) at (-360/10:\R){2} 
       node (3) at (-360*2/10:\R){3}
       node (4) at (-360*3/10:\R){4}
       node (5) at (-360*4/10:\R){5}
       node (6) at (-360*5/10:\R){6}
       node (7) at (-360*6/10:\R){7}
       node (8) at (-360*7/10:\R){8}
       node (9) at (-360*8/10:\R){9}
       node (10) at (-360*9/10:\R){10}
       [dotted] (1)--(2)--(3)--(4)--(5)--(6)--(7)--(8)--(9)--(10)--(1);
  \draw (1)--(6) (3)--(5) (8)--(10);
  
\draw[xshift=5\R] node (1) at (0:\R){1}
       node (2) at (-360/10:\R){2} 
       node (3) at (-360*2/10:\R){3}
       node (4) at (-360*3/10:\R){4}
       node (5) at (-360*4/10:\R){5}
       node (6) at (-360*5/10:\R){6}
       node (7) at (-360*6/10:\R){7}
       node (8) at (-360*7/10:\R){8}
       node (9) at (-360*8/10:\R){9}
       node (10) at (-360*9/10:\R){10}
       [dotted] (1)--(2)--(3)--(4)--(5)--(6)--(7)--(8)--(9)--(10)--(1);
  \draw (1)--(6) (2)--(4) (7)--(9);  

\draw[xshift=7.5\R] node (1) at (0:\R){1}
       node (2) at (-360/10:\R){2} 
       node (3) at (-360*2/10:\R){3}
       node (4) at (-360*3/10:\R){4}
       node (5) at (-360*4/10:\R){5}
       node (6) at (-360*5/10:\R){6}
       node (7) at (-360*6/10:\R){7}
       node (8) at (-360*7/10:\R){8}
       node (9) at (-360*8/10:\R){9}
       node (10) at (-360*9/10:\R){10}
       [dotted] (1)--(2)--(3)--(4)--(5)--(6)--(7)--(8)--(9)--(10)--(1);
   \draw (1)--(3) (4)--(6) (7)--(9);
\end{tikzpicture}    }
\end{center}
\end{Ex}

Now we give a description of $(-d)$-Calabi-Yau configurations of type $A_{n}$ by using maximal $d$-Brauer relations. To each vertex in $\Z A_{n}$, we associate a label in $\Z \times \Z$ as follows.

{\tiny
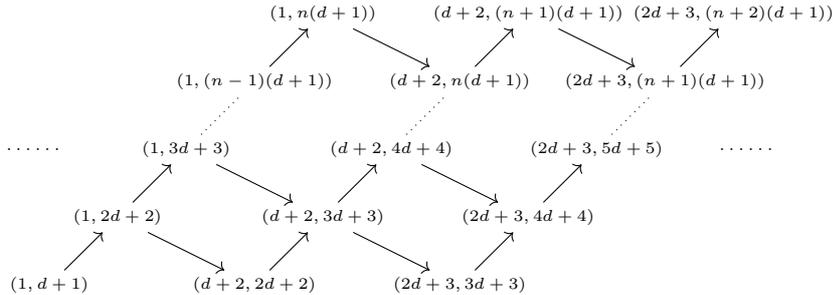
\begin{center}
\begin{tikzpicture}[scale=0.9]
\node at (0,2){$\cdots$};
\node at (-0.42,2){$\cdots$};
\node at (10,2){$\cdots$};
\node at (10.42,2){$\cdots$};
 \node (1d+1) at (0,0){$(1,d+1)$};
 \node (12d+2) at (1,1){$(1,2d+2)$};
 \node (13d+3) at (2,2){$(1,3d+3)$};
\node (1n-1d+1) at (3,3){$(1,(n-1)(d+1))$};
 \node (1nd+1) at (4,4){$(1,n(d+1))$};
 
  \node (d+22d+2) at (3,0){$(d+2,2d+2)$};
   \node (d+23d+3) at (4,1){$(d+2,3d+3)$};
  \node (d+24d+4) at (5,2){$(d+2,4d+4)$};
  \node (d+2nd+1) at (6,3){$(d+2,n(d+1))$};
  \node (d+2n+1d+1) at (7,4){$(d+2,(n+1)(d+1))$}; 
\node (2d+33d+3) at (6, 0){$(2d+3, 3d+3)$}; 
\node (2d+34d+4) at (7, 1){$(2d+3, 4d+4)$}; 
\node (2d+35d+5) at (8, 2){$(2d+3, 5d+5)$}; 
\node (2d+3n+1d+1) at (9, 3){$(2d+3, (n+1)(d+1))$}; 
\node (2d+3n+2d+1) at (10, 4){$(2d+3, (n+2)(d+1))$}; 
  
  \draw[->] (1d+1)--(12d+2) ;
 \draw[->] (12d+2)--(13d+3);
 \draw[dotted] (13d+3)--(1n-1d+1);
 \draw[->](1n-1d+1)--(1nd+1);

\draw[->](d+22d+2)--(d+23d+3);
\draw[->](d+23d+3)--(d+24d+4);
\draw[dotted](d+24d+4)--(d+2nd+1);
\draw[->](d+2nd+1)--(d+2n+1d+1);

\draw[->](12d+2)--(d+22d+2);
\draw[->](13d+3)--(d+23d+3);
\draw[->](1nd+1)--(d+2nd+1); 

\draw[->] (2d+33d+3)--(2d+34d+4);
\draw[->](2d+34d+4)--(2d+35d+5) ;
\draw[dotted](2d+35d+5)-- (2d+3n+1d+1);
\draw[->] (2d+3n+1d+1)--(2d+3n+2d+1);
\draw[->] (d+23d+3)-- (2d+33d+3);
\draw[->] (d+24d+4)-- (2d+34d+4);
\draw[->] (d+2n+1d+1)-- (2d+3n+1d+1);

 \end{tikzpicture}
\end{center} }
Let $\Z A_{n,d}$ be the stable translation quiver $\Z A_{n}/\SSS[d]$.
Since by the labelling above, $\SSS[d]$ sends $(i,j)$ to $(i+N,j+N)$ if $d$ is odd, and to $(j+N, i+N)$ if $d$ is even, then we may label $\Z A_{n,d}$ by taking the labelling in $\Z/N\Z \times \Z/N\Z$, where we identify $(i,j)$ and $(j,i)$. Let us see some examples.

 \begin{Ex}
  \begin{enumerate}[\rm (1)]
  \item
 Let $d=1$ and $n=4$. In this case, the labelling on $\Z_{4,0}$ is as follows. 
 {\tiny
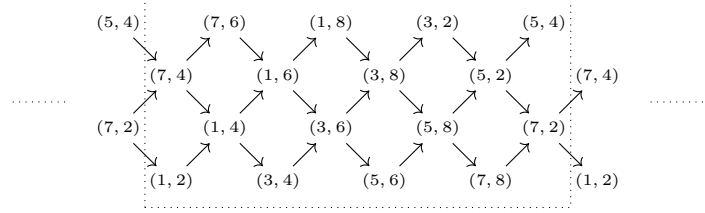
\begin{center}
\begin{tikzpicture}[scale=0.7]
\node (12) at (0,0) {$(1,2)$}; 
\node (34) at (2,0) {$(3,4)$};
\node (56) at (4,0) {$(5,6)$}; 
\node (78) at (6,0) {$(7,8)$}; 
\node (+12) at (8,0) {$(1,2)$};

\node (-72) at (-1,1) {$(7,2)$}; 
\node (14) at (1,1) {$(1,4)$}; 
\node (36) at (3,1) {$(3,6)$}; 
\node (58) at (5,1) {$(5,8)$};
\node (72) at (7,1) {$(7,2)$};  

\node (74) at (0,2) {$(7,4)$}; 
\node (16) at (2,2) {$(1,6)$};
\node (38) at (4,2) {$(3,8)$};
\node (52) at (6,2) {$(5,2)$};
\node (+74) at (8,2) {$(7,4)$};

\node (-54) at (-1,3) {$(5,4)$};
\node (76) at (1,3) {$(7,6)$};
\node (18) at (3,3) {$(1,8)$};
\node (32) at (5,3) {$(3,2)$};
\node (54) at (7,3) {$(5,4)$};

\draw[->] (12)edge(14) (14)edge(16) (16)edge(18)
(18)edge(38) (38)edge(58) (58)edge(78) 
(34)edge(36)  (36)edge(38) (38)edge(32) (32)edge(52) (52)edge(72)  (72)edge(+12) (-72)edge(12) (-72)edge(74) (74)edge(76) (76)edge(16) (16)edge(36) (36)edge(56) (-54)edge(74) (74)edge(14) (14)edge(34)
(56)edge(58) (58)edge(52) (52)edge(54)
(78)edge(72) (72)edge(+74);  
\draw[dotted] (-2,1.5)--(-3,1.5) (9,1.5)--(10,1.5);
\draw[dotted] (-0.5,-0.5)--(-0.5,3.5)--(7.5,3.5)--(7.5,-0.5)--(-0.5,-0.5);
\end{tikzpicture}
\end{center} }

\item Let $d=2$ and $n=4$. Then the labelling on $\Z_{4,1}$ is as follows.

{\tiny
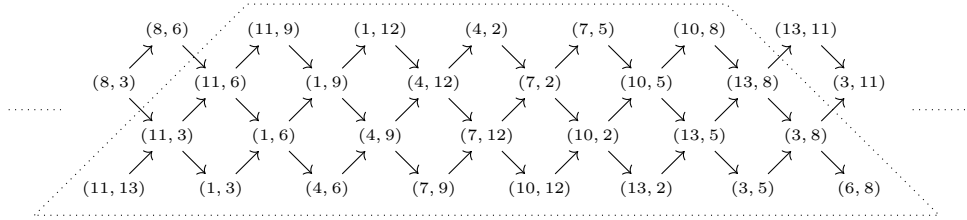
\begin{center}
\begin{tikzpicture}[scale=0.7]
\node (-78) at (-2,0) {$(11,13)$};
\node (12) at (0,0) {$(1,3)$}; 
\node (34) at (2,0) {$(4,6)$};
\node (56) at (4,0) {$(7,9)$}; 
\node (78) at (6,0) {$(10,12)$}; 
\node (+12) at (8,0) {$(13,2)$};
\node (+34) at (10,0) {$(3,5)$};
\node (+56) at (12,0) {$(6,8)$};

\node (-72) at (-1,1) {$(11,3)$}; 
\node (14) at (1,1) {$(1,6)$}; 
\node (36) at (3,1) {$(4,9)$}; 
\node (58) at (5,1) {$(7,12)$};
\node (72) at (7,1) {$(10,2)$};  
\node (+14) at (9,1) {$(13,5)$};
\node (+36) at (11,1) {$(3,8)$};

\node(-52) at (-2,2) {$(8,3)$};
\node (74) at (0,2) {$(11,6)$}; 
\node (16) at (2,2) {$(1,9)$};
\node (38) at (4,2) {$(4,12)$};
\node (52) at (6,2) {$(7,2)$};
\node (+74) at (8,2) {$(10,5)$};
\node (+16) at (10,2) {$(13,8)$};
\node (+38) at (12,2) {$(3,11)$};

\node (-54) at (-1,3) {$(8,6)$};
\node (76) at (1,3) {$(11,9)$};
\node (18) at (3,3) {$(1,12)$};
\node (32) at (5,3) {$(4,2)$};
\node (54) at (7,3) {$(7,5)$};
\node (+76) at (9,3) {$(10,8)$};
\node (+18) at (11,3) {$(13,11)$};

\draw[->] (12)edge(14) (14)edge(16) (16)edge(18)
(18)edge(38) (38)edge(58) (58)edge(78) 
(34)edge(36)  (36)edge(38) (38)edge(32) (32)edge(52) (52)edge(72)  (72)edge(+12) (-72)edge(12) (-72)edge(74) (74)edge(76) (76)edge(16) (16)edge(36) (36)edge(56) (-54)edge(74) (74)edge(14) (14)edge(34)
(56)edge(58) (58)edge(52) (52)edge(54)
(78)edge(72) (72)edge(+74);  
\draw[->] (+74)edge(+76) (+76)edge(+16) (+16)edge(+36) (+36)edge(+56)
(+12)edge(+14) (+14)edge(+16) (+16)edge(+18) (+18)edge(+38)
(54)edge(+74) (+74)edge(+14) (+14)edge(+34) (+34)edge(+36) (+36)edge(+38)
(-52)edge(-72) (-52)edge(-54) (-78)edge(-72);

\draw[dotted] (-3,1.5)--(-4,1.5) (13,1.5)--(14,1.5);
\draw[dotted] (-3.5,-0.5)--(0.5,3.5)--(9.5,3.5)--(13.5, -0.5)--(-3.5,-0.5);
\end{tikzpicture}
\end{center} }
\end{enumerate}
\end{Ex}

By the labelling above, we have the following theorem. This result has been show in \cite{CS}, we put a new proof in Appendix by using concepts developed here.  Let ${\bf C}$ be the set of $(-d)$-CY configurations in $\Z A_{n,d}$.
 \begin{Thm}\cite[Theorem 6.5]{CS}\label{Thm:bijection} 
   \begin{enumerate}[\rm(1)]
   \item
   There is a bijection between the vertices of $\Z A_{n,d}$  and the $d$-diagonals in $\Pi$  sending the vertex $(i,j)$ of $\Z A_{n,d}$ to the diagonal $(i,j)$ of $\Pi$.

   \item
The bijection  in {\rm (1)} gives a bijection between ${\bf C}$ and ${\bf B}$;
\item Any  $(-d)$-CY configuration in $\Z A_{n,d}$ contains exactly $n$ elements.
\end{enumerate}
\end{Thm}

 We give an example to show how the bijection works.
\begin{Ex}
Let $n=2$ and $d=2$. We associate to each vertex of $\Z A_{2,1}$ a label in $\Z/7\Z\times \Z/7\Z$ as following:
{\tiny
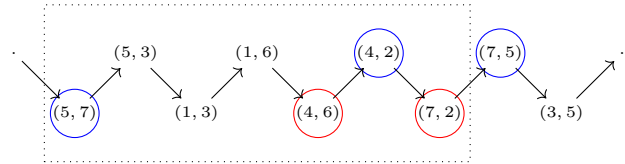
\begin{center}
\begin{tikzpicture}[scale=0.8]
 \node (k) at (0,1){$(1,6)$};
 \node (a21r) at (2,1){$(4,2)$};
 \node (a2r) at (4,1){$(7,5)$};
 \node (k2r) at (6,1){$.$};
 \node (k1l) at (-2,1){$(5,3)$};
 \node (k2l) at (-4,1){$.$};

 \node (a21l) at (-3,0){$(5,7)$};
 \node (a2l) at (-1,0){$(1,3)$};
 \node (k2) at (1,0){$(4,6)$};
 \node (k1r) at (3,0){$(7,2)$};

 \node (kr) at (5,0){$(3,5)$};
 \draw[red] (1,0) circle [radius=0.4];
\draw [red] (3,0) circle [radius=0.4];
 \draw[blue] (2,1) circle [radius=0.4];
 \draw[blue] (-3,0) circle [radius=0.4];
 \draw[blue] (4,1) circle [radius=0.4];

\draw[->] (k2l)--(a21l);
\draw[->] (a21l)--(k1l);
\draw[->] (k1l)--(a2l);
\draw[->] (a2l)--(k);
\draw[->] (k)--(k2);
\draw[->] (k2)--(a21r);
\draw[->] (a21r)--(k1r);
\draw[->] (k1r)--(a2r);
\draw[->] (a2r)--(kr);
\draw[->] (kr)--(k2r);
\draw[dotted] (-3.5,-0.8) rectangle (3.5,1.8);

\end{tikzpicture}
\end{center}}

It is easy to check the set $\{ (4,6), (7,2) \}$ is a $(-2)$-CY configuration of $\Z A_{2, 1}$ and it gives rise to a maximal $d$-Brauer relation of $7$-gon as follows (left part):
{\tiny
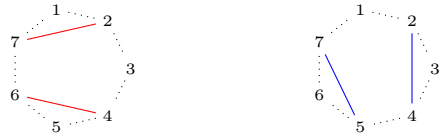
\begin{center}
 \newdimen \R \R=0.8cm
  \begin{tikzpicture}
 \draw
 node (2) at (360/7:\R) {$2$}
 node (1) at (720/7:\R) {$1$}
 node (7) at (1080/7:\R) {$7$}
node (6) at (1440/7:\R) {$6$}
node (5) at (360*5/7:\R) {$5$}
node (4) at (360*6/7:\R) {$4$}
node (3) at (360*7/7:\R) {$3$}
[dotted] (1)--(2)--(3)--(4)--(5)--(6)--(7)--(1);
\draw[red] (7)--(2) (4)--(6);

\draw[xshift=5\R]
 node (2) at (360/7:\R) {$2$}
 node (1) at (720/7:\R) {$1$}
 node (7) at (1080/7:\R) {$7$}
node (6) at (1440/7:\R) {$6$}
node (5) at (360*5/7:\R) {$5$}
node (4) at (360*6/7:\R) {$4$}
node (3) at (360*7/7:\R) {$3$}
[dotted] (1)--(2)--(3)--(4)--(5)--(6)--(7)--(1);
\draw[blue] (5)--(7) (4)--(2);
\end{tikzpicture}
\end{center}}
On the other hand, any maximal $d$-Brauer relation, for example $\{(2,4), (7,5) \}$, gives us a $(-2)$-CY configuration in $\Z A_{2,1}$.
\end{Ex}
 
The following lemma is immediately from the definition.
 \begin{Lem}\label{theta}
 Let $X\in \Z A_{n,d}$ with labelling $(x_{1},x_{2})$. Then $X[1]=(x_{2}+1,x_{1}+1)$ and $X[1]=\theta(X)$  as  $d$-diagonals.
 \end{Lem}

In the rest of this subsection, we introduce some technical concepts and results. They give us a better understanding of maximal $d$-Brauer relations and in particular, Proposition \ref{Prop:B=B'} will play a crucial role in the proof of  Theorem \ref{Thm:section}.

\begin{Def}\label{Def:Bcycle}
\begin{enumerate} [\rm(1)]
\item
Let $C$ be a set of diagonals in $\Pi$. We call $C$ a \emph{cycle} if $C$ is contained in the closure of  some connect component (denoted by $\Pi_{C}$) of the subset $\Pi \backslash \bigcup_{X\in C}X$ of $\Pi$. In this case, elements in $C$ has a anti-clockwise  ordering $C=\{ X_{1}, \cdots, X_{s}\}$ given as follows.
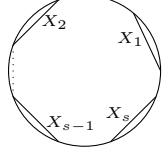
\newdimen\R
\R=1cm
{\tiny
\begin{center} 
\begin{tikzpicture}
 \draw[xshift=0\R] circle (\R);
 \draw[xshift=0\R]     
   (0:\R)--(50:\R) (110:\R)--(160:\R)     (200:\R)--(250:\R) (290:\R)--(340:\R)
    node[left=1]  at (25:\R){$X_{1}$}
       node[right=3]  at (140:\R){$X_{2}$}
      node[right=3]  at (225:\R){$X_{s-1}$}
   node[left=3]  at (325:\R){$X_{s}$};
    \draw[xshift=0\R, dotted] (160:\R)--(200:\R);
    \end{tikzpicture} 

\end{center} }

\item
Let $B\in \B$ and $C\subset B$. We call $C$ a \emph{$B$-cycle} if $C$ is a cycle and $\{X\in \B\mid X\in \overline{\Pi_{C}} \}=C$. \end{enumerate}
\end{Def}

\begin{Ex}
Let $d=2$ and $n=4$. Let $B$  be the following maximal $d$-Brauer relation.
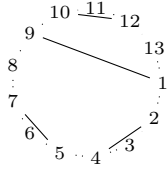
\begin{center}
\newdimen \R \R=1cm
\scriptsize{
\begin{tikzpicture}
 \draw node (1) at (0:\R){1}
       node (2) at (-360/13:\R){2} 
       node (3) at (-360*2/13:\R){3}
       node (4) at (-360*3/13:\R){4}
       node (5) at (-360*4/13:\R){5}
       node (6) at (-360*5/13:\R){6}
       node (7) at (-360*6/13:\R){7}
       node (8) at (-360*7/13:\R){8}
       node (9) at (-360*8/13:\R){9}
       node (10) at (-360*9/13:\R){10}
       node (11) at (-360*10/13:\R){11}
       node (12) at (-360*11/13:\R){12}
       node (13) at (-360*12/13:\R){13}
       [dotted] (1)--(2)--(3)--(4)--(5)--(6)--(7)--(8)--(9)--(10)--(11)--(12)--(13)--(1);
       \draw (1)--(9) (10)--(12) (2)--(4) (5)--(7);
\end{tikzpicture} }
\end{center}
By the definition above, $C=\{ (2,4), (5,7), (10,12) \}$ is a cycle but not a $B$-cycle  and $C'=\{ (2,4), (5,7), (1,9)\}$ is a $B$-cycle. 
\end{Ex}

Here are some elementary properties of these concepts. The  proof is  left to the reader.
\begin{Prop} \label{Prop:cycleproperty}
Let $B\in \B$, then
 \begin{enumerate}[\rm(1)]
 \item  $B$ is the union of $B$-cycles;
 \item Any two $B$-cycles have at most one common diagonal;
 \item Let $C:=\{  X_{1}, \ldots, X_{s} \}$ be a set of diagonals in $\Pi$. Let $X_{s+1}:=X_{1}$. Then $C$ is a cycle if and only if  for any $i$, $2\le i\le s$, $X_{i-1}$ and $X_{i+1}$ are in the same connect component of $\Pi\backslash X_{i}$；
 \item Let $X, Y\in B$. Then $X$ and $Y$  are in the same $B$-cycle  if and only if  for any $Z\not=X,Y$ in $B$, $X$ and $Y$ are in the same connected component part of $\Pi\backslash Z$；
 \item Let $X,Y,Z\in B$. $X$ and $Y$ are in the same connected component of $\Pi\backslash Z$ if and only if there is a sequence $X=X_{1}, X_{2}, \cdots, X_{t}=Y$ of $B$, such that $Y\not=X_{i}$, $1\le i\le t$ and $X_{j}, X_{j+1}$ are in the same $B$-cycle for $1\le j\le t-1$.
 \end{enumerate}
\end{Prop}

We give an easy observation.

\begin{Lem}\label{Lem:l}
Let $B\in\B$ and $X\in B$. 
Let $\Pi_{1}$ and $\Pi_{2}$ be two connect components of $\Pi\backslash X$. Then
 \begin{enumerate}[\rm (1)]
  \item
  $B\cap \Pi_{i}:=\{ Y\in B\mid Y\subset \Pi_{i}\}$ is a maximal $d$-Brauer relation of $\Pi_{i}$ for $i=1,2$;
  \item If $X$ has the form $(i, i+d+1+(d+2)j)$, then $\{ \#B\cap \Pi_{1}, \#B\cap\Pi_{2} \}= \{ j, n-j-1\}$.
 \end{enumerate} 
\end{Lem}

Let $X$ and $Y$ be two disjoint $d$-diagonals. We denote by  $\delta(X,Y)$ the smallest positive integer $m$ such that $\theta^{-m}(X)\cap Y\not=\emptyset$.

\begin{Rem}\label{delta}
Let $X, Y\in \Z A_{n,d}$. If $X$ and $Y$ are disjoint as $d$-diagonals, then $\delta (X, Y)= {\rm min}\{ i>0 \mid \bar{h}_{X}(Y[i])\not=0 \}$ by Lemma \ref{theta}.
\end{Rem}

We will give a description of $B$-cycles by $\delta$. Before this, we show a lemma. Let $\mathfrak S_{s}$ be the permutation group.

\begin{Lem}\label{Lem:cycle}
Let $B\in \B$ and let $C\subset B$ be a cycle with anti-clockwise ordering $\{ X_{1}, \cdots ,X_{s} \}$. Let $\Pi_{C}$ be the  connect component of $\Pi\backslash C$ given in Definition \ref{Def:Bcycle}.
Let $X_{s+1}=X_{1}$.
Then the following statement holds.
\begin{enumerate}[\rm(1)]
\item Let $m= \# (B\cap \Pi_{C})$, then $\sum_{l=1}^{s}\delta(X_{l},X_{l+1})=d+s+(d+1)m-1$;
\item For any $\tau\in \mathfrak S_{s}$, we have
$\sum_{l=1}^{s}\delta(X_{\tau(l)}, X_{\tau(l+1)})\ge d+s+(d+1)m-1$. Moreover, the equality holds if and only if $\tau(l+1)=\tau(l)+1$ for all $1\le l \le s$;
\item  $C$ is a $B$-cycle if and only if  $\sum_{l=1}^{s}\delta(X_{l},X_{l+1})=d+s-1$.
 \end{enumerate}
\end{Lem}

\begin{proof}
(1)
Assume $X_{i}$ has the form $(x_{i},y_{i})$ as follows, where $y_{i}=x_{i}+d+(d+1)j_{i}$ with $0\le j_{i}\le n-1$.
\newdimen\R
\R=1cm
{\tiny
\begin{center} 
\begin{tikzpicture}
 \draw[xshift=0\R] circle (\R);
 \draw[xshift=0\R]     
   (0:\R)--(50:\R) (110:\R)--(160:\R)     (200:\R)--(250:\R) (290:\R)--(340:\R)
   node at (0,0){$\Pi_{C}$}
    node[left=1]  at (25:\R){$X_{1}$}
    node[right] at (0:\R){$y_{1}$}
    node[right] at (50:\R){$x_{1}$}
    node[above] at (110:\R){$y_{2}$}
    node[left] at (160:\R){$x_{2}$}
   node[right=3]  at (140:\R){$X_{2}$}
   node[left] at (200:\R){$y_{s-1}$}
   node[below] at (250:\R){$x_{s-1}$}
   node[below] at (290:\R){$y_{s}$}
   node[right] at (340:\R){$x_{s}$}
   node[right=3]  at (225:\R){$X_{s-1}$}
   node[left=3]  at (325:\R){$X_{s}$};
    \draw[xshift=0\R, dotted] (160:\R)--(200:\R);
    \end{tikzpicture} 
\end{center} }
Since by definition, $\delta(X_{i},X_{i+1})=1+ \mbox{ the number of vertices between $X_{i}$ and $X_{i+1}$ (anti-clockwise)}$,  then  $\sum_{l=1}^{s}\delta(X_{l},X_{l+1})=s+\#\Pi_{C}$.
We count $\Pi\backslash\Pi_{C}$ first. By our labelling, it is easy to see the number of vertices in $ \Pi\backslash\Pi_{C}$ is $\sum_{i=1}^{s}(d+1)(j_{i}+1)$. Then  $\sum_{l=1}^{s}\delta(X_{l},X_{l+1})=s+d+(d+1)(n-\sum_{i}^{s}(j_{i}+1))-1$.
By Lemma \ref{Lem:l} (2), $\#B\cap(\Pi\backslash \Pi_{C})=\sum_{i=1}^{s}(j_{i}+1)$. Then by Theorem \ref{Thm:bijection} (3), $m=\#B\cap \Pi_{C}=n-\sum_{i}^{s}(j_{i}+1)$. Thus $\sum_{l=1}^{s}\delta(X_{l},X_{l+1})=d+s+(d+1)m-1$.

(2) For any $\tau\in \mathfrak S_{n}$, we have $\delta(X_{\tau(l)}, X_{\tau(l+1)})\ge \delta(X_{\tau(l)}, X_{\tau(l)+1})$ and the equality holds if and only if $\tau(l+1)=\tau(l)+1$. Then by (1), 
\[\sum_{l=1}^{s}\delta(X_{\tau(l)}, X_{\tau(l+1)})\ge \sum_{l=1}^{s}\delta(X_{\tau(l)}, X_{\tau(l)+1})=d+s+(d+1)m-1.\]

(3) By Definition \ref{Def:Bcycle}, $C$ is a $B$-cycle if and only if $m=0$, then by (1), it holds if and only if   $\sum_{l=1}^{s}\delta(X_{l},X_{l+1})=d+s-1$.
\end{proof}

 The following proposition gives us a useful  criterion for  being $B$-cycle. 
\begin{Prop}\label{Prop:cycle}
Let $B\in \B$ and let $C$ be a subset of $B$. Then $C$ is a $B$-cycle if and only if  there is a numbering  $C=\{X_{1},\dots,X_{s}\}$ such that  $\sum_{l=1}^{s}\delta(X_{l}, X_{l+1})=d+s-1$, where $X_{s+1}=X_{1}$. In this case, $\{X_{1},\dots,X_{s}\}$ is an anti-clockwise ordering or $C$.
\end{Prop}

\begin{proof}
The ``only if'' part. Assume $C$ is a $B$-cycle with anti-clockwise ordering $\{X_{1}, \cdots, X_{l}\}$, then by Lemma \ref{Lem:cycle} (1), $\sum_{l=1}^{s}\delta(X_{l}, X_{l+1})=d+s-1$.

The ``if'' part. To prove $C$ is a $B$-cycle, it suffices to show $C$ is a cycle by Lemma \ref{Lem:cycle} (2) and (3). If it is not true, then by Proposition \ref{Prop:cycleproperty} (3),  there exists some $i$, $2\le i\le s$, such that  $X_{i-1}$ and $X_{i+1}$ are  in  different connect components of $\Pi\backslash X_{i}$ as follows. 
\newdimen\R
\R=0.8cm
{\tiny
\begin{center}
\begin{tikzpicture}
 \draw[xshift=0\R] circle (\R);
 \draw[xshift=0\R]     
   (0:\R)--(50:\R)   (90:\R)--(270:\R)
   (200:\R)--(250:\R)     node[right]  at (25:\R){$X_{i+1}$}
    node[right]  at ($(90:\R)!0.5!(270:\R)$){$X_{i}$}
   node[below]  at (225:\R){$X_{i-1}$};
   \end{tikzpicture}
\end{center} }
In this case, we have 
\begin{eqnarray}\label{xixi+1}
\delta(X_{i-1}, X_{i})+\delta(X_{i}, X_{i+1})=\delta(X_{i-1},X_{i+1}).
\end{eqnarray}
Now consider the new set $C':=C\backslash X_{i}$. If it is a cycle, then it is clear that $X_{i}\in B\cap \Pi_{C'}$, where $\Pi_{C'}$ is the connected component given in Definition \ref{Def:Bcycle} (1). Then $\# B\cap \Pi_{C'}\ge1$ and by Lemma \ref{Lem:cycle} (2), the following inequality holds.
 \begin{eqnarray}\label{xi-1}
 \sum_{l=1}^{i-2}\delta(X_{l},X_{l+1})+\delta(X_{i-1},X_{i+1})+ \sum_{l=i+1}^{s}\delta(X_{l},X_{l+1}) \ge 2d+s-1. \end{eqnarray}
Notice that by equation \eqref{xixi+1}, the left hand of \eqref{xi-1} equals $d+s-1$. Then $d+s-1\ge 2d+s-1$, a contradiction. If $C'$ is not a cycle, we do the same thing on $C'$ as on $C$, and after finite steps, we get a contradiction. 
Thus $C$ is a cycle, therefore a $B$-cycle.
\end{proof}

Let $B\in\B$. Then $B$ is determined by $B$-cycles in the following sense.

\begin{Prop}\label{Prop:B=B'}
Let $B, B'\in \B$ and let $\phi: B\ra B'$ be a bijective map. If for any $B$-cycle $C$ with anti-clockwise ordering $C=\{ X_{1}, \cdots, X_{s}\}$, we have 
$\delta(X_{i},X_{i+1})=\delta(\phi(X_{i}), \phi(X_{i+1}))$.
 Then $\phi$ is the restriction of $\theta^{n}$ for some integer $n$, that is, $B$ is isomorphic to $B'$ up to rotation.
\end{Prop}
To prove this proposition, we need prepare several lemmas.

\begin{Lem} \label{Lem:Bcycle-B'cycle}
Let $B,B'$ and $\phi$ as above.
Let $C=\{X_{1},\cdots,X_{s}\}$ be a subset of $B$. The following are equiavlent.
 \begin{enumerate}[\rm (1)]
  \item $C=\{X_{1}, \cdots, X_{s}\}$ is a $B$-cycle with anti-clockwise ordering;
  \item $\phi(C)=\{ \phi(X_{1}), \cdots, \phi(X_{s})\}$ is a $B'$-cycle with anti-clockwise ordering.
 \end{enumerate}
\end{Lem}
\begin{proof}
(1) to (2). If $C=\{X_{1}, \cdots, X_{s}\}$ is a $B$-cycle with anti-clockwise ordering, then  
\[ \sum_{i=1}^{s}\delta(\phi(X_{i}), \phi(X_{i+1}))=\sum_{i=1}^{s}\delta(X_{i}, X_{i+1})=d+s-1.\]
Then by Proposition \ref{Prop:cycle}, $\phi(C)=\{ \phi(X_{1}), \cdots, \phi(X_{s})\}$ is a $B'$-cycle with anti-clockwise ordering.

(2) to (1). It is suffices to show if $\phi(X_{i})$ and $\phi(X_{j})$ are in the same $B'$-cycle, then so are $X_{i}$ and $X_{j}$. If it is not true, then by Proposition \ref{Prop:cycleproperty}, there exists $Y\in B$, such that $X_{i}$ and $X_{j}$ are in  different connected component of $\Pi\backslash Y$, which is contradict to Lemma \ref{Lem:connectcomp} below.
\end{proof}

\begin{Lem}\label{Lem:connectcomp}
Let $X, Y, Z \in B$. Then $X,Y$ are in the same connected component of $\Pi\backslash Z$ if and only if $\phi(X)$ and $\phi(Y)$ are in the same connected component of $\Pi\backslash \phi(Z)$.
\end{Lem}
\begin{proof}
It is immediately from Proposition \ref{Prop:cycleproperty} (5) and Lemma \ref{Lem:Bcycle-B'cycle} (1) to (2) part.
\end{proof}

\begin{proof}[Proof of Proposition \ref{Prop:B=B'}]
We first show for any $X$ in $B$, $X$ and $\phi(X)$ have the same length.
Let $X=(x, x+d+(d+1)j)\in B$, $1\le j\le n-1$. Let $\Pi_{1}$ and $\Pi_{2}$ be the connected components of $\Pi\backslash X$.    By Lemma \ref{Lem:l}, $j$ is determined by the set $\{ \#B\cap\Pi_{1}, \#B\cap\Pi_{2}\}$. Since by Lemma \ref{Lem:connectcomp}, $\{ \#B\cap\Pi_{1}, \#B\cap\Pi_{2}\}= \{ \#B'\cap\Pi_{1}', \#B'\cap\Pi_{2}'\}$, then $\phi(X)$ has the form $(x', x'+d+(d+1)j)$, where $\Pi'_{1}$ and $\Pi'_{2}$ are the connected components of $B'\backslash \phi(X)$. So there is an integer $n$, such that $\phi(X)=\theta^{n}(X)$.

We claim $\theta^{n}(B)=B'$. Let $C=\{ X=X_{1}, \cdots, X_{s}\}$ be a B-cycle. Since $\delta(X_{i},X_{i+1})=\delta(\phi(X_{i}),\phi(X_{i+1}))$, then $\phi(C)=\theta^{n}(C)$. For any $Y\in B$, $Y$ and $X$ are connected by a series of $B$-cycles, thus $\theta^{n}(B)=\phi(B)$ holds.
\end{proof}

\subsection{The cardinality of maximal  $d$-Brauer relations}
\label{appendixB}

In this section, we compute the cardinality of maximal Brauer relations.  Let $d\ge 1$ and $n>0$ be two integers. Let $\Pi$ be a $((d+1)n+d-1)$-gon. Recall we denote by $\B$ the set of maximal $d$-Brauer relations on $\Pi$. We have the following theorem.

\begin{Thm}\label{Thm:n}
$\#\B=\frac{1}{n+1}\binom{(d+1)n+d-1}{n}$.
\end{Thm}

\begin{Cor}\label{Cor:n}
There are $\frac{1}{n+1}\binom{(d+1)n+d-1}{n}$ different $(-d)$-CY configurations in $\Z A_{n,d}$.
\end{Cor}

\begin{Rem}
For the classical case $d=1$, the cardinality of Riedtmann's configurations in $\Z A_{n,d}$ is given by the positive Fuss-Catalan number (see \cite[Corollary 5.8]{CS0}).
\end{Rem}
Let $\mathbf V:=\{ \text{ subset $V$ of vertices of $\Pi$ such that $\# V=n$ } \}$. Then the cardinality of $\bf V$ is $\binom{(d+1)n+d-1}{n}$. The main idea of the proof of Theorem \ref{Thm:n} is to construct a surjective map from $\bf V$ to $\B$.
For any $V\in \bf V$, to construct a maximal $d$-Brauer relation corresponds to $V$, we need the following observation.

\begin{Lem}\label{Lem:map1}
Let $V=\{ v_{1}, \dots, v_{n}\} \in \bf V$. Then for any $v_{i}\in V$, there exists a $d$-diagonal with the form $(v_{i}, v_{i}+d+(d+1)a_{i})$, $0\le a_{i} \le n-1$, such that 
\[ \#\{ v\in V \mid v_{i}<v < v_{i}+d+(d+1)a_{i}\}=a_{i} \]
and $ v_{i}+d+(d+1)a_{i}\not \in V$.
\end{Lem}
\begin{proof}
Let $b_{i}\in \{ 0,1,2,\dots, n-1\}$ be the biggest  number such that $v_{i}+d+(d+1)b_{i}\not\in V$. Since $\# V=n$, then 
\[ \#\{ v\in V \mid v_{i}<v < v_{i}+d+(d+1)b_{i}\}\le b_{i}. \]
On the other hand, we have 
\[ \#\{ v\in V \mid v_{i}<v < v_{i}+d\}\ge 0. \]
So there exists $0\le a_{i}\le b_{i}$ satisfies our conditions.
\end{proof}

 For any $v_{i}\in V$, let $J_{v_{i}}=(v_{i},w_{i}:=v_{i}+d+(d+1)a_{i})$ be the $d$-diagonal such that $a_{i}$ is the smallest number satisfies the conditions in Lemma \ref{Lem:map1}.  We have the following result.
 \begin{Prop}\label{Prop:map}
 Let $V=\{ v_{1}, \dots, v_{n}\}\in \bf V$. Then 
$ \{J_{v_{1}},\dots, J_{v_{n}} \}$ defined above is a maximal $d$-Brauer relation on $\Pi$.
 \end{Prop}
 
 Before prove this proposition, we give some basic properties of $J_{v_{i}}$ first.
 
\begin{Lem} \label{Lem:map2}
Let $J_{v_{i}}=(v_{i}, w_{i}=v_{i}+d+(d+1)a_{i})$ defined as above, then
\begin{enumerate}[\rm (1)] 
 \item For any $0\le c_{i}<a_{i}$, we have
  \[ \#\{ v\in V \mid v_{i}<v \le v_{i}+d+(d+1)c_{i}\}> c_{i}. \]
  \item   \[ \#\{v\in V \mid v_{i}+d+(d+1)(a_{i}-1)<v < v_{i}+d+(d+1)a_{i}\}=0. \]
 \end{enumerate}
\end{Lem}
\begin{proof}
(1) If $\#\{ v\in V \mid v_{i}<v \le v_{i}+d+(d+1)c_{i}\}\le c_{i}$, then we can find $0\le d_{i}\le c_{i}$, such that $d_{i}$ satisfies the conditions in Lemma \ref{Lem:map1}, it contradicts the minimality of $a_{i}$. Then the assertion is true.

(2) By (1), we have  $\#\{ v\in V \mid v_{i}<v \le v_{i}+d+(d+1)(a_{i}-1)\}> a_{i}-1$. On the other hand, 
$\#\{ v\in V \mid v_{i}<v < v_{i}+d+(d+1)a_{i}\}=a_{i}$, then the statement holds clearly.
\end{proof}

\begin{proof}[Proof of Proposition \ref{Prop:map}]
By Theorem \ref{Thm:bijection}, we only need to show that any two diagonals in $ \{J_{v_{1}},\dots, J_{v_{n}} \}$ are disjoint. 
Let $v_{i}, v_{j} \in V$.  If neither $v_{j}<v_{i}<w_{j}$ nor $v_{i}<v_{j}<w_{i}$ holds, then it is clear $J_{v_{i}}$ and $J_{v_{j}}$ are disjoint. Otherwise, we may assume $v_{j}<v_{i}<w_{j}$. It suffices to show $v_{j}<u_{i}<w_{j}$. We consider the following two cases.

If $v_{j}+d+(d+1)b_{j}<v_{i}\le v_{j}+d+(d+1)(b_{j}+1)$, for some $0\le b_{j}< a_{j}$. By Lemma \ref{Lem:map2} (2), we know that $b_{j}+1\le a_{j}-1$. Consider the diagonal $(v_{i}, v_{i}+d+(d+1)(a_{j}-b_{j}-2))$, we claim that 
\[ \#\{ v\in V \mid v_{i}<v \le v_{i}+d+(d+1)(a_{j}-b_{j}-2)\}\le a_{j}-b_{j}-2. \]
Indeed by Lemma \ref{Lem:map2} (1),  
  \[ \#\{ v\in V \mid v_{j}<v \le v_{j}+d+(d+1)b_{j}\}> b_{j}, \]
and by the definition of $w_{j}$,
\[ \#\{ v\in V \mid v_{j}<v < v_{j}+d+(d+1)a_{i}\}=a_{i}. \]
Then $\#\{ v\in V \mid v\not=v_{i}  \ {\rm and}\ v_{j}+d+(d+1)b_{j}<v < w_{j}\}\le a_{j}-b_{j}-2$. So the claim is true and  $a_{i}\le a_{j}-b_{j}-2<a_{j}$. Then $w_{i}<w_{j}$ and $J_{v_{i}}$ and $J_{v_{j}}$ are disjoint.

If $v_{j}<v_{i}\le v_{j}+d$. Consider the diagonal $(v_{i}, v_{i}+d+(d+1)(a_{j}-1))$. It is clear that $ \#\{ v\in V \mid v_{i}< v\le v_{i}+d+(d+1)(a_{j}-1)\} \le a_{j}-1$. Then $a_{i}\le a_{j}-1<a_{j}$. So $w_{i}<w_{j}$. Moreover, $Y_{v_{i}}$ and $Y_{v_{j}}$ are disjoint. 

Thus $ \{J_{v_{1}},\dots, J_{v_{n}} \}$ is a maximal $d$-Brauer relation. 
\end{proof}

Now we can construct a map $\Theta:{\bf V}\lra \B $ by sending $V\in {\bf V}$ to $\Theta(V):= \{ J_{v}\mid v\in V \}$. By Proposition \ref{Prop:map}, it is well defined.  Next  for  $B\in \B$, we need to determine the preimage of $B$.

\begin{Lem}\label{Lem:map4}
Let $\Theta$ be defined as above. Then $\Theta$ is surjective. More precisely, for any $B\in \B$, we have $\# \{ V\in {\bf V} \mid \Theta(V)=B\} =n+1$.
\end{Lem}

\begin{proof}
Let $B=\{ X_{1}, \dots, X_{n}\}$ be a maximal $d$-Brauer relation.  Assume $X_{t}$ has the form $(x_{t}, y_{t})$ for $1\le t\le n$. Given any $x_{t}$, we construct a set  $V_{x_{t}}\in \bf V$  as follows. 
\begin{enumerate}[\rm (1)]
  \item For any $1\le s\le n$, one of $x_{s}$ and $y_{s}$ belongs to $V_{x_{t}}$;
 \item $i_{s}\in V_{i_{t}}$ if and only if $i_{t}\le i_{s}<j_{s}$ by clockwise ordering.
\end{enumerate}
We construct $V_{y_{t}}$ in a similar way. It is easy to show  $\Theta(V_{x_{t}})=\Theta(V_{y_{t}})=B$. Then $\Theta$ is surjective. 

We claim that $\# \{ V_{x_{t}}, V_{y_{t}} \mid 1\le t \le n \} =n+1$. We show this by induction. If $n=1$, it is clear. Assume the claim holds for $n\le m-1$. For the case $n=m$. Let $\Pi_{1}$ and $\Pi_{2}$ be the two connect components of $\Pi\backslash X_{i}$. Assume $y_{i}=x_{i}+d+(d+1)j$, where $0\le j\le n-1$.
 By Lemma \ref{Lem:l}, $B\cap \Pi_{l}$ is a maximal $d$-Brauer relation on $\Pi_{l}$, $1\le l \le2$ and moreover $\{\# B\cap \Pi_{1}, \# B\cap \Pi_{2}\}= \{ m-j-1, j\}$. Then by induction, 
  $\# \{ V_{x_{t}}, V_{y_{t}} \mid 1\le t \le n \} =(m-j-1+1)+(j+1)= m+1$. So the claim is true. 

For $V\in\Theta^{-1}(B)$, by our construction of $\{ J_{v}\mid v\in V\}$, one can show that $V$ is given by some $V_{x_{t}}$ or $V_{y_{t}}$. Thus by the claim above  $\# \{ V\in {\bf V} \mid \Theta(V)=B\} =n+1$.
\end{proof}

Theorem \ref{Thm:n} is  deduced  by the Lemma \ref{Lem:map4} directly.

\begin{proof}[Proof of Theorem \ref{Thm:n}]
By Lemma \ref{Lem:map4}, we know $\#\B=\frac{1}{n+1}\# {\bf V}=\frac{1}{n+1}\binom{(d+1)n+d-1}{n}$.
\end{proof}

\subsection{Brauer tree dg algebras}
\label{Brauertreedga}
We first introduce a graded quiver from given maximal $d$-Brauer relation in the following way.

\begin{Def}\label{Def:Brauerquiver}
Let $B\in \B$. The graded quiver $Q_{B}$   associated to $B$ is defined as  follows.
\begin{enumerate}[\rm (1)]
\item The vertices of $Q_{B}$ are given by the $d$-diagonals in $B$;
\item 
For any $B$-cycle $C$ with anti-clockwise ordering $\{ X_{1},\cdots, X_{s}\}$, we draw arrows  $X_{i}\ra X_{i+1}$ with degree $1-\delta(X_{i},X_{i+1})$, where $1\le i\le s$ and $X_{s+1}=X_{1}$.  
\end{enumerate}
We say a cycle in $Q_{B}$ is \emph{minimal}, if it is given by some $B$-cycle.
\end{Def}

In Example \ref{Ex:d=1&n=3}, we give the maximal $d$-Brauer relations for the case $d=2$ and $n=3$. Now we draw the $d$-Brauer quivers associate to them. 

\begin{Ex}
Let $d=2$ and $n=3$. The graded quivers associate to the maximal $d$-Brauer relations are as follows.
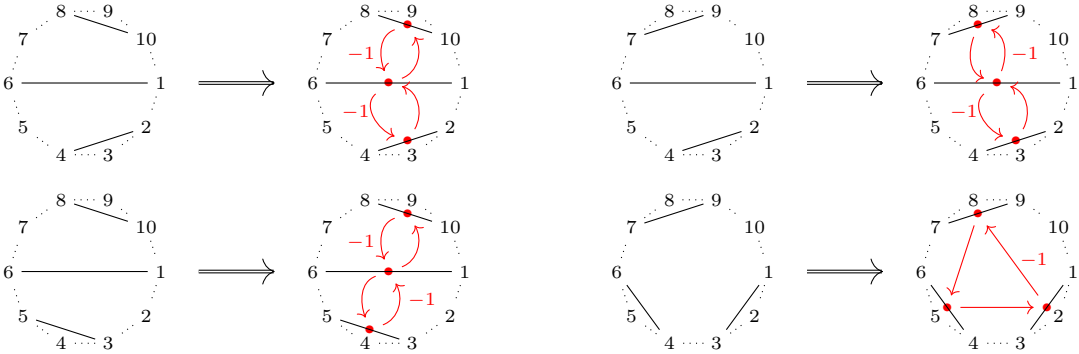
\begin{center}
\newdimen \R \R=1cm
\scriptsize{
\begin{tikzpicture}
\draw[xshift=8\R, yshift=-2.5\R] node (1) at (0:\R){1}
       node (2) at (-360/10:\R){2} 
       node (3) at (-360*2/10:\R){3}
       node (4) at (-360*3/10:\R){4}
       node (5) at (-360*4/10:\R){5}
       node (6) at (-360*5/10:\R){6}
       node (7) at (-360*6/10:\R){7}
       node (8) at (-360*7/10:\R){8}
       node (9) at (-360*8/10:\R){9}
       node (10) at (-360*9/10:\R){10}
       [dotted] (1)--(2)--(3)--(4)--(5)--(6)--(7)--(8)--(9)--(10)--(1);
   \draw (1)--(3) (4)--(6) (7)--(9);
\draw[->, double, yshift=-2.5\R] (9.5,0) -- (10.5,0);

\draw[xshift=12\R,yshift=-2.5\R] node (1) at (0:\R){1}
       node (2) at (-360/10:\R){2} 
       node (3) at (-360*2/10:\R){3}
       node (4) at (-360*3/10:\R){4}
       node (5) at (-360*4/10:\R){5}
       node (6) at (-360*5/10:\R){6}
       node (7) at (-360*6/10:\R){7}
       node (8) at (-360*7/10:\R){8}
       node (9) at (-360*8/10:\R){9}
       node (10) at (-360*9/10:\R){10}
       node[red] (79) at ($(7)!0.5!(9)$){$\bullet$}
       node[red] (46) at ($(4)!0.5!(6)$){$\bullet$}
       node[red] (13) at ($(1)!0.5!(3)$){$\bullet$}
       node[red,right]  at ($(13)!0.5!(79)$){$-1$}
       [dotted] (1)--(2)--(3)--(4)--(5)--(6)--(7)--(8)--(9)--(10)--(1);
   \draw (1)--(3) (4)--(6) (7)--(9);
   \draw[red,->] (46)--(13); \draw[red,->] (13)--(79); \draw[red,->] (79)--(46);

\draw[yshift=0\R] node (1) at (0:\R){1}
       node (2) at (-360/10:\R){2} 
       node (3) at (-360*2/10:\R){3}
       node (4) at (-360*3/10:\R){4}
       node (5) at (-360*4/10:\R){5}
       node (6) at (-360*5/10:\R){6}
       node (7) at (-360*6/10:\R){7}
       node (8) at (-360*7/10:\R){8}
       node (9) at (-360*8/10:\R){9}
       node (10) at (-360*9/10:\R){10}
       [dotted] (1)--(2)--(3)--(4)--(5)--(6)--(7)--(8)--(9)--(10)--(1);
\draw[->, thin, double ,yshift=0\R] (1.5,0) -- (2.5,0);
 \draw (1)--(6) (2)--(4) (8)--(10);

\draw[yshift=0\R,xshift=4\R] node (1) at (0:\R){1}
       node (2) at (-360/10:\R){2} 
       node (3) at (-360*2/10:\R){3}
       node (4) at (-360*3/10:\R){4}
       node (5) at (-360*4/10:\R){5}
       node (6) at (-360*5/10:\R){6}
       node (7) at (-360*6/10:\R){7}
       node (8) at (-360*7/10:\R){8}
       node (9) at (-360*8/10:\R){9}
       node (10) at (-360*9/10:\R){10}
              node[red] (24) at ($(2)!0.5!(4)$){$\bullet$}
       node[red] (16) at ($(1)!0.5!(6)$){$\bullet$}
       node[red] (810) at ($(8)!0.5!(10)$){$\bullet$}
       node[red,left=6]  at ($(16)!0.5!(810)$){$-1$}
       node[red,left=8]  at ($(16)!0.5!(24)$){$-1$}
       [dotted] (1)--(2)--(3)--(4)--(5)--(6)--(7)--(8)--(9)--(10)--(1);
  \draw (1)--(6) (2)--(4) (8)--(10);
 \draw[red,->] (16) to [out=220, in=160] (24); \draw[red,->] (24)to [out=60,in=-15] (16); \draw[red,->] (810)to [out=200,in=110](16); \draw[red,->](16) to [out=20,in=-60] (810);

 \draw[yshift=-2.5\R,xshift=4\R] node (1) at (0:\R){1}
       node (2) at (-360/10:\R){2} 
       node (3) at (-360*2/10:\R){3}
       node (4) at (-360*3/10:\R){4}
       node (5) at (-360*4/10:\R){5}
       node (6) at (-360*5/10:\R){6}
       node (7) at (-360*6/10:\R){7}
       node (8) at (-360*7/10:\R){8}
       node (9) at (-360*8/10:\R){9}
       node (10) at (-360*9/10:\R){10}
              node[red] (35) at ($(3)!0.5!(5)$){$\bullet$}
       node[red] (16) at ($(1)!0.5!(6)$){$\bullet$}
       node[red] (810) at ($(8)!0.5!(10)$){$\bullet$}
       node[red,left=6]  at ($(16)!0.5!(810)$){$-1$}
       node[red,right=8]  at ($(16)!0.5!(35)$){$-1$}
       [dotted] (1)--(2)--(3)--(4)--(5)--(6)--(7)--(8)--(9)--(10)--(1);
  \draw (1)--(6) (3)--(5) (8)--(10);
 \draw[red,->] (16) to [out=200, in=110] (35); \draw[red,->] (35)to [out=20,in=-60] (16); \draw[red,->] (810)to [out=200,in=110](16); \draw[red,->](16) to [out=20,in=-60] (810);
 
 \draw[yshift=-2.5\R] node (1) at (0:\R){1}
       node (2) at (-360/10:\R){2} 
       node (3) at (-360*2/10:\R){3}
       node (4) at (-360*3/10:\R){4}
       node (5) at (-360*4/10:\R){5}
       node (6) at (-360*5/10:\R){6}
       node (7) at (-360*6/10:\R){7}
       node (8) at (-360*7/10:\R){8}
       node (9) at (-360*8/10:\R){9}
       node (10) at (-360*9/10:\R){10}
       [dotted] (1)--(2)--(3)--(4)--(5)--(6)--(7)--(8)--(9)--(10)--(1);
\draw[->, double,yshift=-2.5\R] (1.5,0) -- (2.5,0);
 \draw (1)--(6) (3)--(5) (8)--(10);

\draw[xshift=8\R] node (1) at (0:\R){1}
       node (2) at (-360/10:\R){2} 
       node (3) at (-360*2/10:\R){3}
       node (4) at (-360*3/10:\R){4}
       node (5) at (-360*4/10:\R){5}
       node (6) at (-360*5/10:\R){6}
       node (7) at (-360*6/10:\R){7}
       node (8) at (-360*7/10:\R){8}
       node (9) at (-360*8/10:\R){9}
       node (10) at (-360*9/10:\R){10}
       [dotted] (1)--(2)--(3)--(4)--(5)--(6)--(7)--(8)--(9)--(10)--(1);
\draw[->, double ,xshift=8\R] (1.5,0) -- (2.5,0);
 \draw (1)--(6) (2)--(4) (7)--(9);

\draw[yshift=0\R,xshift=12\R] node (1) at (0:\R){1}
       node (2) at (-360/10:\R){2} 
       node (3) at (-360*2/10:\R){3}
       node (4) at (-360*3/10:\R){4}
       node (5) at (-360*4/10:\R){5}
       node (6) at (-360*5/10:\R){6}
       node (7) at (-360*6/10:\R){7}
       node (8) at (-360*7/10:\R){8}
       node (9) at (-360*8/10:\R){9}
       node (10) at (-360*9/10:\R){10}
              node[red] (24) at ($(2)!0.5!(4)$){$\bullet$}
       node[red] (16) at ($(1)!0.5!(6)$){$\bullet$}
       node[red] (79) at ($(7)!0.5!(9)$){$\bullet$}
       node[red,right=6]  at ($(16)!0.5!(79)$){$-1$}
       node[red,left=8]  at ($(16)!0.5!(24)$){$-1$}
       [dotted] (1)--(2)--(3)--(4)--(5)--(6)--(7)--(8)--(9)--(10)--(1);
  \draw (1)--(6) (2)--(4) (7)--(9);
 \draw[red,->] (16) to [out=220, in=160] (24); \draw[red,->] (24)to [out=60,in=-15] (16); 
 \draw[red,->] (79)to [out=250,in=160](16); \draw[red,->](16) to [out=70,in=-20] (79);
  
\end{tikzpicture} }
\end{center}
where the quivers are drawn by red lines and the numbers with red color are  degrees correspond to the arrows near them.
\end{Ex}

We give some basic properties on $Q_{B}$, which are induced by Proposition \ref{Prop:cycleproperty} and Lemma \ref{Lem:cycle} (2).

\begin{Prop} Let $B\in \B$. Then $Q_{B}$ satisfies the following 
\begin{enumerate}[\rm(1)]
 \item Every vertex of $Q$ belongs to one or two minimal cycles;
 \item Any two minimal cycles meet in one vertex at most;
  \item There are no loops in $Q$;

 \item Every  arrow is equipped with  a non-positive degree and the sum of degrees of each minimal cycle is  $-d+1$.
 \end{enumerate}
\end{Prop} 

\begin{Rem}
The above properties (1), (2), (3)  imply that $Q_{B}$  a Brauer quiver in the sense of Gabriel and Riedtmann (see \cite{GR}).
\end{Rem}
 
Now we introduce the following main object in this section.

\begin{Def}
Let $B\in \bf B$. The \emph{Brauer tree dg algebra}  $A_{Q_{B}}$ is defined as $kQ_{B}/I_{B}$ with zero differential and  grading given by that of $Q_{B}$, where the admissible ideal  $I_{B}$ is generated by the following relations.
\begin{enumerate}[\rm (1)]
\item For any minimal cycle 
 \[ X_{1}\xra{\alpha_{1}}X_{2} \xra{}  \cdots \xra{} X_{m-1} \xra{\alpha_{m-1}} X_{m} \xra{\alpha_{m}} X_{1},\]
$\alpha_{i}\alpha_{i+1}\cdots\alpha_{m}\alpha_{1}\cdots\alpha_{i} \in I$ for each $1\le i\le m$;
\item If $X$ is  the common $d$-diagonal of two $B$-cycles 
 \begin{eqnarray*} X &= &X_{1}\xra{\alpha_{1}}X_{2} \xra{}  \cdots \xra{} X_{m-1} \xra{\alpha_{m-1}} X_{m} \xra{\alpha_{m}} X_{1}\\
  X&=& Y_{1}\xra{\beta_{1}}Y_{2} \xra{}  \cdots \xra{} Y_{s-1} \xra{\beta_{m-1}} Y_{s} \xra{\beta_{s}} Y_{1}, \end{eqnarray*}
 then 
$\beta_{s}\alpha_{1}\in I$  and $\alpha_{m}\beta_{1}\in I$ and $\alpha_{1}\alpha_{2}\cdots\alpha_{m}-\beta_{1}\beta_{2}\cdots\beta_{s}\in I$.
\end{enumerate}
\end{Def}

The following proposition  is an easy generalization of well-known result for ungraded case. 
\begin{Prop}
The dg algebra $A_{Q_{B}}$ is $d$-symmetric.
\end{Prop}
 
Now we  are ready to state the following main result, which implies Theorem \ref{Thm:main} for the case $\Delta=A_{n}$. Recall from Definition \ref{Def:newstable} the definition of $(\Z A_{n,d})_{C}$.
\begin{Thm}\label{Thm:section}
Let $B$ be a maximal $d$-Brauer relation on $((d+1)n+d-1)$-gon and let $C$ be the $(-d)$-CY configuration in $\Z A_{n,d}$ corresponding to $B$. Then for the Brauer tree dg algebra $A_{Q_{B}}$, the $AR$ quiver of $\CM A_{Q_{B}}$ is isomorphic to $(\Z A_{n,d})_{C}$.
\end{Thm}

The outline  of our proof is the following.
 Consider the $(-d)$-CY configuration $C_{A}$ given by the simples of $A_{Q_{B}}$. Then the AR quiver of $\CM A_{Q_{B}}$ is isomorphic to $(\Z A_{n,d})_{C_{A}}$. So we only need to show $C=C_{A}$. To show this, let $B_{A}$ be the maximal $d$-Brauer relation   corresponds to  $C_{A}$. 
 \[C\longleftrightarrow B \lra A_{Q_{B}} \xra{\rm simples} C_{A} \longleftrightarrow B_{A} \]
 Then it suffices to  prove $B$ is isomorphic to $B_{A}$ up to rotation. 

We first describe the AR quiver of the stable category $\un{\CM}A_{Q_{B}}$.

\begin{Prop}\label{Prop:ARofA_{Q_{B}}}
The AR quiver of  $\un{\CM}A_{Q_{B}} $ is $\Z A_{n,d}$.
\end{Prop}

To prove this proposition, we need  some observations. Let $Y\in (Q_{B})_{0}$ be a vertex and $a\in \Z$. We construct a new graded quiver $Q_{Y,a}$. It is isomorphic to $Q_{B}$ as ungraded quiver. The degrees of arrows ending at $Y$ and starting  at $Y$ are changed as follows.
{\tiny 
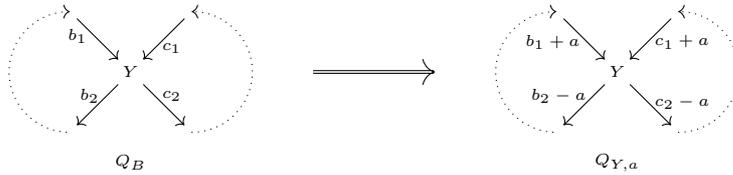
\begin{center}
 \begin{tikzpicture}[scale=0.8]
  \draw
  node (Y) at (0,0) {$Y$}
  node (lu) at (-1,1){} node (ru) at (1,1){}
  node (ld) at (-1,-1){} node (rd) at (1,-1){}
  node[left] at (-0.6,0.6) {$b_{1}$}
  node[left] at (-0.4,-0.4) {$b_{2}$}
  node[right] at (0.4,0.4) {$c_{1}$}
  node[right] at (0.4,-0.4) {$c_{2}$}
  node at (0,-1.5) {$Q_{B}$}
  [->] (lu)edge(Y) (ru)edge(Y) (Y)edge(ld) (Y)edge(rd);
   \draw[dotted, <-] (-1,1) arc(90:270:1);
   \draw[dotted, ->](1,-1) arc(270:450:1);
   \draw[->, double] (3,0) -- (5,0);
      \draw[xshift=8cm]
  node (Y) at (0,0) {$Y$}
  node (lu) at (-1,1){} node (ru) at (1,1){}
  node (ld) at (-1,-1){} node (rd) at (1,-1){}
  node[left] at (-0.5,0.5) {$b_{1}+a$}
  node[left] at (-0.4,-0.4) {$b_{2}-a$}
  node[right] at (0.5,0.5) {$c_{1}+a$}
  node[right] at (0.5,-0.5) {$c_{2}-a$}
  node at (0,-1.5) {$Q_{Y,a}$}
  [->] (lu)edge(Y) (ru)edge(Y) (Y)edge(ld) (Y)edge(rd);
   \draw[xshift=8cm, dotted, <-] (-1,1) arc(90:270:1);
   \draw[xshift=8cm, dotted, ->](1,-1) arc(270:450:1);
 \end{tikzpicture}
\end{center}}
And other degrees of arrows are the same as in $Q_{B}$. Let $T_{Y,a}:=P_{Y}[a]\bop (\bop_{Y'\in B, Y'\not=Y} P_{Y'}) $ be a dg $A_{Q_{B}}$-module, where $P_{Y}$ is the indecomposable projective module corresponds to the vertex $Y$.   Consider the Brauer tree dg algebra $A_{Q_{Y,a}}$. Then  one can show  that $A_{Q_{Y,a}}$ is isomorphic to the endmorphism dg algebra $\shEnd(T_{Y,a})$.
Immediately, we have 
 \begin{Lem}\label{degreechanging}
  The functor $\RshHom(T_{Y,a}, ?)$ induces a triangle equivalence $\D^{\bb}(A_{Q_{B}})/\per A_{Q_{B}}\ra \D^{\bb}(A_{Q_{Y,a}})/\per A_{Q_{Y,a}}$.
\end{Lem}
\begin{proof}
It is clear that $T_{Y,a}$ is a compact generator of $\per A_{Q_{B}}$. Then $\RshHom(T_{Y,a}, ?): \per A_{Q_{B}}\ra \per A_{Q_{Y,a}}$ is an equivalence, which induces a triangle equivalence $\D^{\bb}(A_{Q_{B}})\ra \D^{\bb}(A_{Q_{Y,a}})$. Thus the assertion is true.
\end{proof}

Now we prove Proposition \ref{Prop:ARofA_{Q_{B}}} by adjusting degrees of $Q_{B}$ to some special case. 
\begin{proof}[Proof of Proposition \ref{Prop:ARofA_{Q_{B}}}]
Let $B\in\B$. We say $Q_{B}$ is \emph{admissible} if each minimal cycle in $Q_{B}$ has an arrow with degree $-d+1$ and other arrows with degree $0$. We consider the following two cases.

(1) If $Q_{B}$ is admissible. 
Let $D$  be the set of arrows in $Q_{B}$  with degree $-d+1$. It is an admissible cutting set in the sense of \cite{FP, Schroll}. 
 Therefore  $A_{Q_{B}}$  is isomorphic to the trivial extension $\Lambda\op D\Lambda[d-1]$ by \cite[Theorem 1.3]{Schroll}, where $\Lambda$ is the factor algebra $A_{Q_{B}}/(D)$. By Corollary \ref{Cor:cluster}, $\un{\CM}(A_{Q_{B}})$ is triangle equivalent to $\D^{\rm b}({\mod}\Lambda)/\nu[d]$.   By \cite[Theorem 6.7]{Happel}, $\Lambda$ is an iterated titled algebra of type $A_{n}$. So the AR-quiver of  $\un{\CM}(A_{Q_{B}})$ is given by $\Z A_{n,d}$.

(2) For general $Q_{B}$. We claim  there exists $B'\in \B$, such  that $Q_{B'}$ is  admissible and there is a triangle equivalence $\un{\CM}(A_{Q_{B}})\xra{\sim} \un{\CM}(A_{Q_{B'}})$.
 In fact, we can start from any minimal cycle. Under a suitable ordering, we may change the degrees of $Q_{B}$ to obtain an admissible quiver $Q_{B'}$ step by step by our discussion above. Then by Theorem \ref{Thm:properties} (3) and by Lemma \ref{degreechanging}, $\un{\CM} A_{Q_{B}}= \D^{\bb}(A_{Q_{B}})/\per A_{Q_{B}}$ is triangle equivalence to $\un{\CM} A_{Q_{B'}}=\D^{\bb}(A_{Q_{B'}})/\per A_{Q_{B'}}$.  Then by (1), the  AR quiver of $\un{\CM}(A_{Q_{B}})$ is $\Z A_{n,d}$.
 \end{proof}
 
Let $B=\{Y_{1}, \cdots, Y_{n}\} $ be a  maximal $d$-Brauer relation  on  $((d+1)n+d-1)$-gon. Recall that the vertices of $Q_{B}$ are given by $\{ Y_{1}, \cdots, Y_{n}\}$.
By Theorem \ref{Thm:configuration}, 
the set $C_{A}:=\{ S_{1}, \cdots, S_{n}\}$ of simple dg $A_{Q_{B}}$-modules is a $(-d)$-CY configuration, where $S_{i}$ is the simple module corresponds to vertex $Y_{i}$. And by Proposition \ref{Prop:ARofA_{Q_{B}}}, the AR quiver of $\un{\CM}A_{Q_{B}}$ is $\Z A_{n,d}$. Thus we can also regard $C_{A}$ as the subset of $\Z A_{n,d}$. Let $B_{A}$  be the maximal $d$-Brauer relation corresponds to $C_{A}$. By abuse of notation, the $d$-diagonals in $B_{A}$ are also denoted by $\{S_{1}, \cdots,S_{n}\}$.

Let $\{ Y_{j_{1}}, Y_{j_{2}}, \cdots, Y_{j_{s}} \}$ be a $B$-cycle with anti-clockwise ordering. Then it gives a minimal cycle in $Q_{B}$.
\[ Y_{j_{1}} \xra{\alpha_{1}}  Y_{j_{2}} \xra{\alpha_{2}} \cdots \xra{\alpha_{s-1}} Y_{j_{s}} \xra{\alpha_{s}} Y_{j_{1}}\]  
 where $\deg\alpha_{i}=1-\delta(Y_{j_{i}}, Y_{j_{i+1}})$ by Definition \ref{Def:Brauerquiver}.
  The following proposition gives us some information  which determines $B$ uniquely.

\begin{Prop}\label{Prop:key}
Assume $\alpha_{i}: Y_{j_{i}}\ra Y_{j_{i+1}}$ is an arrow in $Q_{B}$. Then
 $\delta(S_{j_{i}}, S_{j_{i+1}})=\delta(Y_{j_{i}}, Y_{j_{i+1}})$, where we regard $S_{j_{i}}$ and $Y_{j_{i}}$ as $d$-diagonals in $B_{A}$ and $B$ respectively. 
\end{Prop}
\begin{proof}
By Remark \ref{delta},  $\delta (S_{j_{i}}, S_{j_{i+1}})= {\rm min}\{ t>0 \mid \bar{h}_{S_{j_{i}}}(S_{j_{i+1}}[t])\not=0 \}$ and by  Proposition \ref{Prop:vs}, we have
$\bar{h}_{S_{j_{i}}}(S_{j_{i+1}}[t])=\Hom_{\un{\CM}A_{Q_{B}}}(S_{j_{i}}, S_{j_{i+1}}[t])$. Thus
\begin{eqnarray*}\delta (S_{j_{i}}, S_{j_{i+1}}) &=&
\min \{ t> 0 \mid \Hom_{\un{\CM}A_{Q_{B}}}(S_{j_{i}},S_{j_{i+1}}[t])\not=0 \} \\
 &=&   \min \{ t> 0 \mid \Hom_{\D^{\bb}(A_{Q_{B}})}(S_{j_{i}},S_{j_{i+1}}[t])\not=0 \} \end{eqnarray*}
 where the second equality holds by the fact that  $\Hom_{\D^{\bb}(A_{Q_{B}})}(S_{j_{i}}, A)=\h^{-d+1}(S_{j_{i}})=0$.
 Let $l=-\deg \alpha_{i}$. By our construction of $Q_{B}$, it is clear that  every path from $Y_{j_{i}}$ to $Y_{j_{i+1}}$ has degree no more than $-l$. Then
 $\Hom_{\D^{\bb}(A_{Q_{B}})}(S_{j_{i}},S_{j_{i+1}}[t])=0$ for any $0\le t\le l$   and
 $\Hom_{\D^{\bb}(A_{Q_{B}})}(S_{j_{i}},S_{j_{i+1}}[l+1])\not=0$  by Proposition \ref{Prop:key2}. Thus $\delta(S_{j_{i}},S_{j_{i+1}})=l+1=1-\deg \alpha_{i}=\delta(Y_{j_{i}}, Y_{j_{i+1}})$.
  \end{proof}

Now we are ready to prove Theorem \ref{Thm:section}.
\begin{proof}[Proof of Theorem \ref{Thm:section}]
Consider the map $\phi: B \ra B_{A}$ sending $Y_{j}$ to $S_{j}$. It is clearly bijective and 
   for any $B$ cycle $C$ with anti-clockwise ordering $\{ Y_{j_{1}}, Y_{j_{2}}, \cdots, Y_{j_{m}}\}$, we have  $\delta(Y_{j_{i}}, Y_{j_{i+1}})=\delta(S_{j_{i}}, S_{j_{i+1}})$ by  Proposition \ref{Prop:key}. Then  by Proposition \ref{Prop:B=B'}, $B$ is isomorphic to $B_{A}$ up to rotation. Then  the $AR$ quiver of $\CM A_{Q_{B}}$ is isomorphic to $(\Z A_{n,d})_{C}$.
 \end{proof}

\appendix

\section{A  new proof of Theorem \ref{Thm:bijection}}\label{appendixA}

In this part, we give a new proof of Theorem \ref{Thm:bijection} by using the results developed in Section \ref{maximalBrauer}. We first point out the following property.

\begin{Prop}\label{Prop:n}
  For any $B\in \B$, we have $\# B=n$.
\end{Prop}

\begin{proof}
Let $B\in {\bf B}$.
We apply the induction on $n$.

If $n=1$, then $\Pi$ is a $2d$-gon and every $d$-diagonal has the form $(i,i+d)$. In this case, any two $d$-diagonals intersect, which implies that $B$ contains only one $d$-diagonal.

Assume our argument is true for $n\le m$, where  $m\ge 1$. Now consider the case $n=m+1$. 
Assume $I\in B$ has the form $(i_{1}, i_{1}+d+(d+1)j)$. 
 Then $\Pi\backslash I$ has two connect components $\Pi_{1}$ and $\Pi_{2}$, where $\Pi_{1}$ is a $((d+1)j+d-1)$-gon and $\Pi_{2}$ is a $((d+1)(n-j-1)+d-1)$-gon. By Lemma \ref{Lem:l}, $ B\cap \Pi_{1}$ (resp. $B\cap \Pi_{2}$) is a maximal $d$-Brauer relation of $\Pi_{1}$ (resp. $\Pi_{2}$). By induction, $\# (B\cap\Pi_{1})=j$ and $\# (B\cap\Pi_{2})=n-j-1$.
Then $\# B=\# (B\cap\Pi_{1})+ \# (B\cap\Pi_{2}) +1 =n$.
 Therefore the statement holds for any $n\ge 1$.
\end{proof}

The following lemma  is  immediately from  our labelling on $\Z A_{n}$.
\begin{Lem}\label{Lem:homspace}
Let $X, Y\in \Z (A_{n})_{0}$, where $X=(x, x+d+(d+1)m)$, $0\le m \le n-1$. Then  ${h}_{X}(Y)\not=0$ if and only if $Y=(x+(d+1)i,x+d+(d+1)j)$, where $0\le i \le m \le j \le n-1$. 
 \end{Lem}

The following  lemma gives us a way to read $\bar{h}_{X}(Y)$ from the relative position  of $X$ and $Y$ in $\Pi$.

\begin{Lem} \label{Lem:A3}
Let $X, Y\in \Z A_{n,d}$. We also regard them as $d$-diagonals in $\Pi$. Then
 \begin{enumerate}[\rm(1)]
  \item If $X$ and $Y$ are disjoint, then $\bar{h}_{X}(Y)=0$;
  \item If $X$ and $Y$ are joint, then  $\bar{h}_{X}(Y)\not=0$ if and only if $X$ and $Y$ are connected by $d$-diagonals as follows
  \newdimen\R
\R=0.8cm
{\tiny
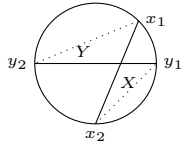
\begin{center}
\begin{tikzpicture}
\draw[xshift=0\R] circle (\R);
 \draw[xshift=0\R] (45:\R)--(270:\R);
 \draw[xshift=0\R] (0:\R)--(180:\R)
 node[right] at (45:\R){$x_{1}$} 
 node[below](x22) at (270:\R){$x_{2}$}
  node[right] at (0:\R){$y_{1}$} 
 node[left] at (180:\R){$y_{2}$} 
     node[right]  at ($(45:\R)!0.6!(270:\R)$){$X$}
         node[above]  at ($(0:\R)!0.6!(180:\R)$){$Y$};  \draw[xshift=0\R, dotted](270:\R)--(0:\R)  (45:\R)--(180:\R);             
    \end{tikzpicture}
\end{center} }
 that is, if and only if $(y_{1},x_{2})$ {\rm(}or equivalently, $(y_{2},x_{1})${\rm)} is a $d$-diagonal. \end{enumerate}

\end{Lem}

Then by the description above, we have the following result.
\begin{Prop}\label{Prop:disjoint}
Let $X, Y\in \Z A_{n,d}$. Then the following are equivalent
\begin{enumerate}[\rm (1)]
\item $\bar{h}_{X}(Y[-s])=0$ and $\bar{h}_{Y}(X[-s])=0$ for $0\le s \le d-1$;
\item $X$ and $Y$ are disjoint as $d$-diagonals.
\end{enumerate}
\end{Prop}

\begin{proof} 
From $\rm (1)$ to $\rm (2)$. If $X\cap Y\not=\emptyset$. We may assume $X=(x_{1},x_{2}=x_{1}+d+(d+1)i)$ and $x_{1}\le y_{1}<x_{2} \le y_{2}$, where  $0\le i\le n-1$. 
We consider the following cases.
\begin{enumerate}[\rm $\bullet$]
\item If $x_{2}\le y_{2}\le x_{2}+d-1$ and $y_{1}-x_{1}>y_{2}-x_{2}$, then $\bar{h}_{X}(Y[x_{2}-y_{2}])\not=0$;
 \item  If $x_{2}\le y_{2}\le x_{2}+d-1$ and $y_{1}-x_{1}\le y_{2}-x_{2}$, then $\bar{h}_{X}(Y[x_{1}-y_{1}])\not=0$;
\item If $x_{2}+d<y_{2}$ and $y_{1}=x_{1}+d+(d+1)i'$, $0\le i'\le i$, then $(x_{1},y_{1})$ is a $d$-diagonal. Then by our discussion above, $\bar{h}_{Y}(X)\not=0$;  
\item If $x_{2}+d-1<y_{2}$ and $y_{1}\not=x_{1}+d+(d+1)i'$ for any $0\le i'\le i$. Then there exist $0\le t\le d$, such that $Y[-t]$ has the form $(x_{1}+(d+1)j, y_{2}-t)$ for some $0\le j< i$. In this case, $\bar{h}_{X}(Y[-t])\not=0$.
\end{enumerate}
All the cases above are contradictory to the condition $\rm (1)$. So we know  $X$ and $Y$ are disjoint.

From $\rm (2)$ to $\rm (1)$. Assume $X$ and $Y$ are disjoint as follows
 \newdimen\R
\R=0.8cm
{\tiny
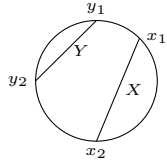
\begin{center}
\begin{tikzpicture}
 \draw[xshift=0\R] circle (\R);
 \draw[xshift=0\R] (45:\R)--(270:\R);
 \draw[xshift=0\R] (90:\R)--(180:\R)
 node[right] at (45:\R){$x_{1}$} 
 node[below](x21) at (270:\R){$x_{2}$}
  node[above] at (90:\R){$y_{1}$} 
 node[left] at (180:\R){$y_{2}$} 
     node[right]  at ($(45:\R)!0.5!(270:\R)$){$X$}
         node[right]  at ($(90:\R)!0.5!(180:\R)$){$Y$};   
\end{tikzpicture} \end{center} }
For $0\le s\le d-1$, $Y[-s]=(y_{1}-s,y_{2}-s)$. If $X\cap Y[-s]=\emptyset$, it is clear $\bar{h}_{X}(Y[-s])=0$. If $X\cap Y[-s]\not=\emptyset$, $i.e. \ y_{2}-s\le x_{2}$.  Then $(y_{2}-s,x_{2})$ can not be a $d$-diagonal (it is possible only when $s>d$). So we still have $\bar{h}_{X}(Y[-s])=0$. Similarly, $\bar{h}_{Y}{X[-s]}=0$ for $0\le s\le d-1$.
\end{proof}

\begin{Rem}\label{Rem:zero}
Let $X=(x_{1},x_{2})$ and $Y=(y_{1},y_{2})$ be two $d$-diagonals. Assume $x_{1}<y_{1}<x_{2}<y_{2}$. Then by the proof of Proposition \ref{Prop:disjoint}, 
if $y_{1}\not=x_{1}+d+(d+1)i'$ for any $0\le i'\le i$, in other words, if $(x_{1},y_{1})$ is not a $d$-diagonal, then there exists $0\le s\le d-1$ such that $\bar{h}_{X}(Y[-s])\not =0$.
\end{Rem}

To prove Theorem \ref{Thm:bijection}, we need another lemma.

\begin{Lem}\label{Lem:existence}
Let $B\in \B$ and let $M$ be a $d$-diagonal. Then there exists $X\in B$ and $0\le i \le d-1$ such that $\bar{h}_{X}(M[-i])\not=0$.
\end{Lem}
\begin{proof}
Since $B$ is maximal, there exists $X\in B$ such that $X\cap M\not=\emptyset$. Up to rotation, there are three types of positional relationships between $M$ and $X$ as follows.
 \newdimen\R
\R=0.8cm
{\tiny
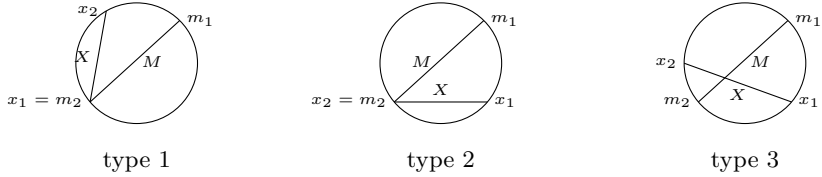
\begin{center}
\begin{tikzpicture}
\draw[xshift=-5\R] circle (\R);
 \draw[xshift=-5\R] (45:\R)--(220:\R);
 \draw[xshift=-5\R] (220:\R)--(120:\R)
 node[right] at (45:\R){$m_{1}$} 
 node[left](x21) at (220:\R){$x_{1}=m_{2}$}
  node[left] at (120:\R){$x_{2}$} 
     node[right]  at ($(45:\R)!0.5!(220:\R)$){$M$}
         node[left]  at ($(220:\R)!0.5!(120:\R)$){$X$}
         node[below=8] at (270:\R){\small type $1$};   
         
 \draw[xshift=0\R] circle (\R);
 \draw[xshift=0\R] (45:\R)--(220:\R);
 \draw[xshift=0\R] (220:\R)--(320:\R)
 node[right] at (45:\R){$m_{1}$} 
 node[left](x21) at (220:\R){$x_{2}=m_{2}$}
  node[right] at (320:\R){$x_{1}$} 
     node[left]  at ($(45:\R)!0.5!(220:\R)$){$M$}
         node[above]  at ($(220:\R)!0.5!(320:\R)$){$X$}
         node[below=8] at (270:\R){\small type $2$};  
         
          \draw[xshift=5\R] circle (\R);
 \draw[xshift=5\R] (45:\R)--(220:\R);
 \draw[xshift=5\R] (180:\R)--(320:\R)
 node[right] at (45:\R){$m_{1}$} 
 node[left](x21) at (220:\R){$m_{2}$}
  node[right] at (320:\R){$x_{1}$} 
 node[left] at (180:\R){$x_{2}$} 
     node[right]  at ($(45:\R)!0.5!(220:\R)$){$M$}
         node[below]  at ($(180:\R)!0.5!(320:\R)$){$X$}
         node[below=8] at (270:\R){\small type $3$};    
\end{tikzpicture} \end{center} }
We show the statement case by case. For type $1$, it is clear $\bar{h}_{X}(M)\not=0$ by Lemma \ref{Lem:A3}. For type $2$, if there is $m_{1}<t<x_{1}$, such that $T=(m_{1},t)$ is a $d$-diagonal in $B$, then $\bar{h}_{T}(M)\not=0$. If there is no such a $T$, we claim that $\exists \ Y\in B$ such that $Y$ and $M$ are of type $3$.

To prove this claim, let us consider the $B$-cycle $B_{X}$ containing $X$ such that $B_{X}$ and $M$ are on the same side of $X$. If the claim is not true, 
then for any $X'\in B_{X}$, $M$ and $X'$ are disjoint or of type $2$ (Notice that by our assumption, type $1$ never happens). Labelling $B_{X}$ anti-clockwise starting from $X$. Let $X_{s+1}=X=X_{1}$ (see figure (a) below). 
We may write $x_{1}=x_{2}+d+(d+1)i'$ and $m_{1}=m_{2}+d+(d+1)i''$, where $0\le i''< i' \le n-1$, then  $x_{1}-m_{1}=(d+1)(i'-i'')$ and the number of vertices between $m_{1}$ and $x_{1}$ is $ (d+1)(i'-i'')-1=d+(d+1)(i'-i''-1)$. 
Let $j$ be the smallest number such that $X_{j}$ and $X_{1}$ are on the different sides of $M$. Then  the sum of number of vertices between $X_{i}$ and $X_{i+1}$ for $1\le i\le j-1$ is at least $d$. Then $\sum_{i=1}^{s}\delta(X_{i},X_{i+1})\ge d+s$. It is contradictory to Proposition \ref{Prop:cycle}, which says $\sum_{i=1}^{s}\delta(X_{i},X_{i+1})=d+s-1$. 
So the claim holds. Then we only need to consider type $3$.
{\tiny
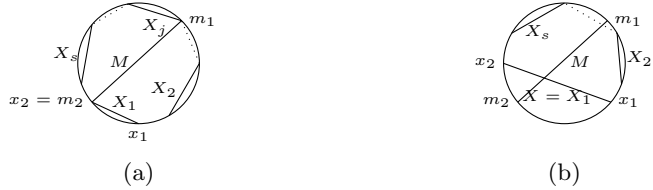
\begin{center}
\begin{tikzpicture}
 \draw[xshift=0\R] circle (\R);
 \draw[xshift=0\R] (45:\R)--(220:\R) (300:\R)--(0:\R)
 (45:\R)--(100:\R) (140:\R)--(200:\R);
 \draw[dotted] (0:\R)--(45:\R) (100:\R)--(140:\R);
 \draw[xshift=0\R] (220:\R)--(270:\R)
     node[right] at (45:\R){$m_{1}$} 
     node[left](x21) at (220:\R){$x_{2}=m_{2}$}
     node[below] at (270:\R){$x_{1}$} 
     node[left]  at ($(45:\R)!0.5!(220:\R)$){$M$}
     node[left]  at ($(300:\R)!0.5!(0:\R)$){$X_{2}$}
     node[left]  at ($(140:\R)!0.5!(200:\R)$){$X_{s}$}
     node[below]  at ($(45:\R)!0.5!(100:\R)$){$X_{j}$}
     node[above]  at ($(220:\R)!0.7!(270:\R)$){$X_{1}$}
     node[below=12] at (270:\R){\small (a)};  
     
               \draw[xshift=7\R] circle (\R);
 \draw[xshift=7\R] (45:\R)--(220:\R);
 \draw[xshift=7\R, dotted] (30:\R)--(90:\R);
 \draw[xshift=7\R] (180:\R)--(320:\R) (340:\R)--(30:\R) (90:\R)--(150:\R)
 node[right] at (45:\R){$m_{1}$} 
 node[left](x21) at (220:\R){$m_{2}$}
  node[right] at (320:\R){$x_{1}$} 
 node[left] at (180:\R){$x_{2}$} 
     node[right]  at ($(45:\R)!0.5!(220:\R)$){$M$}
     node[right]  at ($(340:\R)!0.5!(30:\R)$){$X_{2}$}
     node[below]  at ($(90:\R)!0.5!(150:\R)$){$X_{s}$}
         node[below]  at ($(180:\R)!0.5!(320:\R)$){$X=X_{1}$}
         node[below=12] at (270:\R){\small (b)};    
\end{tikzpicture} \end{center} }
Assume $X$ and $M$ are of type $3$. We may assume there is no $X'\in B$, such that $X'$ and the vertex $m_{1}$ are on the same side of $X$, and $X'$, $M$ are of type $3$ (if such $X'$ exists, replace $X$ by $X'$). Now we show $(x_{2},m_{1})$ is not a $d$-diagonal. If $(x_{2},m_{1})$ is a $d$-diagonal, consider the $B$-cycle $B_{X}$ containing $X$, $B_{X}$ and the vertex $m_{1}$ are on the same side of $X$. Labelling $B_{X}$ anti-clockwise (see figure (b) above). Since $(x_{2},m_{1})$ is a $d
$-diagonal, then $(d+1)| (x_{1}-m_{1})$. Similar to our discussion for type 2, we have $\sum_{i=1}^{s}\delta(X_{i},X_{i+1})\ge d+s$, which is contradictory to Proposition \ref{Prop:cycle}. So we know $(x_{2},m_{1})$ is not a $d$-diagonal. Then by Remark \ref{Rem:zero},  there exists $0\le i\le d-1$ such that $\bar{h}_{X}(M[-i])\not=0$. Therefore the assertion is true.
\end{proof}

 We are ready to  prove Theorem \ref{Thm:bijection} now.
\begin{proof}[The proof of Theorem \ref{Thm:bijection}] 

Given a $(-d)$-CY configuration $ C$ in $\Z A_{n,d}$. 
By Definition \ref{Def:combi}, for any two different objects $X$ and $Y$ in $ C$, we have 
 $\bar{h}_{X}(Y[-s])=0$ and $\bar{h}_{Y}(X[-s])=0$ for $0\le s \le d-1$.
  Then by Proposition \ref{Prop:disjoint}, $X$ and $Y$ are disjoint. So the set $\{ X | X\in C \}$ gives rise to a $d$-Brauer relation $B$. We claim $B$ is maximal. If not, there exists a $d$-diagonal $M$ such that for any $X\in  C$,  $X$ and $M$ are disjoint. Then by Proposition \ref{Prop:disjoint}, $\bar{h}_{X}(M[-s])=0$ for $0\le s \le d-1$. It is contradictory to that $C$ is a $(-d)$-CY configuration (see Definition \ref{Def:combi}). 

On the other hand, given a maximal $d$-Brauer relation $B$. Let $C$ be the set of vertices of $\Z A_{n,d}$ corresponds to the $d$-diagonals in $B$. By Proposition \ref{Prop:disjoint},  for any two different objects $X$ and $Y$ in $C$, we have $\bar{h}_{X}(Y[-j])=0$, where $0\le j\le d-1$.
Let $M$ be any vertex in $\Z A_{n,d}$. Since $B$ is maximal, then by Lemma \ref{Lem:existence}, there exists  $X\in \cal C$ and $0\le i\le d-1$ such that $\bar{h}_{X}(M[-i])\not=0$. So $C$ is a $(-d)$-CY configuration.
 \end{proof}



\end{document}